\documentclass[a4paper, oneside,11 pt]{article}

\usepackage[latin1]{inputenc}
\usepackage[T1]{fontenc}
\usepackage[english]{babel}

\usepackage[all]{xy}
\usepackage{verbatim}
\usepackage{amsmath}
\usepackage{amsthm}
\usepackage{graphicx}
\usepackage{amssymb}
\usepackage{epstopdf}
\usepackage{url}
\usepackage{amsfonts}
\usepackage{todonotes}
\usepackage{fullpage}
\newcommand{\mc}{\mathcal}
\newcommand{\pt}{\partial}

\newcommand{\br}{\mathbb{R}}

\newcommand{\bt}{\mathbb{T}}

\newcommand{\e}{\varepsilon}

\renewcommand{\(}{\left(}
\renewcommand{\)}{\right)}
\renewcommand{\[}{\left[}
\renewcommand{\]}{\right]}

\newcommand{\bb}{\mathbb}

\newtheorem{thm}{Theorem}
\newtheorem{lem}[thm]{Lemma}

\newtheorem{prop}[thm]{Proposition}
\newtheorem{defi}[thm]{Definition}
\newtheorem{remark}[thm]{Remark}

\def\be{\begin{equation}}
\def\ee{\end{equation}}
\def\bea{\begin{eqnarray}}
\def\eea{\end{eqnarray}}

\usepackage{enumerate}

\usepackage{bbm}

\numberwithin{thm}{section}
\numberwithin{equation}{section}

\usepackage{hyperref}
\usepackage{mathtools}
\mathtoolsset{showonlyrefs}

\newcommand*\di{\mathop{}\!\mathrm{d}}

\title{Singular Limits for Plasmas with Thermalised Electrons}

\author{
Megan Griffin-Pickering 
  \thanks{University of Cambridge, DPMMS Centre for Mathematical Sciences, Wilberforce Road, Cambridge CB3 0WB, UK. Email: \textsf{m.griffin-pickering@maths.cam.ac.uk}}
  \and
Mikaela Iacobelli
  \thanks{Durham University, Department of Mathematical Sciences, Lower Mountjoy, Stockton Road, Durham DH1 3LE, UK. Email: \textsf{mikaela.iacobelli@durham.ac.uk}}
}

\begin{document}

\maketitle

\begin{abstract}
This work is concerned with the study of singular limits for the Vlasov-Poisson system in the case of massless electrons (VPME), which is a kinetic system modelling the ions in a plasma.
Our objective is threefold: first, we provide a mean field derivation of the VPME system in dimensions $d=2,3$ from a system of $N$ extended charges. Secondly, we prove a rigorous quasineutral limit for initial data that are perturbations of analytic data, deriving the Kinetic Isothermal Euler (KIE) system from the VPME system in dimensions $d=2,3$. Lastly, we combine these two singular limits in order to show how to obtain the KIE system from an underlying particle system.

\end{abstract}


\section{Introduction}

In this article, we will study a hierarchy of models for plasma. A plasma forms when a neutral gas undergoes a process of dissociation, so that some of the gas particles split into electrons and positively charged ions. These charged particles can interact with each other through the electromagnetic fields they generate. This long-range effect is the dominant form of interaction within the plasma.

As a mathematical idealisation, we can model a plasma as a system of point particles. We have two distinguished species of particle - ions and electrons. The ions have a much higher mass than electrons, which leads to a separation of timescales between the two species. The electrons typically evolve on a much faster timescale than the ions. This makes it reasonable to consider the two species separately.

Suppose first that we would like to model the electrons. In this case, it is usual to take the ions to be fixed in a spatially uniform distribution. This is a reasonable assumption since the ions evolve on a much slower timescale than the electrons. A well-known model for this situation is the \textit{classical Vlasov-Poisson system:}
\begin{equation}
\label{vp}
(VP):= \left\{ \begin{array}{ccc}\pt_t f+v\cdot \nabla_x f+ E\cdot \nabla_v f=0,  \\
E=-\nabla U, \\
\Delta U=1- \int_{\br^d} f\, dv= 1- \rho,\\
f\vert_{t=0}=f_{0}\ge0,\ \  \int_{\bt^d \times \br^d} f_{0}\,dx\,dv=1.
\end{array} \right.
\end{equation}
Here we are using a statistical description of the electrons, through a time dependent (probability) density function $f_t(x,v)$ on the phase space $\bt^d \times \br^d$. We note that VP is an appropriate model in the electrostatic approximation when magnetic effects are negligible.

\smallskip

Here, we are interested in the reverse case: we wish instead to model the ions in the plasma. To do this, we need to account for the behaviour of the electrons. As we have explained, the electrons are much lighter than the ions and so evolve much more quickly. This means that they undergo collisions at a much faster rate, and hence rapidly approach thermodynamic equilibrium. We will therefore assume that the electrons are thermalised. They follow a Maxwell-Boltzmann law so that their spatial density is given by $e^U$, where $U$ is the electrostatic potential in the plasma. The corresponding kinetic equation is the \textit{Vlasov-Poisson system with massless electrons} (VPME):
\begin{equation}
\label{vpme}
(VPME):= \left\{ \begin{array}{ccc}\pt_t f+v\cdot \nabla_x f+ E\cdot \nabla_v f=0,  \\
E=-\nabla U, \\
\Delta U=e^{U}- \int_{\br^d} f\, dv=e^{U}- \rho,\\
f\vert_{t=0}=f_{0}\ge0,\ \  \int_{\bt^d \times \br^d} f_{0}\,dx\,dv=1.
\end{array} \right.
\end{equation}
The nonlinearity in the Poisson equation is the key difference between the classical and VPME systems, and a source of additional mathematical richness.
The VPME system has been studied less widely than the classical system, due to the additional difficulties created by this nonlinear coupling. 
In particular, while the global well-posedness in two and three dimensions is well understood in the classical case (see for example \cite{Ukai-Okabe}, \cite{Pfaffelmoser},\cite{Schaeffer}, \cite{BR}, \cite{Lions-Perthame}), this problem remained completely open for massless electrons until very recently.
Indeed, at least to our knowledge, the only general result on this model in dimension $d=2,3$ 
was the global existence of weak solutions shown by Bouchut in \cite{Bouchut}. In a very recent paper \cite{IGP-WP} we filled this gap by proving uniqueness for VPME in the class of solutions with bounded density, and global existence of solutions with bounded density for a general class of initial data, generalising to this setting all the previous results known for VP. See also the recent paper by Bardos-Golse-Nguyen-Sentis \cite{BGNS18}  for some related results on a variant of VPME.
\smallskip

Kinetic equations such as VP and VPME provide a medium scale or `mesoscopic' description of physical systems. They lie between a microscopic description, in which one tracks the dynamics of individual particles, and a coarser macroscopic description. Our main goal in this work is to study two types of limit that connect the VPME system \eqref{vpme} to related macroscopic and microscopic models: the \textit{quasineutral limit} and the \textit{mean field limit}. The limits we study were previously proved rigorously in particular regimes for the classical system \eqref{vp}. Here, we will prove analogous results for the VPME system \eqref{vpme}. We will first  derive the VPME system in dimensions $d = 2, 3$ from a system of $N$ extended charges, with assumptions on the choice of parameters that are the same as the ones found by Lazarovici for the classical VP system in \cite{Laz}. Then we will show the validity of the quasineutral limit for initial data that are perturbations of analytic data, deriving the Kinetic Isothermal Euler (KIE) system from the VPME system in dimensions $d = 2, 3$. Finally, we will combine these two singular limits in order to recover the KIE system from the underlying particle system, in analogy with the results available for the VP system. To conclude our analysis we will show that the regime of parameters that allows the mean field and the quasineutral limit to ``commute'' is compatible with a ``large'' set of admissible configurations.

\begin{paragraph}{The quasineutral limit}

A plasma has a characteristic scale, known as the Debye length, that describes the scale of electrostatic interaction within the plasma. For instance, it characterises the scale of charge separation between electrons and ions. We introduce a small parameter $\e := \frac{\lambda_D}{\Lambda}$, where $\lambda_D$ is the Debye length and $\Lambda$ is the scale of observation. Under an appropriate scaling, the VPME system becomes
\begin{equation}
\label{vpme-quasi}
(VPME)_\e:= \left\{ \begin{array}{ccc}\pt_t f_\e+v\cdot \nabla_x f_\e+ E_\e\cdot \nabla_v f_\e=0,  \\
E_\e=-\nabla U_\e, \\
\e^2\Delta U_\e=e^{U_\e}- \int_{ \br^d} f_\e\, dv=e^{U_\e}- \rho_\e,\\
f_\e\vert_{t=0}=f_{0,\e}\ge0,\ \  \int_{\bt^d \times \br^d} f_{0,\e}\,dx\,dv=1.
\end{array} \right.
\end{equation}
In real plasmas, the Debye length is typically very small. In this case, the plasma is called \textit{quasineutral}, in reference to the fact that the plasma appears to be neutral overall at the observation scale. In the physics literature, quasineutrality is often included in the very definition of plasma. It is therefore interesting to consider the limit in which $\e$ tends to zero. The formal limit of \eqref{vpme-quasi} is the \textit{kinetic isothermal Euler} system:
\be
\label{KIE}
(KIE):= \left\{ \begin{array}{ccc}\pt_t f+v\cdot \nabla_x f+ E\cdot \nabla_v f=0,  \\
E=-\nabla U, \\
U= \log \rho,\\
f_0\ge0,\ \  \int_{\bt^d \times \br^d} f_0\,dx\,dv=1.
\end{array} \right.
\end{equation}
A related equation, derived in the quasineutral limit from a version of \eqref{vpme-quasi} in which the electron density $e^U$ is approximated by the linearisation $1 + U$, was named \textit{Vlasov-Dirac-Benney} by Bardos and studied in \cite{Bardos-Besse} and \cite{Bardos-Nouri}. In the classical case, where $E$ is replaced by a pressure term which is a Lagrange multiplier corresponding to the constraint $\rho = 1$, Bossy-Fontbona-Jabin-Jabir \cite{BFJJ} showed local-in-time existence of analytic solutions in the one-dimensional case. Global existence of weak solutions is not known for \eqref{KIE}.

The rigorous justification of the quasineutral limit is a non-trivial and subtle problem. The limit has a direct correspondence to a long-time limit for the Vlasov-Poisson system, and is therefore vulnerable to known instability mechanisms inherent to the physical system under consideration. In fact, for the classical system \eqref{vp}, it was shown by Hauray and Han-Kwan in \cite{HKH} that the quasineutral limit is false in general if the initial data are assumed to have only Sobolev regularity. 

Rigorous results on the quasineutral limit go back to the works of Brenier-Grenier \cite{BG} and Grenier \cite{Grenier95} for the classical system \eqref{vp}. A result of particular relevance for our purposes is the work of Grenier \cite{Grenier96}, proving the limit for the classical system assuming uniformly analytic data. The works of Han-Kwan-Iacobelli \cite{IHK2,IHK1} extended this result to data that are very small, but possibly rough, perturbations of the uniformly analytic case, in dimension 1, 2 and 3. 

In the massless electrons case, Han-Kwan-Iacobelli \cite{IHK1} showed a rigorous limit in dimension one, again for rough perturbations of analytic data, while Han-Kwan-Rousset \cite{HKR} consider Penrose-stable data with sufficiently high Sobolev regularity. In this work, we extend the results of \cite{IHK1} to higher dimensions by showing a rigorous quasineutral limit for the VPME system \eqref{vpme-quasi} in dimension 2 and 3, for data that are very small, but possibly rough, perturbations of some uniformly analytic functions.
\end{paragraph}

\begin{paragraph}{The mean field limit}
The second type of limit we will consider is the \textit{mean field limit}. This refers to the general problem of deriving a Vlasov equation, such as the Vlasov-Poisson system, from an underlying microscopic particle system. In a typical formulation of this problem, one considers a system of $N$ point particles evolving under the influence of binary interactions between the particles, described by an interaction force $\nabla \Phi$ derived from a potential $\Phi$, and possibly an external force $\nabla \Psi$ arising from a potential $\Psi$. The dynamics are modelled by a system of ODEs describing the phase space position $(X_i, V_i)_{i=1}^N$ of each of the $N$ particles: 
\be \label{ODE-gen}
\begin{cases}
\dot X_i = V_i \\
\dot V_i = \frac{1}{N} \sum_{j \neq i} \nabla_x \Phi(X_i - X_j) + \nabla \Psi(X_i) .
\end{cases}
\ee
The factor $1/N$ appears due to a choice of scaling, designed to be the appropriate one to obtain a Vlasov equation in the limit.

The connection to a Vlasov equation is formulated via the \textit{empirical measure} $\mu^N$ associated to \eqref{ODE-gen}, defined by:
\be \label{def:mu}
\mu^N(t) : = \frac{1}{N} \sum_{i=1}^N \delta_{(X_i(t), V_i(t))} .
\ee
If $\mu^N$ converges to some measure $f$ as $N$ tends to infinity, then we expect $f$ to be a solution of the associated \textit{Vlasov equation}
\be \label{vlasov}
\partial_t f + v \cdot \nabla_x f + (\nabla \Phi \ast_x \rho_f + \nabla \Psi) \cdot \nabla_v f = 0 .
\ee
By `deriving \eqref{vlasov} from the particle system \eqref{ODE-gen}', we mean showing rigorously that $\mu^N$ converges to $f$ in the sense of measures, where $f$ is a solution of \eqref{vlasov}, assuming convergence of the initial data. 

Both the VP and the VPME systems fit into this general framework. We can see this by introducing the Green function for the Laplacian on the torus. This is a function $G$ satisfying
\be \label{def:G}
\Delta G = \delta_0 - 1.
\ee
We also define the Coulomb kernel $K = \nabla G$. More explicitly, $G$ may be written in the form
\be \label{def:G0}
G(x) = \begin{cases}
\frac{1}{2 \pi} \log{|x|} + G_0(x) & d=2 \\
- \frac{1}{4 \pi |x|} + G_0(x) & d=3 ,
\end{cases}
\ee
for some smooth function $G_0$. For a proof of this representation, see \cite{Titch} or \cite[Lemma 2.1]{IGP}.

In the case of the classical VP system \eqref{vp}, we have the representation
\be
E = K \ast \rho .
\ee
Thus \eqref{vp} is of the form \eqref{vlasov}, where we choose the kernel $\Phi = G$. Of course, in this case, $\nabla \Phi = K$ has a very strong singularity at the origin, of order $|x|^{-(d-1)}$. This singularity is a key source of difficulty for the mathematical study of the Vlasov-Poisson system, particularly for the mean field limit.

Similarly, for the VPME system, we can represent the force in the form
\be
E = K \ast \rho - K \ast e^U .
\ee
We can think of the VPME system as being of the form \eqref{vlasov} by taking $\Phi = G$ and an `external' potential $\Psi = G \ast e^U$. Of course $\Psi$ is not truly an external potential because $U$ depends nonlinearly on $f$.

Early works on the mathematical justification of the mean field limit include, among others, Braun-Hepp \cite{Braun-Hepp} and Neunzert-Wick \cite{Neunzert-Wick}. Dobrushin \cite{Dob} proved a rigorous limit for Lipschitz force fields ($\nabla \Phi, \nabla \Psi \in W^{1,\infty}$). The approach relies on the observation that the empirical measure $\mu^N$ is a weak solution of the Vlasov equation \eqref{vlasov}, if the forces are sufficiently regular. Dobrushin proved the mean field limit by showing a stability result for \eqref{vlasov} in the class of measure solutions, using an estimate in a Wasserstein distance. The paper \cite{Dob} is a direct ancestor of much of the modern work on the subject. See the reviews of Golse \cite{Golse} and Neunzert \cite{Neunzert} for further background and literature.

In many physical systems, the force is described by an inverse power law $|\nabla \Phi| \sim |x|^{- \alpha}$. In this case, the singularity at the origin prevents the application of Dobrushin's results.
Notice that the VP and VPME systems correspond to the strongly singular case $\alpha = d-1$, and no stability results are known in these cases. Hence, to deal with the mean field limit for VMPE, we will consider a suitable regularisation of the microscopic particle system.

There have been several works aimed at deriving Vlasov equations with singular forces from regularised particle systems.
For instance, Hauray-Jabin \cite{Hauray-Jabin} considered a truncation method in which the force is cut off below a certain distance from the origin $r_N$, dependent on the number of particles $N$. They showed that the mean field limit holds for a large set of initial configurations, for inverse power law forces with $\alpha < d-1$ (in particular not the Vlasov-Poisson case), from a particle system with force truncated at $r_N$, provided that $r_N$ converges to zero sufficiently slowly as $N$ tends to infinity. They also proved a true mean field limit, without truncation, for the case of `weakly singular' forces in which $\alpha < 1$. More recently, Lazarovici-Pickl \cite{Laz-Pickl} achieved a similar result for the classical Vlasov-Poisson case $\alpha = d-1$ with this type of truncation, with $r_N \sim N^{- \frac{1}{d} + \eta}$ for any $\eta > 0$. They use a law of large numbers approach to compare the mean field force from the particle system to the limiting force. Their results show that there exists a large set of initial configurations for which the mean field limit holds, but it is not possible to identify them from the initial configurations alone, since the argument relies on a law of large numbers throughout the evolution.
In a different direction, Lazarovici \cite{Laz} considered the alternative method of regularisation by convolution. In this approach, the point particles are replaced by delocalised packets of charge, with some smooth, compactly supported shape $\chi$, fixed throughout the evolution. For the classical Vlasov-Poisson case, this results in the particle system
\be \label{ODE-VP-reg}
\begin{cases}
\dot X_i = V_i \\
\dot V_i = \frac{1}{N} \sum_{j \neq i} \[\chi \ast_x \nabla_x G \ast_x \chi \](X_i - X_j) .
\end{cases}
\ee
The shape is then allowed to depend on a regularisation parameter $r$ by taking $\chi_r(x) : = r^{-d} \chi \left ( \frac{x}{r} \right ) $. Lazarovici showed that a mean field limit holds with high probability, provided that $r_N \geq C N^{-\frac{1}{d(d+2)} + \eta}$ for some $\eta > 0$. The admissible configurations are identified by a condition on the initial configuration alone. The appearance of the double regularisation $\chi_r \ast_x \nabla_x G \ast_x \chi_r$ is very important in this analysis. This type of regularisation was previously considered by Horst \cite{Horst} in the Vlasov-Maxwell case and later used by Rein \cite{Rein-book}; a version also appears in Bouchut \cite{Bouchut}. It has the advantage that the microscopic dynamics correspond to a Hamiltonian system, for which the corresponding energy converges as $r$ tends to zero to the energy of the true Vlasov-Poisson system. In this paper, we will prove a regularised mean field limit of this type for VPME that is, to the best of our knowledge, the first derivation of this system from the particle dynamics. As noted in Remark \ref{rmk:assumptions} we recover the same assumptions on $r$ as the ones obtained by Lazarovici for the VP system in \cite{Laz}. Furthermore, as already mentioned, we will prove a series of quantitative estimates that allow us to relate the mean field limit to the quasineutral one.

\end{paragraph}

\section{Structure of the paper and main results}
In section \ref{prelim} we start by giving the basic definitions and preliminary results needed to state our main results.
In sections \ref{sec:statement-quasi}, \ref{sec:statement-MF}, \ref{sec:statement-MFQN} we state, respectively, Theorem \ref{thm:quasi} about the convergence in the quasineutral limit, Theorem \ref{thm:MFL} concerning the mean field limit, and Theorem \ref{thm:MFQN} where we combine these two regimes. Moreover in section \ref{sec:statement-typicality} we state Theorems \ref{thm:typ-MF} and \ref{thm:typ-MFQN} where we prove typicality results that complete our analysis.

These results are based on some estimates on the electric field that will be performed in Section \ref{sec:electric}. In Section \ref{sec:stability} we establish a quantitative strong-strong stability estimate \`a la Loeper between solutions of the VPME system. This will be a crucial step towards the proof of the quasineutral limit.
Since the strong-strong stability estimates rely on the $L^\infty$ bounds on the densities associated to the solutions of the VPME system, in Section \ref{sec:growth} we will study how the support of solutions of the VPME system grow in time in order to have a control on such an $L^\infty$ norm. As explained later, we will need to develop two different proofs for the two and three dimensional case, respectively in Sections \ref{sec:growth2d} and \ref{sec:growth3d}.
All the analysis performed in these sections will be used in Sections \ref{sec:QN}, \ref{sec:MFL}, and \ref{sec:proof-MFLQN} to prove the main Theorems \ref{thm:quasi}, \ref{thm:MFL}, and \ref{thm:MFQN}, respectively.
Finally, in Section \ref{sec:proof-typicality} we prove Theorems \ref{thm:typ-MF} and \ref{thm:typ-MFQN}.

\subsection{Preliminary definitions} \label{prelim}

We begin by introducing some important quantities and technical tools needed to state our results.

\begin{paragraph}{Energy:}
We introduce the energy associated to the Vlasov-Poisson system for massless electrons \eqref{vpme}. In quasineutral scaling, it is given by the functional
\be \label{def:Ee}
\mc{E}_\e[f_\e] := \frac{1}{2}\int_{\bt^d \times \br^d} |v|^2 f_\e \di x \di v + \frac{\e^2}{2} \int_{\bt^d \times \br^d} |\nabla U_\e|^2 \di x +  \int U_\e e^{U_\e} \di x .
\ee
Here, as before, $U_\e$ denotes the electrostatic potential induced by the ion distribution $f_\e$ in the case of massless electrons. That is, $U_\e$ satisfies
\be
\e^2 \Delta U_\e = e^{U_\e} - \rho[f_\e].
\ee
Throughout the paper, we shall use both $\rho[f]$ and $\rho_f$ to denote the density associated to $f$, namely
$$
\rho[f](t,x)=\rho_f(t,x):=\int_{\br^d}f(t,x,v)\di v.
$$
Control of the energy implies a bound on the $L^{\frac{d+2}{d}}(\bt^d)$ norm of the mass density.
\begin{lem} \label{lem:rho-Lp} \begin{enumerate}[(i)]
\item Let $f(x,v)$ satisfy, for some constant $C_0$,
\be
\mc{E}_\e[f] \leq C_0 .
\ee
Then there exists a constant $C$, depending on $C_0$ but independent of $\e$, such that
\be
\int_{\bt^d \times \br^d} |v|^2 f(x,v) \di x \di v \leq C .
\ee
\item Assume that
\be
\lVert f \rVert_{L^{\infty}(\bt^d \times \br^d)} \leq C_0,\qquad
\int_{\bt^d \times \br^d} |v|^2 f(x,v) \di x \di v \leq C_0 .
\ee
Then the mass density
\be \label{def:rho}
\rho(x) : = \int_{\br^d} f(x,v) \di v
\ee
lies in $L^{(d+2)/d}(\bt^d)$ with
\be \label{rho-Lp}
\lVert \rho \rVert_{L^{\frac{d+2}{d}}(\bt^d)} \leq C ,
\ee
for some constant $C$ depending on $C_0$ only.
\item In particular, if
\be
\lVert f \rVert_{L^{\infty}(\bt^d \times \br^d)} \leq C_0,\qquad
\mc{E}_\e[f] \leq C_0 ,
\ee
then there exists a constant $C$ independent of $\e$ such that
\be \label{rho-Lp-summary}
\lVert \rho \rVert_{L^{\frac{d+2}{d}}(\bt^d)} \leq C .
\ee
\end{enumerate}
\end{lem}
\begin{proof}
Observe that $x e^x \geq - 1$. This implies that
\be
\int_{\bt^d \times \br^d} |v|^2 f \di x \di v \leq 2(C_0 + 1) .
\ee
Part (ii) then follows from a standard interpolation argument. See for example \cite[Lemma 5.1]{IGP-WP}.
\end{proof}

\end{paragraph}

\begin{paragraph}{Wasserstein distances:}
Many of our results involve controlling the distance between solutions to some PDEs, for instance the VPME or KIE systems. For the equations considered in this article, it is often useful to look at solutions in the class of measures. A very natural way to quantify the distance between two such solutions is given by the Wasserstein distances, a particular family of metrics on measures. For our purposes it suffices to consider probability measures. To define the Wasserstein distances, we first need to introduce the notion of a \textit{coupling} of two measures. If $(\Omega, \mc{F})$ is a measurable space, and $\mu$ and $\nu$ are probability measures on this space, then a coupling of $\mu$ and $\nu$ is a measure $\pi$ on the product space $\Omega \times \Omega$ from which $\mu$ and $\nu$ may be recovered as the marginals of $\pi$. More precisely, this means that for any $A \in \mc{F}$, we have the two identities
\be
\pi(A \times \Omega)  = \mu(A), \qquad \pi(\Omega \times A)  = \nu(A) .
\ee
We will use the notation $\mathcal P(\Omega)$ to denote the space of probability measures on $\Omega$, and $\Pi(\mu, \nu)$ to denote the set of possible couplings of $\mu$ and $\nu$. A Wasserstein distance between $\mu$ and $\nu$ is constructed by optimising a particular cost functional over the set $\Pi(\mu,\nu)$.

\begin{defi}[Wasserstein distances] \label{def:Wass}
Let $(\Omega,d)$ be a Polish space with metric $d$ and let $\mc{F}$ denote its Borel $\sigma$-algebra. Let $p \in [1, \infty)$. The Wasserstein distance of order $p$, denoted $W_p$, is defined by
\be
W_p^p(\mu, \nu) = \inf_{\pi \in \Pi(\mu,\nu)} \int_{\Omega \times \Omega} d(x,y)^p \di \pi(x,y),
\ee
for all $\mu, \nu \in \mc{P}(\Omega)$ such that the right hand side is finite. In particular this is well-defined for $\mu, \nu \in \mc{P}_p$, where $\mc{P}_p$ denotes the set of probability measures $\gamma$ for which 
\be
\int_{\Omega} d(x, x_0)^p \di \gamma(x) < \infty,
\ee
for some $x_0 \in \Omega$.
\end{defi}
\noindent In this work we will be using the cases $\Omega = \bt^d \times \br^{d}$ and $p = 1,2$.

We note two important properties of the Wasserstein distances. We refer to \cite{Vil03} for proofs and further background. The first is monotonicity with respect to the order $p$.
\begin{lem} Let $\mu, \nu$ be probability measures on $(\Omega, \mc{F})$, each having a $q$th moment. Let $p \leq q$. Then
\be
W_p(\mu, \nu) \leq W_q(\mu,\nu).
\ee
\end{lem}

The second property is Kantorovich duality, which we will recall here only for the case $W_1$.
\begin{lem}[Kantorovich duality]
Let $\mu, \nu$ be probability measures $(\Omega, \mc{F})$, each having finite first moment. Then
\be
W_1(\mu, \nu) = \sup_{\phi\,,\, \|\phi\|_{\text{Lip}} \leq 1} \left \{ \int_\Omega \phi  \di \mu - \int_\Omega \phi \di \nu \right \} ,
\ee
where
$$
\|\phi\|_{\text{Lip}}:=\sup_{x\neq y}\frac{|\phi(x)-\phi(y)|}{|x-y|}.
$$
\end{lem}
\end{paragraph}

\begin{paragraph}{Analytic norms:}
For our results on the quasineutral limit, we will work with solutions that are close to being analytic, up to a small perturbation that may have much lower regularity. To measure analyticity we will use the norms $\lVert \cdot \rVert_{B_\delta}$, defined for $\delta > 1$ by
\be
\lVert g \rVert_{B_\delta} : = \sum_{k \in \bb{Z}^d} |\hat g(k)| \delta^{|k|} ,
\ee
where $\hat g(k)$ denotes the Fourier coefficient of $g$ of index $k$. 
This type of analytic norm was notably used in the context of quasineutral limits in the work of Grenier \cite{Grenier96}.
\end{paragraph}

\begin{paragraph}{Existence of solutions}
We will work in the class of solutions for the VPME system with bounded density. Solutions of \eqref{vpme} with bounded density have a uniqueness property, and exist for compactly supported initial data. Moreover, the energy \eqref{def:Ee} is conserved for these solutions. See \cite[Theorems 2.3-2.4]{IGP-WP}, which we recall below:

\begin{thm}[Uniqueness]
Let $d = 2, 3$. Consider an initial datum $f_0 \in \mc{P}(\bt^d \times \br^d)$. Then there exists at most one solution $f_t \in C([0,T] ; \mc{P}(\bt^d \times \br^d))$ of \eqref{vpme} with bounded density, i.e. such that the mass density $\rho_f$ lies in the space $L^\infty_{\text{loc}}([0, \infty) ; L^\infty(\bt^d))$.
\end{thm}

\begin{thm}[Existence of bounded density solutions] \label{thm:exist}
Let $d = 2, 3$. Consider an initial datum $f_0 \in \mc{P} \cap L^1 \cap L^\infty(\bt^d \times \br^d)$ with finite energy, namely $\mc{E} [f_0] \leq C_0$ for some constant $C_0>0$
and compact support in $\bt^d \times \br^d$. Then there exists a global in time solution $f_t \in C([0,\infty); \mc{P}(\bt^d \times \br^d))$ of \eqref{vpme} with initial data $f_0$, which satisfies
\be \label{solution-lp}
\lVert f_t \rVert_{L^p{(\bt^d \times \br^d})} \leq \lVert f_0 \rVert_{L^p{(\bt^d \times \br^d})}
\ee
for all $p \in [1,\infty]$, and has conserved energy and locally bounded mass:
\be \label{solution-energy}
\mc{E} [f_t] = \mc{E} [f_0], \qquad \rho_{f}  \in L^{\infty}_{\text{loc}} ([0, \infty) ; L^{\infty}(\bt^d)) .
\ee
\end{thm}

\end{paragraph}

\begin{paragraph}{Notation:}
Later on, we will use the notation $\mathbf{\overline{\exp}_n} $ to denote the $n$-fold iteration of the exponential function. For example
\be
\,\mathbf{\overline{\exp}_3}{(x)} : = \exp \exp  \exp{(x)} .
\ee

In our proofs, $C$ will denote a positive constant independent of the relevant parameters (i.e., $N,\e,r$, depending on the context) and that may change from line to line. Subscripts may be used to indicate a constant depending on a given parameter - for instance, $C_N$ denotes a constant depending on $N$ but independent of other parameters.
\end{paragraph}

\subsection{Quasineutral limit}\label{sec:statement-quasi}

Our first result is a quasineutral limit for initial data that are small but rough perturbations of uniformly analytic functions.
This extends the one dimensional result in \cite{IHK1} to higher dimensions, and generalises the results in \cite{IHK2} from the classical VP system to the massless electron case.
With respect to the arguments in \cite{IHK2,IHK1}, here we need to face several new challenges. Indeed, at the same time we need to deal with the nonlinear coupling in the Poisson equation (in \cite{IHK1} a solution to this issue was found only for $d=1$)
and with the fact that weaker regularity estimates are available for the Poisson equation (in \cite{IHK2} this was solved in the case of the classical linear Poisson equation). As already discussed in \cite{IHK2,IHK1}, we note that the exponential smallness (in $\e$) of the perturbation is necessary.

Let us observe that we work with compactly supported data in order to get control of the mass density of the solution. However we can allow the size of the support to grow at a controlled rate, exponential in $\e^{-1}$. These assumptions are chosen to match the results obtained in \cite{IHK2}. Although we believe that the hypothesis on the support may be slightly weakened to include densities that decay exponentially fast in velocity, achieving such extension here would go completely beyond the scope of this paper. 
The interested reader is referred to the papers \cite{IHK2, IHK1} for a discussion about possible initial data that satisfy our assumptions.

\begin{thm}[Quasineutral limit] \label{thm:quasi}
Let $d = 2, 3$. Consider initial data $f_\e(0)$ satisfying the following conditions:
\begin{itemize}
\item (Uniform bounds) $f_\e(0)$ is bounded and has bounded energy, uniformly with respect to $\e$:
\be \label{unif-energy}
\lVert f_\e(0) \rVert_{L^{\infty}(\bt^d \times \br^d)}  \leq C_0, \qquad\mc{E}_\e[f_\e(0)]  \leq C_0,
\ee
for some constant $C_0>0$.
\item (Control of support) There exists $C_1>0$ such that 
\be \label{quasi:data-spt}
f_\e(0, x, v) = 0 \qquad \text{for } |v| > \exp(C_1 \e^{-2}) .
\ee
\item (Perturbation of an analytic function) There exist $g_\e(0)$ satisfying, for some $\delta > 1$, $\eta>0$, and $C>0$,\begin{align}
\sup_\e \sup_{v \in \br^d} (1 + |v|^2) \lVert g_\e(0, \cdot, v) \rVert_{B_\delta} \leq C , \\ \label{unif-dens}
\sup_\e \left \| \int_{\br^d} g_\e(0, \cdot, v) \di v - 1  \right \|_{B_\delta} \leq \eta ,
\end{align}
as well as the support condition \eqref{quasi:data-spt}, such that
\be
W_2(f_\e(0), g_\e(0)) \leq \varphi(\e)
\ee
for some function $\varphi$ converging to 0 sufficiently quickly as $\e \to 0$. For example,
\be
\varphi(\e) = \left [ \,\mathbf{\overline{\exp}_4}  (C \e^{-2}) \right ]^{-1}
\ee
is an admissible choice for $C$ sufficiently large with respect to $C_0, C_1$.
\item(Convergence of data) $g_\e(0)$ has a limit $g(0)$ in the sense of distributions as $\e \to 0$.
\end{itemize}
Let $f_\e$ denote the unique solution of \eqref{vpme-quasi} with bounded density and initial datum $f_\e(0)$. Then there exist $T_*$, depending on the initial data, and a solution $g$ of \eqref{KIE} on the time interval $[0, T_*]$ with initial datum $g(0)$, such that
\be
\lim_{\e \to 0}\, \sup_{t \in [0, T_*]} W_1(f_\e(t), g(t)) = 0 .
\ee
\end{thm}

\begin{remark}
The condition \eqref{quasi:data-spt} should be understood as giving the fastest growth rate on the support for which the inverse quadruple exponential is an admissible choice of $\varphi$. In particular, this would still be the rate achievable by our methods even if the support of the data was uniform in $\e$.
\end{remark}

\subsection{Mean field limit}\label{sec:statement-MF}

As mentioned in the introduction, the validity of a particle system approximation to kinetic equations is a fundamental problem.
In Section~\ref{sec:MFL} we will derive the VPME system as the mean field limit of a regularised microscopic particle system. In this paper we will use a regularisation of the kind proposed by Lazarovici \cite{Laz}, however we expect that it would be possible to adapt other similar regularisation methods to this setting.

More precisely, we fix a smooth, compactly supported, radially symmetric function $\chi$ with unit mass and define
\be \label{Def_chi}
\chi_r(x) = r^{-d} \chi \left ( \frac{x}{r} \right ) .
\ee
We then consider a microscopic system describing the dynamics of a system of `delocalised ions' of shape $\chi_r$. For $1 \leq i \leq N$, let $(X_i, V_i)$ denote the position and velocity of the centre of the $i$th delocalised ion. The system is described by the following system of ODEs:
\be \label{ODEreg}
\begin{cases} \dot X_i = V_i \\
\dot V_i = \frac{1}{N} \sum_{i \neq j} \chi_r \ast K \ast \chi_r (X_i - X_j) - \chi_r \ast K \ast e^U ,
\end{cases}
\ee
where $U$ satisfies
\be \label{ODEreg-U}
\Delta U(x) = e^{U (x)} - \frac{1}{N} \sum_{i=1}^N \chi_r(x - X_i) 
\ee
and $K$ denotes the Coulomb kernel on $\bt^d$. That is, $K = \nabla G$ where $G$ is defined by \eqref{def:G}.
Note that we can rewrite the velocity equation as
$$
\dot V_i = - \chi_r \ast \nabla U (t, X_i) .
$$
and the equation for $U$ as
\be
\Delta U = e^U - \chi_r \ast \rho_{\mu^N},
\ee
where $\mu^N$ denotes the empirical measure as defined in \eqref{def:mu}. This is valid because $\chi_r \ast K \ast \chi_r(0) = 0$, and so
\be
\frac{1}{N} \sum_{i \neq j} \chi_r \ast K \ast \chi_r (X_i - X_j) = \frac{1}{N} \sum_{i =1}^N \chi_r \ast K \ast \chi_r (X_i - X_j) = \chi_r \ast K \ast \chi_r \ast \rho_{\mu^N}(X_i) .
\ee
Indeed, since $K$ is odd and $\chi_r$ is even, we have
\begin{align}
\chi_r \ast K \ast \chi_r (0) &= \int_{\bt^d \times \bt^d} \chi_r(-x) K(x-y) \chi_r (y) \di x \di y \\
& = \int_{\bt^d \times \bt^d} \chi_r(-y) K(y-x) \chi_r (x) \di x \di y \\
&= - \int_{\bt^d \times \bt^d} \chi_r(y) K(x-y) \chi_r (-x) \di x \di y \\
&= - \chi_r \ast K \ast \chi_r (0) .
\end{align}

The system \eqref{ODEreg}-\eqref{ODEreg-U} is of the form \eqref{ODE-gen}, but where the `external' potential $\Psi = \chi_r \ast G \ast e^U$ in fact depends nonlinearly on the solution of the system, via the Poisson equation \eqref{ODEreg-U}.

The next result shows a general statement about the validity of the mean field limit for the regularised particle system. We shall discuss later, in the context of Theorem \ref{thm:typ-MF}, assumptions on the regularisation parameter $r$ for which one can find initial data to which this result applies.

\begin{thm}[Regularised mean field limit] \label{thm:MFL}
Let $d=2,3$, and let $f_0$ be a choice of initial datum for \eqref{vpme} such that there exist a solution $f$ of \eqref{vpme} and, for each $r$, a solution $f_r$ of \eqref{vpme-reg} such that on some time interval $[0, T_*]$, all these solutions have bounded density: for some $M>0$,
\be
\| \rho_f \|_{L^\infty([0, T_*] \times \bt^d)}, \, \, \sup_r \| \rho_{f_r} \|_{L^\infty([0, T_*] \times \bt^d)} \leq M .
\ee
Assume that $r = r_N$ and the initial configurations for \eqref{ODEreg} are chosen such that the corresponding empirical measures satisfy, for some constant $C_{M,T_*}$, sufficiently large depending on $M$ and $T_*$,
\be \label{config-rate}
\limsup_{r \to 0} \frac{W_2^2(f_0, \mu^N_r(0)) }{r^{d + 2 + C_{M,T_*}|\log{r}|^{-1/2}}} < 1.
\ee
Then the empirical measure associated to the particle system dynamics starting from this configuration converges to $f$:
\be
\lim_{r \to 0} \sup_{t \in [0,T_*]} W_2(f_t, \mu^N_r(t)) = 0 .
\ee

\end{thm}

\subsection{Combined mean field and quasineutral limit}\label{sec:statement-MFQN}
As discussed in the introduction of \cite{IGP}, it is natural to 
 perform a combined mean field and quasineutral limit to derive the kinetic isothermal Euler system from a microscopic particle system. A result of this type for the classical VP system was proved in \cite{IGP}. Here, we extend this result to the case of massless electrons. Consider the following system for $(Z_i)_{i=1}^N = (X_i, V_i)_{i=1}^N \in (\br^{2d})^N$: 
\be\label{ODE:rQN}
\begin{cases} \dot X_i = V_i \\
\dot V_i = \frac{\e^{-2}}{N} \sum_{i \neq j} \chi_r \ast K \ast \chi_r (X_i - X_j) -  \e^{-2} \chi_r \ast K \ast e^{U_\e} ,
\end{cases}
\ee
where $U_\e$ satisfies
\be
\e^2 \Delta U_{\e}(x) = e^{U_\e(x)} - \frac{1}{N} \sum_{i=1}^N \chi_r(x - X_i) .
\ee
For this system, we can prove the following limit.

\begin{thm}[From particles to KIE] \label{thm:MFQN}
Let $d=2$ or $3$, and let $f_\e(0), g_\e(0)$ and $g(0)$ satisfy the assumptions of Theorem~\ref{thm:quasi}. Given $\e,r,N$, let $(Z_{0, i}^{(\e,r)})^N_{i=1} \in (\br^{2d})^N$ be a choice of initial data for the regularised and scaled $N$-particle ODE system \eqref{ODE:rQN}. Let $(Z_{t, i}^{(\e,r)})^N_{i=1}$ be the solution of \eqref{ODE:rQN} with this initial data and let $\mu^N_{\e,r}$ denote the associated empirical measure as defined in \eqref{def:mu}.

Let $T_*$ be the maximal time of convergence from Theorem~\ref{thm:quasi}. There exists a constant $C>0$ depending on $\{f_\e(0)\}_\e$ such that if the parameters $(r,\e) = (r(N), \e(N))$ and the initial data $\mu^N_{\e,r}(0)$ satisfy
\be \label{MFQN-initconv5}
r \leq \[ \,\mathbf{\overline{\exp}_3}{(C \e^{-2})} \]^{-1},\qquad
\lim_{N \to \infty}   \frac{W_2\left (\mu^N_{\e,r}(0),  f_\e(0)  \right)}{r^{(d+ 2+\eta)/2} } = 0,
\ee
for some $\eta > 0$, then
\be
\lim_{N \to \infty} \sup_{t \in [0,T_*]} W_1\left (\mu^N_{\e,r}(t), g(t) \right) = 0,
\ee
where $g$ is a solution of the KIE system \eqref{KIE} with initial data $g(0)$ on the time interval $[0,T_*]$.

\end{thm}

\subsection{Existence of admissible configurations}\label{sec:statement-typicality}

Note that
Theorems~\ref{thm:MFL} and \ref{thm:MFQN} hold only for special initial configurations for the particle system, i.e. those that converge sufficiently quickly to a given measure as $N$ tends to infinity. This raises the natural question of whether configurations satisfying this rate of convergence exist. The goal of the results in this section is to identify ranges of the parameters $r,N, \e$ for which suitable configurations exist. In fact we can show that there is, in some sense, a `large' set of admissible configurations.

One way to approximate a fixed measure $\nu$ by an empirical measure is to choose the points $(Z_i)_{i=1}^N$ defining the empirical measure by drawing independent samples from $\nu$. That is, $(Z_i)_{i=1}^N$ should have joint law $\nu^{\otimes N}$. This produces a random empirical measure $\nu^N$. Then, a law of large numbers result shows that $\nu^N$ converges to $\nu$ almost surely as $N$ tends to infinity. In Theorems~\ref{thm:typ-MF} and \ref{thm:typ-MFQN} we show that, for certain regimes of the parameters, this method of constructing the initial configurations will provide admissible configurations for Theorems~\ref{thm:MFL} and \ref{thm:MFQN} (respectively) with probability 1. 

\begin{thm}[Typicality for mean field limit with $\e=1$] \label{thm:typ-MF}
Let $d=2$ or $3$, and let $f_0$ be a choice of initial datum for \eqref{vpme} satisfying the assumptions of Theorem~\ref{thm:MFL} and having a finite $k$th moment for some $k > 4$:
\be
\int_{\bt^d \times \br^d} \(|x|^k + |v|^k\) f_0(\di x \di v) < + \infty.
\ee
Let $r = c N^{-\gamma}$ for some $\gamma$ satisfying
\be \label{def:gamma}
\gamma < \frac{1}{d+2} \min{\left \{ \frac{1}{d}, 1 - \frac{4}{k}\right \}} .
\ee
For each $N$, select initial configurations for the regularised $N$-particle system \eqref{ODEreg} by taking $N$ independent samples from $f_0$. Then with probability 1, this gives an admissible set of configurations for Theorem~\ref{thm:MFL}, i.e. the regularised mean field limit holds:
\be
\lim_{N \to \infty} \sup_{t \in [0,T_*]} W_2(\mu^N_r(t), f(t)) = 0.
\ee
\end{thm}
\begin{remark}\label{rmk:assumptions}
It is worth noticing that our assumptions on $r$ are the same as the ones found by Lazarovici for the classical VP system in \cite{Laz}.
\end{remark}

\begin{thm}[Typicality with quasineutral scaling] \label{thm:typ-MFQN}
Let $d=2$ or $3$, and
$f_\e(0), g_\e(0)$ and $g(0)$ satisfy the assumptions of Theorem~\ref{thm:quasi}. Let $(r,\e) = (r(N), \e(N))$ be chosen to satisfy
\be
r \leq \[ \,\mathbf{\overline{\exp}_3}{(C \e^{-2})} \]^{-1}
\ee
where $C>0$ is the constant from Theorem~\ref{thm:MFQN}. Assume that $r = r(N)$ are chosen such that
\be
r = c N^{-\gamma},
\ee
where $c>0$ is an arbitrary constant and $\gamma$ satisfies
\be \label{def:gamma-QN}
0\leq \gamma < \frac{1}{d(d+2)}.
\ee
For each $N$, let the initial configurations for the particle system \eqref{ODE:rQN} be chosen by taking $N$ independent samples from $f_{\e}(0)$. Then, with probability 1, this procedure selects a set of configurations for which the combined mean field and quasineutral limit holds, that is 
\be
\lim_{N \to \infty} \sup_{t \in [0,T_*]} W_1(\mu^N_{\e,r}(t), g(t)) = 0,
\ee
where $g$ is the solution of \eqref{KIE} with initial datum $g(0)$ on the time interval $[0, T_*]$ provided by Theorem~\ref{thm:quasi}.
\end{thm}
\begin{remark}
Observe that these assumptions ultimately result in a relationship between $N$ and $\e$ of the form 
\be
\e > \frac{C}{\sqrt{\log \log \log N}},
\ee 
so that the Debye length must converge to zero very slowly in relation to $N$. This rate is much slower than the one found for the classical case in \cite{IGP}. This is due to the singular nature of the interaction in the massless electrons case, which leads to the appearance of iterated exponential factors in our estimates for the electrostatic potential. 
\end{remark}

\section{Estimates on the electric field}
\label{sec:electric}

In this section we prove a series of quantitative regularity estimates on the electric field for the VPME system.
A crucial idea is to split the electrostatic potential $U_\e$ into two parts:
\be
U_\e = \bar U_\e + \widehat U_\e ,
\ee
where $\bar U_\e$ and $\widehat U_\e$ are solutions of the equations
\be
\e^2 \Delta \bar U_\e = 1 - \rho_\e, \qquad \e^2 \Delta \widehat U_\e = e^{\bar U_\e + \widehat U_\e} - 1 .
\ee
We will use the notation $\bar E_\e = - \nabla \bar U_\e$ and $\widehat E_\e = - \nabla \widehat U_\e$.
This approach was used previously in \cite{IHK1} and \cite{IGP-WP}. $\bar U_\e$ is the singular part of the potential, which behaves like the potential in the classical Vlasov-Poisson system, while $\widehat U_\e$ is a smoother correction.

\subsection{Regularity}

To state our regularity estimates on $\bar U_\e$ and $\widehat U_\e$, we introduce a space of zero-mean Sobolev functions, $\bar H$:
\be
\bar H : = \left \{ \phi \in W^{1,2}(\bt^d) : \int_{\bt^d} \phi (x) \di x = 0 \right \} .
\ee
We also define a set $A_d$ of allowed H\"{o}lder exponents, which depends on the dimension $d$.
\be \label{def:Ad}
A_d : =  \begin{cases} (0,1) \text{ if } d=2 \\ (0, \frac{1}{5}]  \text{ if } d=3 . \end{cases} .
\ee

The following is a version of Proposition~3.1 of \cite{IGP-WP} with quasineutral scaling. We will give the proof of these estimates below.
\begin{prop}[Regularity estimates on $\bar U_\e$ and $\widehat U_\e$] \label{prop:regU} Let $d = 2,3$. Let $h \in L^\infty \cap L^{(d+2)/d}(\bt^d)$. Then there exist unique $\bar U_\e \in \bar H$ and $\widehat U_\e \in W^{1,2} \cap L^{\infty}(\bt^d)$ satisfying
$$
\e^2\Delta \bar U_\e=1-h ,\qquad \e^2\Delta \widehat U_\e=e^{\bar U_\e+\widehat U_\e} - 1.
$$
Moreover we have the following estimates for some constant $C_{\alpha,d}>0$:
\begin{align}
\lVert  \bar U_\e \rVert_{C^{0,\alpha}(\bt^d)} & \leq C_{\alpha, d} \,\e^{-2} \left (1 +  \lVert h \rVert_{L^{\frac{d+2}{d}}(\bt^d)} \right ), &
\alpha \in A_d  \\
\lVert \bar U_\e \rVert_{C^{1, \alpha}(\bt^d)} & \leq C_{\alpha, d} \,\e^{-2} \left (1 +  \lVert h \rVert_{L^{\infty}(\bt^d)} \right ),  & \alpha \in (0,1) \\
\|\widehat U_\e \|_{C^{1,\alpha}(\bt^d)} & \leq C_{\alpha,d} \exp{\Bigl(C_{\alpha,d} \e^{-2}  \Bigl(1 +  \lVert h \rVert_{L^{\frac{d+2}{d}}(\bt^d)} \Bigr) \Bigr)} & \alpha \in (0,1)\\
\|\widehat U_\e \|_{C^{2,\alpha}(\bt^d)} & \leq C_{\alpha,d} \,\mathbf{\overline{\exp}_2} {\Bigl(C_{\alpha,d} \e^{-2}  \Bigl(1 +  \lVert h \rVert_{L^{\frac{d+2}{d}}(\bt^d)} \Bigr) \Bigr)} & \alpha \in A_d .
\end{align}

\end{prop}

\subsubsection{Regularity of $\bar U_\e$}

The singular part of the potential, $\bar U_\e$, satisfies a classical Poisson equation on $\bt^d$:
\be \label{eq:Poisson}
\e^2 \Delta \bar U_\e = 1 - h .
\ee
The existence of $\bar U_\e$ for $h \in L^\infty \cap L^{(d+2)/d}(\bt^d)$ is well understood, and there is no loss of generality to assume that $\bar U_\e$ has zero mean.

The regularity theory of \eqref{eq:Poisson} is likewise well understood. If $h \in L^p(\bt^d)$, then by Calder\'{o}n-Zygmund estimates for the Laplacian \cite{GT} we have the estimate
\be
\| \bar U_\e \|_{W^{2,p}(\bt^d)} \leq C_d \e^{-2} \left ( 1 + \| h \|_{L^p(\bt^d)} \right ) .
\ee

We can then use Sobolev embedding \cite{Evans} to deduce higher-order integrability or H\"{o}lder regularity for $\bar U_\e$. In the case $p = \frac{d+2}{d}$, we obtain
\be
\| \bar U_\e \|_{C^{0, \alpha}} \leq C_{\alpha, d} \e^{-2} \left ( 1 + \| h \|_{L^{\frac{d+2}{d}}(\bt^d)} \right )
\ee
for $\alpha \in A_d$, where $A_d$ is defined by \eqref{def:Ad}. In the case $p = \infty$ we obtain
\be
\lVert \bar U_\e \rVert_{C^{1, \alpha}(\bt^d)} \leq C_{\alpha, d} \,\e^{-2} \left (1 +  \lVert h \rVert_{L^{\infty}(\bt^d)} \right )
\ee
for all $ \alpha \in (0,1)$.

Later we will perform estimates on the trajectories of the flow induced by $E_\e$. In order to do this, we would like to have Lipschitz regularity for $E_\e$. Unfortunately, for $\bar E_\e$ we are not able to prove Lipschitz regularity under our assumptions. However, we do have a log-Lipschitz regularity estimate. See \cite[Lemma 8.1]{BM} for the proof in the case where the spatial domain is $\br^2$, or \cite[Lemma 3.3]{IGP-WP} for the case $\bt^d$ for general $d$.

\begin{lem}[Log-Lipschitz regularity of $\bar E_\e$] \label{lem:logLip}
Let $\bar U_\e$ be a solution of 
$$
\e^2\Delta \bar U_\e= h
$$
for $h \in L^{\infty}(\bt^d)$. Then
$$
|\nabla \bar U_\e (x) - \nabla \bar U_\e(y)| \leq \e^{-2} C \lVert h \rVert_{L^{\infty}} |x-y|  \left ( 1 + \log{\left ( \frac{\sqrt{d}}{|x-y|} \right )} \mathbbm{1}_{|x-y|\leq \sqrt{d}} \right ) .
$$
\end{lem}

\subsubsection{Regularity of $\widehat U_\e$}

In \cite[Proposition 3.4]{IGP-WP}, we proved the existence of a unique $W^{1,2}(\bt^d)$ solution $\widehat U$ of the equation
\be \label{eq:Uhat}
\Delta \widehat U = e^{\bar U + \widehat U} - 1,
\ee
under the assumption that $\bar U \in L^\infty (\bt^d) \cap W^{1,2}(\bt^d)$. We showed that this solution in fact belongs to the H\"{o}lder space $C^{1,\alpha}(\bt^d)$ for $\alpha \in (0,1)$. Furthermore, if $\bar U \in C^{0,\alpha}(\bt^d)$, then $\widehat U \in C^{2,\alpha}(\bt^d)$. The argument in \cite[Proposition 3.4]{IGP-WP} works for general $\e$. We will revisit part of the proof here in order to quantify how the constants depend on $\e$. This dependence will be crucial in our later proofs.

\begin{lem} \label{lem:Uhat}
Consider the nonlinear Poisson equation
\be \label{eq:Uhat-e}
\e^2 \Delta \widehat U_\e = e^{\bar U_\e + \widehat U_\e} - 1 .
\ee
Assume that $\bar U_\e \in L^\infty (\bt^d) \cap W^{1,2}(\bt^d)$ with
\be
\| \bar U_\e \|_{L^\infty(\bt^d)} \leq M_1 .
\ee
Then there exists a unique solution $\widehat U_\e \in W^{1,2}(\bt^d)$ of \eqref{eq:Uhat-e}. Moreover, $\widehat U_\e \in C^{1,\alpha}(\bt^d)$ for all $\alpha \in (0,1)$, with the estimate
\be \label{hatU-C1}
\| \widehat U_\e \|_{C^{1,\alpha}(\bt^d)} \leq C \e^{-2} \left ( e^{2 M_1} + 1 \right ) .
\ee

If, in addition, $\bar U_\e$ is H\"{o}lder regular with the estimate
\be
\| \bar U_\e \|_{C^{0,\alpha}(\bt^d)} \leq M_2
\ee
for some $\alpha \in (0,1)$, then $\widehat U_\e \in C^{2,\alpha}(\bt^d)$ with the estimate
\be \label{hatU-C2}
\| \widehat U_\e \|_{C^{1,\alpha}(\bt^d)} \leq \left [ M_2 + C \e^{-2} \left ( e^{2 M_1} + 1 \right )\right ] \exp{\left [ C \e^{-2} \left ( e^{2 M_1} + 1 \right )\right ]} .
\ee
\end{lem}

\begin{proof}
The existence and uniqueness of $\widehat U_\e \in W^{1,2}(\bt^d)$ satisfying \eqref{eq:Uhat-e} can be proved by using a Calculus of Variations method, as in \cite[Proposition 3.4]{IGP-WP}. Here we will focus on the regularity estimates \eqref{hatU-C1} and \eqref{hatU-C2}.

First, we want to prove $L^p(\bt^d)$ estimates on $e^{U_\e}$. We will use this to deduce estimates on $\widehat U_\e$ using regularity theory for the Poisson equation. Since $\widehat U_\e$ is a weak solution of \eqref{eq:Uhat-e}, we know that for all test functions $\phi \in W^{1,2}(\bt^d) \cap L^\infty(\bt^d)$,
\be
\int_{\bt^d} \e^2 \, \nabla \widehat U_\e \cdot \nabla \phi + \left ( e^{\bar U_\e + \widehat U_\e} - 1 \right ) \phi \di x = 0 .
\ee
If we formally take $e^{\widehat U_\e}$ as a test function, we obtain
\be \label{eq:Uhat-wk}
\e^2 \int_{\bt^d} |\nabla \widehat U_\e|^2 e^{\widehat U_\e} \di x + \int_{\bt^d} e^{\bar U_\e} \cdot e^{2 \widehat U_\e} \di x = \int_{\bt^d} e^{\widehat U_\e} \di x .
\ee
Then
\be
 \int_{\bt^d} e^{\bar U_\e} \cdot e^{2 \widehat U_\e} \di x \leq \int_{\bt^d} e^{\widehat U_\e} \di x .
\ee
We can make this estimate rigorous by using a truncation argument, as described in \cite{IGP-WP}. Since
\be
\|\bar U_\e\|_{L^\infty(\bt^d)} \leq M_1,
\ee
we have
\be
 \int_{\bt^d} e^{\bar U_\e} \cdot e^{2 \widehat U_\e} \di x \geq e^{- M_1} \| e^{\widehat U_\e} \|^2_{L^2(\bt^d)} .
\ee
Similarly,
\be
\| e^{U_\e} \|_{L^1(\bt^d)} = \int_{\bt^d} e^{\bar U_\e + \widehat U_\e} \di x \geq e^{- M_1}  \int_{\bt^d} e^{\widehat U_\e} \di x .
\ee
By \eqref{eq:Uhat-e},
\be
0 = \e^2 \int_{\bt^d} \Delta \widehat U_\e \di x = \int_{\bt^d} \left ( e^{U_\e} - 1 \right ) \di x,
\ee
which implies that
\be
\| e^{U_\e} \|_{L^1(\bt^d)} = 1 .
\ee
Therefore,
\be
 \| e^{\widehat U_\e} \|^2_{L^2(\bt^d)} \leq e^{2 M_1} .
\ee

Similarly, taking $e^{(n-1) \widehat U_\e}$ as a test function, we have 
\be
\int_{\bt^d} e^{(n-1) \widehat U_\e} \di x = \e^2 (n-1) \int_{\bt^d} |\nabla \widehat U_\e|^2 e^{\widehat U_\e} \di x + \int_{\bt^d} e^{\bar U_\e} \cdot e^{n \widehat U_\e} \di x 
\ee
and so
\be
 \| e^{\widehat U_\e} \|^n_{L^n(\bt^d)} \leq e^{M_1}  \| e^{\widehat U_\e} \|^{n-1}_{L^{n-1}(\bt^d)} \, ;
\ee
again we can make this rigorous by a truncation argument. By induction, we conclude that for all integer $n$,
\be \label{est:eU-Lp}
 \| e^{\widehat U_\e} \|_{L^n(\bt^d)} \leq e^{M_1} .
\ee

Next, we use these estimates to deduce regularity for $\widehat U_\e$. By Calder\'{o}n-Zygmund estimates for the Poisson equation,
\begin{align}
\| \widehat U_\e \|_{W^{2,n}(\bt^d)}  \leq C \e^{-2} \| e^{\bar U_\e + \widehat U_\e} \|_{L^n(\bt^d)} 
 \leq C \e^{-2} \left ( e^{2 M_1} + 1 \right ) .
\end{align}
By Sobolev embedding with $n$ sufficiently large, for any $\alpha \in (0,1)$,
\be
\| \widehat U_\e \|_{C^{1,\alpha}(\bt^d)} \leq C \e^{-2} \left ( e^{2 M_1} + 1 \right ) .
\ee

Now assume that for some $\alpha$,
\be
\| \bar U_\e \|_{C^{0,\alpha}(\bt^d)} \leq M_2 .
\ee
Then $U_\e \in C^{0,\alpha}(\bt^d)$, with
\begin{align}
\| U_\e \|_{C^{0,\alpha}(\bt^d)}  \leq \| \bar U_\e \|_{C^{0,\alpha}(\bt^d)} + \| \widehat U_\e \|_{C^{0,\alpha}(\bt^d)}  \leq M_2 + C \e^{-2} \left ( e^{2 M_1} + 1 \right ) .
\end{align}
Since 
\be
\left \lvert e^{U_\e(x)} - e^{U_\e(y)} \right \rvert \leq e^{\max\{ U_\e(x), U_\e(y) \}} \left | {U_\e(x)} - {U_\e(y)} \right | ,
\ee
it follows that
\begin{align}
\|e^{U_\e} \|_{C^{0,\alpha}(\bt^d)} \leq \| U_\e \|_{C^{0,\alpha}(\bt^d)} \exp{ \left [ C \e^{-2} \left ( 1 + e^{2 M_1} \right ) \right ]} .
\end{align}
By Schauder estimates \cite[Chapter 4]{GT},
\begin{align}
\| \widehat U_\e \|_{C^{2,\alpha}(\bt^d)} & \leq C \left ( \| \widehat U_\e \|_{L^\infty (\bt^d)} + \e^{-2 }\|e^{U_\e} - 1 \|_{C^{0,\alpha}(\bt^d)} \right ) \\
& \leq  \left [ M _2 +  C \e^{-2} \left ( 1 + e^{2 M_1} \right ) \right ] \exp{ \left [  C \e^{-2} \left ( 1 + e^{2 M_1} \right ) \right ] } .
\end{align}

\end{proof}

\subsection{Stability}

We are also interested in the stability of $U_\e$ with respect to the charge density. We want to prove a quantitative version of  \cite[Proposition 3.5]{IGP-WP}.
\begin{prop} \label{prop:Ustab}
For each $i=1,2$, let $\bar U_\e^{(i)}$ be a zero-mean solution of
\begin{align}
\e^2\Delta \bar U_\e^{(i)}= h_i - 1,
\end{align}
where $h_i \in L^{\infty} \cap L^{(d+2)/d}(\bt^d)$. Then
\be \label{stab-Ubar}
\lVert \nabla \bar U_\e^{(1)} - \nabla \bar U_\e^{(2)} \rVert^2_{L^2(\bt^d)} \leq \e^{-4} \max_i\, \lVert h_i \rVert_{L^{\infty}(\bt^d)}  \, W^2_2(h_1, h_2) .
\ee

Now, in addition, let $\widehat U_\e^{(i)}$ be a solution of
\begin{align}
\e^2\Delta \widehat U_\e^{(i)}= e^{\bar U_\e^{(i)} + \widehat U_\e^{(i)}} - 1 .
\end{align}
Then
\begin{align}
\label{stab-Uhat}
\lVert \nabla \widehat U_\e^{(1)} - \nabla \widehat U_\e^{(2)} \rVert^2_{L^2(\bt^d)} &\leq \mathbf{\overline{\exp}_2} {\left [ C_d \e^{-2} \left ( 1 + \max_i\, \lVert h_i \rVert_{L^{(d+2)/d}(\bt^d)}  \right ) \right ]}\\
&\times \max_i\, \lVert h_i \rVert_{L^{\infty}(\bt^d)} \, W^2_2(h_1, h_2).
\end{align}
\end{prop}

For the stability of $\nabla \bar U_\e$, we use a result proved originally by Loeper \cite{Loep} in the whole space $\br^d$, and adapted to the torus $\bt^d$ in \cite{IHK2}.
\begin{lem}[Loeper-type estimate for Poisson's equation] \label{lem:Loep}
For each $i=1,2$, let $\bar U_\e^{(i)}$ be a solution of
\begin{align}
\e^2\Delta \bar U_\e^{(i)}= h_i - 1,
\end{align}
where $h_i \in L^{\infty}(\bt^d)$. Then
$$
\lVert \nabla \bar U_\e^{(1)} - \nabla \bar U_\e^{(2)} \rVert^2_{L^2(\bt^d)} \leq \e^{-4} \max_i\, \lVert h_i \rVert_{L^{\infty}(\bt^d)} \, W^2_2(h_1, h_2) .
$$
\end{lem} 

For $\widehat U_\e$ we use a version of \cite[Lemma 3.7]{IGP-WP} in which we quantify the dependence of the constants on $\e$.
\begin{lem} \label{lem:hatU-stab}
For each $i=1,2$, let $\widehat U_\e^{(i)} \in  W^{1,2} \cap L^{\infty}(\bt^d)$ be a solution of
\begin{align} \label{PoiU}
\e^2\Delta \widehat U^{(i)}_\e&=e^{\bar U^{(i)}_\e+\widehat U^{(i)}_\e} - 1 ,
\end{align}
for some given potentials $\bar{U}^{(i)}_\e \in L^{\infty}(\bt^d)$ . Then
\begin{align} \label{H1dot}
\e^2 \lVert \nabla \widehat{U}^{(1)}_\e - \nabla \widehat{U}^{(2)}_\e \rVert^2_{L^2(\bt^d)} \leq C_\e \lVert \bar{U}^{(1)}_\e - \bar{U}^{(2)}_\e \rVert^2_{L^2(\bt^d)} ,
\end{align}
where $C$ depends on the $L^{\infty}$ norms of $\widehat{U}^{(i)}_\e$ and $\bar{U}^{(i)}_\e$. More precisely, $C_\e$ can be chosen such that
$$
C_\e \leq  \exp{\left [ C \left ( \max_i\,  \lVert \bar{U}^{(i)}_\e \rVert_{L^{\infty}(\bt^d)} + \max_i\, \lVert \widehat{U}^{(i)}_\e \rVert_{L^{\infty}(\bt^d)}  \right )\right ]} ,
$$
for some sufficiently large constant $C$, independent of $\e$.

\end{lem}

\begin{proof}
We look at the equation solved by the difference $\widehat U^{(1)}_\e - \widehat U^{(2)}_\e$. By subtracting the two equations in \eqref{lem:hatU-stab}, we find that
\be \label{eq:hatU-diff}
\e^2 \Delta (\widehat U^{(1)}_\e - \widehat U^{(2)}_\e) = e^{\bar U^{(1)}_\e+\widehat U^{(1)}_\e} -e^{\bar U^{(2)}_\e+\widehat U^{(2)}_\e} .
\ee
By assumption, $\widehat U_\e^{(i)} \in  W^{1,2} \cap L^{\infty}(\bt^d)$. Hence $\widehat U^{(1)}_\e - \widehat U^{(2)}_\e$ may be used as a test function in the weak form of \eqref{eq:hatU-diff}. We multiply both sides by $\widehat U^{(1)}_\e - \widehat U^{(2)}_\e$, integrate over $\bt^d$, and integrate by parts to find
\be
- \e^2 \int_{\bt^d} \lvert \nabla \widehat{U}^{(1)}_\e - \nabla \widehat{U}^{(2)}_\e \rvert^2 \di x = \int_{\bt^d } \left ( e^{\bar{U}^{(1)}_\e + \widehat{U}^{(1)}_\e} - e^{\bar{U}^{(2)}_\e + \widehat{U}^{(2)}_\e} \right ) ( \widehat{U}^{(1)}_\e - \widehat{U}^{(2)}_\e )  \di x .
\ee

In the proof of \cite[Lemma 3.7]{IGP-WP}, we showed that for any $\bar U, \widehat U, \bar V, \widehat V \in L^\infty(\bt^d)$,
\be
-  \int_{\bt^d } \left ( e^{\bar{U} + \widehat{U}} - e^{\bar{V} + \widehat{V}} \right ) ( \widehat{U} - \widehat{V})  \di x \leq C_{U,V} \| \bar U - \bar V \|^2_{L^2(\bt^d)} ,
\ee
where
\be
C_{U,V} \leq \exp{ \left \{ C \left [ \|\bar U\|_{L^\infty(\bt^d)} + \|\widehat U\|_{L^\infty(\bt^d)} + \|\bar V\|_{L^\infty(\bt^d)} + \|\widehat V\|_{L^\infty(\bt^d)}\right ] \right \}} .
\ee
Thus
\be
\| \nabla \widehat{U}^{(1)}_\e - \nabla \widehat{U}^{(2)}_\e \|^2_{L^2(\bt^d)} \leq C_\e \, \e^{-2} \| \bar U^{(1)}_\e - \bar U^{(2)}_\e \|^2_{L^2(\bt^d)} ,
\ee
where
\be
C_\e \leq  \exp{\left [ C \left ( \max_i\,  \lVert \bar{U}^{(i)}_\e \rVert_{L^{\infty}(\bt^d)} + \max_i\, \| \widehat{U}^{(i)}_\e \|_{L^{\infty}(\bt^d)}  \right )\right ]} .
\ee
\end{proof}

\begin{proof}[Proof of Proposition~\ref{prop:Ustab}]
It suffices to prove \eqref{stab-Uhat}, since \eqref{stab-Ubar} follows immediately from Lemma~\ref{lem:Loep}.
By the Poincar\'{e} inequality for zero-mean functions,
\be
\| \bar U^{(1)}_\e - \bar U^{(2)}_\e \|^2_{L^2(\bt^d)} \leq C \| \nabla \bar U^{(1)}_\e - \nabla \bar U^{(2)}_\e \|^2_{L^2(\bt^d)} .
\ee
By Lemma~\ref{lem:Loep},
\be
\| \nabla \bar U_\e^{(1)} - \nabla \bar U_\e^{(2)} \|^2_{L^2(\bt^d)} \leq \e^{-4} \max_i\, \| h_i \|_{L^{\infty}(\bt^d)} \, W^2_2(h_1, h_2) .
\ee
Then, by Lemma~\ref{lem:hatU-stab},
\be
\lVert \nabla \widehat{U}^{(1)}_\e - \nabla \widehat{U}^{(2)}_\e \rVert^2_{L^2(\bt^d)} \leq C_\e \, \e^{-6}   \max_i\, \| h_i \|_{L^{\infty}(\bt^d)} \, W^2_2(h_1, h_2)  ,
\ee
for
\be
C_\e \leq  \exp{\left [ C \left ( \max_i\,  \lVert \bar{U}^{(i)}_\e \rVert_{L^{\infty}(\bt^d)} + \max_i\, \lVert \widehat{U}^{(i)}_\e \rVert_{L^{\infty}(\bt^d)}  \right )\right ]} .
\ee

For the $L^\infty(\bt^d)$ estimates, we use Proposition~\ref{prop:regU}:
\begin{align}
\| \bar U^{(i)}_\e \|_{L^\infty(\bt^d)} & \leq C_d \, \e^{-2} \left ( 1 + \| h_i \|_{L^{\frac{d+2}{d}}} \right ) \\
\| \widehat U^{(i)}_\e \|_{L^\infty(\bt^d)} & \leq C_d \exp {\left [ C_d \e^{-2} \left ( 1 + \| h_i \|_{L^{\frac{d+2}{d}}} \right ) \right ]} .
\end{align}
Hence
\be
C_\e \leq  \mathbf{ \overline{\exp}_2 }{\left [ C_d \e^{-2} \left ( 1 + \max_i \| h_i \|_{L^{\frac{d+2}{d}}} \right ) \right ]},
\ee
which implies \eqref{stab-Uhat}.

\end{proof}

\section{Strong-strong stability}
\label{sec:stability}

To prove the quasineutral limit, we will need a quantified stability estimate between solutions of the VPME system \eqref{vpme}. The following proposition is an $\e$-dependent version of \cite[Proposition 4.1]{IGP-WP}. We revisit the proof here, keeping track of how all the constants depend on $\e$.

\begin{prop}[Stability for solutions with bounded density] \label{prop:Wstab}
For $i = 1,2$, let $f_\e^{(i)}$ be solutions of \eqref{vpme-quasi} satisfying for some constant $M$ and all $t \in [0,T]$,
\be \label{str-str_rho-hyp}
\rho[f_\e^{(i)}(t)] \leq M .
\ee
Then there exists a constant $C_\e$ such that, for all $t \in [0,T]$,
\be
W_2\(f_\e^{(1)}(t), f_\e^{(2)}(t)\) \leq \begin{cases}
C \exp{\biggl [ C\Bigl(1 + \log{\frac{W_2\(f_\e^{(1)}(0), \, f_\e^{(2)}(0)\)}{4 \sqrt{d}}}\Big) e^{-C_\e t} \biggr ]} & \text{ if } W_2\(f_\e^{(1)}(0), f_\e^{(2)}(0)\) \leq d \\
W_2\(f_\e^{(1)}(0), f_\e^{(2)}(0)\) e^{C_\e t} & \text{ if } W_2\(f_\e^{(1)}(0), f_\e^{(2)}(0)\) > d .
\end{cases}
\ee
If in addition, for some constant $C_0$,
\begin{equation}
\label{eq:hypothesis f}
\sup_{t \in [0,T]} \lVert f_\e^{(i)}(t, \cdot, \cdot) \rVert_{L^{\infty}_{x,v}} \leq C_0,\qquad
\sup_{t \in [0,T]} \mc{E}_\e[f_\e^{(i)}](t)\leq C_0,
\end{equation}
then $C_\e$ may be chosen to satisfy
$$
C_\e \leq  \,\mathbf{\overline{\exp}_2}{(C \e^{-2})} (M + 1) .
$$
\end{prop}

\begin{proof}

The $W_2$ distance is defined as an infimum over couplings \eqref{def:Wass}. Thus it suffices to estimate the $L^2$ distance corresponding to a particular coupling of $f_\e^{(1)}$ and $f_\e^{(2)}$. To do this we will first represent $f_\e^{(i)}$ as the pushforward of $f_\e^{(i)}(0)$ along the characteristic flow induced by $f_\e^{(i)}$. That is, consider the following system of ODEs for $Z^{(i)}_z = \left ( X^{(i)}_z, V^{(i)}_z \right ) \in \bt^d \times \br^d$:
\begin{align} \label{char-Zi}
& \begin{cases}
\dot X^{(i)}_z = V^{(i)}_z, \\
\dot V^{(i)}_z = E^{(i)}_\e \bigl( X^{(i)}_z \bigr), \\
 \bigl( X^{(i)}_z(0), V^{(i)}_z(0) \bigr) = (x,v) = z,
\end{cases}
\end{align}
where the electric field $E^{(i)}_\e$ is given by
\begin{align}
E^{(i)}_\e &= - \nabla U^{(i)}_\e, \\
\e^2 \Delta U^{(i)}_\e &= e^{U^{(i)}_\e} - \rho^{(i)}_\e := e^{U^{(i)}_\e} - \rho[f^{(i)}_\e] .
\end{align}
Since $f^{(i)}_\e$ has bounded mass density, Proposition~\ref{prop:regU} and Lemma~\ref{lem:logLip} imply that $E^{(i)}_\e$ is a log-Lipschitz vector field, with a constant uniform on $[0,T]$. This regularity is enough to imply that there exists a unique solution of \eqref{char-Zi} for every initial condition $z$, resulting in a well-defined characteristic flow. Since the characteristic flow is unique, by \cite[Theorem 3.1]{Amb} the linear Vlasov equation
\be \label{vlasov-lin} \left \{
\begin{array}{r}
\partial_t g + v \cdot \nabla_x g + E_\e^{(i)} (x) \cdot \nabla_v g = 0, \\
g \vert_{t=0} = f^{(i)}_\e(0) \geq 0
\end{array} \right.
\ee
has a unique solution. Moreover, this solution can be represented, in weak form, by the relation
\be \label{rep-pshfwd}
\int_{\bt^d \times \br^d} \phi \, g_t(\di z) = \int_{\bt^d \times \br^d} \phi \left ( Z^{(i)}_z \right )  f^{(i)}_\e(0)(\di z) ,
\ee
for all $\phi \in C_b(\bt^d \times \br^d)$. Since $f^{(i)}_\e$ is certainly a solution of \eqref{vlasov-lin}, we deduce that $g=f^{(i)}_\e$ and so $f^{(i)}_\e$ has the representation \eqref{rep-pshfwd}. We will use this representation to define a coupling between $f^{(1)}_\e$ and $f^{(2)}_\e$.

Fix an arbitrary initial coupling $\pi_0 \in \Pi \left [ f^{(1)}_\e(0), f^{(2)}_\e(0)\right ]$. We define $\pi_t$ to follow the corresponding characteristic flows: for $\phi \in C_b((\bt^d \times \br^d)^2)$, let
\be \label{def:pit0}
\int_{(\bt^d \times \br^d)^2} \phi(z_1, z_2) \di \pi_t(z_1, z_2) = \int_{\bt^d \times \br^d} \phi \left (Z^{(1)}_{z}, Z^{(2)}_{z} \right ) \di \pi_0(z_1, z_2) .
\ee
We can very that $\pi_t$ is indeed a coupling of $f^{(1)}_\e$ and $f^{(2)}_\e$ by checking the marginals:
\begin{align} \label{pit-marg}
\int_{(\bt^d \times \br^d)^2} \phi(z_{i}) \di \pi_t(z_1, z_2) & = \int_{(\bt^d \times \br^d)^2} \phi \left (Z^{(i)}_{z_i} (t) \right ) \di \pi_0(z_1, z_2) \\
& = \int_{\bt^d \times \br^d} \phi \left (Z^{(i)}_{z} (t) \right )  f^{(i)}_\e(0)( \di z ) \\
& = \int_{\bt^d \times \br^d} \phi(z) f^{(i)}_\e(t)( \di z ) .
\end{align}

Next, we define a functional which is greater than (or equal to) the squared Wasserstein distance between $f^{(1)}_\e$ and $f^{(2)}_\e$. Let
\be \label{def:D}
D(t) = \int_{(\bt^d \times \br^d)^2} |x_1-x_2|^2 + |v_1-v_2|^2 \di \pi_t(x_1,v_1,x_2,v_2) .
\ee
Since the Wasserstein distance is an infimum over couplings, while $\pi_t$ is a particular coupling, $D$ must control the squared Wasserstein distance:
\be
W_2^2(f^{(1)}_\e, f^{(2)}_\e) \leq D .
\ee
To prove our stability estimate, it therefore suffices to estimate $D$. We will do this using a Gr\"{o}nwall estimate. Taking a time derivative, we find
\begin{multline}
\dot D = 2 \int_{(\bt^d \times \br^d)^2} \left ( X^{(1)}_{z_1} (t) - X^{(2)}_{z_2} (t) \right ) \cdot \left ( V^{(1)}_{z_1} (t) - V^{(2)}_{z_2} (t) \right )\\
 + \left ( V^{(1)}_{z_1} (t) - V^{(2)}_{z_2} (t) \right ) \cdot \left ( E^{(1)}_\e (X^{(1)}_{z_1} (t)) - E^{(2)}_\e (X^{(2)}_{z_2} (t) ) \right ) \di \pi_0(z_1, z_2) .
\end{multline}
Using the Cauchy-Schwarz inequality, we obtain
\be
\dot D \leq D + \sqrt{D} \left ( \int_{(\bt^d \times \br^d)^2} \left | E^{(1)}_\e (X^{(1)}_{z_1} (t)) - E^{(2)}_\e (X^{(2)}_{z_2} (t) ) \right |^2 \di \pi_0(z_1, z_2) \right )^{1/2} .
\ee
In other words,
\be
\dot D \leq D + \sqrt{D} \left ( \int_{(\bt^d \times \br^d)^2} \left | E^{(1)}_\e (x_1) - E^{(2)}_\e (x_2 ) \right |^2 \di \pi_t(z_1, z_2) \right )^{1/2} .
\ee
We split the electric field term into the form
\be
\int_{(\bt^d \times \br^d)^2} \left | E^{(1)}_\e (x_1) - E^{(2)}_\e (x_2 ) \right |^2  \di \pi_t(z_1, z_2) \leq C \sum_{i=1}^4 I_i ,
\ee
where
\be \label{def:Ii-MF0}
\begin{split}
I_1 & := \int_{(\bt^d \times \br^d)^2} |\bar E_\e^{(1)}(x_1) - \bar E^{(1)}_\e(x_2)|^2 \di \pi_t , \qquad I_2  := \int_{(\bt^d \times \br^d)^2} |\bar E_\e^{(1)}(x_2) - \bar E^{(2)}_\e(x_2)|^2 \di \pi_t, \\
I_3 & := \int_{(\bt^d \times \br^d)^2} |\widehat E_\e^{(1)}(x_1) - \widehat E^{(1)}_\e(x_2)|^2 \di \pi_t,
\qquad
I_4  := \int_{(\bt^d \times \br^d)^2} |\widehat E_\e^{(1)}(x_2) - \widehat E^{(2)}_\e(x_2)|^2 \di \pi_t .
\end{split}
\ee

\begin{paragraph}{Control of $I_1$:}
To estimate $I_1$, observe that by Lemma~\ref{lem:logLip},
\be
I_1 \leq C \e^{-4} \lVert \rho_\e^{(1)} \rVert^2_{L^{\infty}(\bt^d)} \int_{(\bt^d \times \br^d)^2}  |x_1-x_2|^2  \left ( 1 + \log{\left ( \frac{\sqrt{d}}{|x_1-x_2|} \right )} \mathbbm{1}_{|x_1-x_2|\leq \sqrt{d}} \right )^2  \di \pi_t .
\ee
We introduce the function
$$
H(x) := \begin{cases}
x \left (\log{\frac{x}{16d}} \right )^2 \text{ if } x \leq d \\
d \left (\log{16} \right )^2 \text{ if } x > d,
\end{cases}
$$
which is concave on $\br_+$ (see \cite[Lemma 4.4]{IGP-WP}). Then
\be
I_1 \leq C \e^{-4} \lVert \rho_\e^{(1)} \rVert^2_{L^{\infty}(\bt^d)}  \int_{(\bt^d \times \br^d)^2} H \left ( |x_1 - x_2|^2 \right ) \di \pi_t  .
\ee

By Jensen's inequality,
\begin{align}
I_1  \leq C \e^{-4} \lVert \rho_\e^{(1)} \rVert^2_{L^{\infty}(\bt^d)} \, H \left ( \int_{(\bt^d \times \br^d)^2} |x_1 - x_2|^2 \di \pi_t \right )  \leq C \e^{-4} M^2 H(D) .
\end{align}

\end{paragraph}

\begin{paragraph}{Control of $I_2$:}
For $I_2$, observe that
\begin{align}
 I_2  &\leq \int_{\bt^d} |\bar E_\e^{(1)}(x) - \bar E^{(2)}_\e(x)|^2 \rho_\e^{(2)}(\di x) \\
 & \leq \lVert \rho_\e^{(2)} \rVert_{L^{\infty}(\bt^d)} \| \bar E_\e^{(1)} - \bar E^{(2)}_\e \|^2_{L^2(\bt^d)} .
\end{align}
We apply the Loeper stability estimate (Lemma~\ref{lem:Loep}) to obtain
\begin{align}
I_2 \leq \e^{-4} \max_i\, \lVert \rho_\e^{(i)} \rVert^2_{L^{\infty}(\bt^d)} \, W^2_2(\rho_\e^{(1)}, \rho_\e^{(2)}) \leq C \e^{-4} M^2 D .
\end{align}
\end{paragraph}

\begin{paragraph}{Control of $I_3$:}
To estimate $I_3$, we recall a regularity estimate on $\widehat U_\e^{(1)}$ from Proposition~\ref{prop:regU}:
\be
\| \widehat E^{(1)}_\e \|_{C^1(\bt^d)} \leq \| \widehat U^{(1)}_\e \|_{C^2(\bt^d)} \leq C_{d} \,\mathbf{\overline{\exp}_2} {\Bigl(C_{d} \e^{-2}  \Bigl(1 +  \lVert \rho_\e^{(1)} \rVert_{L^{\frac{d+2}{d}}(\bt^d)} \Bigr) \Bigr)} .
\ee
Under conditon \eqref{eq:hypothesis f}, by Lemma~\ref{lem:rho-Lp},
\be
\lVert \rho_\e^{(1)} \rVert_{L^{\frac{d+2}{d}}(\bt^d)} \leq C
\ee
for some $C$ depending on $C_0$ only. Therefore
\begin{align}
I_3  \leq \|\widehat E_\e^{(1)}\|^2_{C^1(\bt^d)} \int_{(\bt^d \times \br^d)^2} |x_1 - x_2|^2 \di \pi_t  \leq \mathbf{\overline{\exp}_2}{(C \e^{-2})} D ,
\end{align}
for some $C$ depending on $C_0$ and $d$ only. If \eqref{eq:hypothesis f} does not hold, we can use the fact that
\be
\lVert \rho_\e^{(1)} \rVert_{L^{\frac{d+2}{d}}(\bt^d)} \leq \lVert \rho_\e^{(1)} \rVert_{L^{\infty}(\bt^d)} \leq M
\ee
and complete the proof in the same way, to find a constant depending on $M$.
\end{paragraph}

\begin{paragraph}{Control of $I_4$:}
First, note that
\begin{align}
I_4   = \int_{\bt^d} |\widehat E_\e^{(1)}(x) - \widehat E^{(2)}_\e(x)|^2 \rho_\e^{(2)}(\di x) \leq \|\rho_\e^{(2)}\|_{L^\infty(\bt^d)} \|\widehat E_\e^{(1)} - \widehat E^{(2)}_\e\|_{L^2(\bt^d)}^2 .
\end{align}
We apply the stability estimate on $\widehat U_\e^{(1)}$ from Proposition~\ref{prop:Ustab} to find
\begin{align}
\lVert \widehat E_\e^{(1)} -  \widehat E_\e^{(2)} \rVert^2_{L^2(\bt^d)} &\leq \mathbf{\overline{\exp}_2} {\left [ C_d \e^{-2} \left ( 1 + \max_i\, \lVert \rho_\e^{(i)} \rVert_{L^{(d+2)/d}(\bt^d)}  \right ) \right ]}\\
&\times \max_i\, \lVert  \rho_\e^{(i)} \rVert_{L^{\infty}(\bt^d)} \, W^2_2( \rho_\e^{(1)},  \rho_\e^{(2)}).
\end{align}
Once again, if conditon \eqref{eq:hypothesis f} holds then
\be
\lVert \widehat E_\e^{(1)} -  \widehat E_\e^{(2)} \rVert^2_{L^2(\bt^d)} \leq \mathbf{\overline{\exp}_2} {\left ( C \e^{-2}\right )} \max_i\, \lVert  \rho_\e^{(i)} \rVert_{L^{\infty}(\bt^d)} \, W^2_2( \rho_\e^{(1)},  \rho_\e^{(2)}),
\ee
for some $C$ depending on $C_0$ and $d$ only. Thus
\be
I_4 \leq  \mathbf{\overline{\exp}_2} {\left ( C \e^{-2}\right )} M^2 D.
\ee

\end{paragraph}

Altogether we find that
\begin{align}
\dot D \leq \begin{cases} C_\e D \left (1 + \lvert \log{\frac{D}{16d}} \rvert \right ) & \text{ if } D < d \\
C_\e (1 + \log{16}) D & \text{ if } D \geq d .
\end{cases}
\end{align}
If \eqref{eq:hypothesis f} holds, then $C_\e$ may be chosen to satisfy
\begin{align}
C_\e &\leq C \e^{-2} M +  \mathbf{\overline{\exp}_2}{(C \e^{-2})} + \mathbf{\overline{\exp}_2}{(C \e^{-2})} M \\
&\leq  \,\mathbf{\overline{\exp}_2}{(C \e^{-2})} (M + 1) .
\end{align}

We conclude that
\be
D(t) \leq C \exp{\left [ \left ( 1 + \log{\frac{D(0)}{16d}} \right ) e^{- C_\e t} \right ] } 
\ee
as long as $D \leq d$; if $D > d$ then
\be
D(t) \leq (d \vee D(0)) e^{C_\e(1 + \log{16})t} .
\ee

\end{proof}

\section{Growth estimates} \label{sec:growth}
In this section we will study how the support of a solution of the VPME system grows in time. This argument is very much dimension dependent and we shall present two different proofs for the two and three dimensional case, respectively in Sections \ref{sec:growth2d} and \ref{sec:growth3d}.
The result we obtain is the following:

\begin{prop}[Mass bounds] \label{prop:growth}
Let $f_\e$ be a solution of \eqref{vpme-quasi} satisfying for some constant $C_0$,
\be
\lVert f_\e \rVert_{L^{\infty}([0,T] \times \bt^d \times \br^d)} \leq C_0,\qquad
\sup_{t \in [0,T]} \mc{E}_\e[f_\e(t)] \leq C_0 .
\ee
Assume that $\rho_{f_\e} \in L^\infty([0,T] ; L^\infty(\bt^d))$. Let $R_0$ be any constant such that the support of $f_\e(0, \cdot, \cdot)$ is contained in $\bt^d \times B_{\br^d}(0,R_0)$.

\begin{enumerate}[(i)]
\item If $d=2$, then:
$$
\sup_{t \in [0,T]} \lVert \rho_{f_\e}(t) \rVert_{L^{\infty}(\bt^2)} \leq C_T\, e^{C \e^{-2}} \left [ R_0 + \e^{-2}\right ]^2 .
$$
The constant $C$ depends on $C_0$ only, while $C_T$ depends on $C_0$ and $T$.
\item If $d=3$, then:
$$
\sup_{t \in [0,T]} \lVert \rho_\e(t) \rVert_{L^{\infty}(\bt^3)} \leq \max \{ T^{-81/8}, C(R_0^3 + e^{C \e^{-2}} T^6) \} .
$$
The constant $C$ depends on $C_0$ only.
\end{enumerate}
\end{prop}

\subsection{Proof of Proposition \ref{prop:growth} in the two dimensional case}
\label{sec:growth2d}
In this section, we prove an estimate on the mass density $\rho_\e$ in the case $d=2$. Observe first that if the support of $f_\e$ is contained in the set $\bt^2 \times B_{\br^2}(0 ; R_t)$, then
\be \label{rho-supp}
\| \rho_\e(t) \|_{L^\infty(\bt^d)} \leq C R_t^2 ,
\ee
where $C$ is a constant depending on $\| f_\e(0) \|_{L^\infty(\bt^2 \times \br^2)}$. Our argument will rely on controlling the growth of the support of $f_\e$. To do this we will find a bound on the growth rate of the velocity component of the characteristic trajectories. Since any characteristic trajectory $(X_t, V_t)$ satisfies
\be
\dot V_t = E_\e(X_t) ,
\ee
we look for a uniform estimate on the electric field $E_\e$.

By Proposition~\ref{prop:regU}, for the smooth part $\widehat E_\e$ we have the estimate
\begin{align}
\| \widehat E_\e \|_{L^\infty(\bt^2)} & \leq \exp{\left \{ C \e^{-2} \left (1 + \| \rho_\e \|_{L^2(\bt^2)} \right ) \right \}} \\
& \leq \exp{\left ( C \e^{-2} \right )} ,
\end{align}
where the constant $C$ depends only on a bound on the initial energy. However, for the singular part $\bar E_\e$, Proposition~\ref{prop:regU} only gives us the estimate
\be
\| \bar E_\e \|_{L^\infty(\bt^2)} \leq C \e^{-2} \left ( 1 + \| \rho_\e \|_{L^\infty(\bt^d)} \right ),
\ee
which depends on an $L^\infty$ bound on the mass density. If we use this estimate in combination with \eqref{rho-supp}, this results in a bound on the size of the support of the form
\be \label{diffineq-attempt}
R_t \leq R_0 + \exp{\left ( C \e^{-2} \right )} t + C \e^{-2} \int_0^t R_s^2 \di s .
\ee
The solution of the ODE
\be
\dot y = C(1 + y^2)
\ee
blows up in finite time and so the differential inequality \eqref{diffineq-attempt} is not enough to imply a bound on $R_t$. We need to use a more careful estimate on $\bar E_\e$. In dimension two, we can make use of the fact that the conservation of energy gives us a uniform bound on $\| \rho_\e \|_{L^2(\bt^2)}$, by Lemma~\ref{lem:rho-Lp}, and use an interpolation argument.

\begin{lem}
Let $\rho \in L^1 \cap L^\infty(\bt^2)$ satisfy the bounds
\be
\| \rho \|_{L^1(\bt^2)} = 1, \qquad
\| \rho \|_{L^2(\bt^2)} \leq C_0, \qquad
\| \rho \|_{L^\infty(\bt^2)} \leq M .
\ee
Let $\bar E_\e = - \nabla \bar U_\e$, where $\bar U_\e$ is the unique $W^{1,2}(\bt^2)$ solution of the Poisson equation
\be
\e^2 \Delta \bar U_\e = 1 - \rho .
\ee
Then there exists a constant $C$ depending only on $C_0$ such that
\be
\| \bar E _\e \|_{L^\infty(\bt^2)} \leq C \e^{-2} \left ( 1 +  |\log M |^{1/2}  \right ) .
\ee
\end{lem}

\begin{proof}
We use the representation
\begin{align}
\bar E_\e = K \ast (\rho - 1) = C \e^{-2} \int_{\bt^d} \frac{x-y}{|x-y|^2} \left ( \rho(y) - 1 \right ) \di y + \e^{-2} K_0 \ast (\rho - 1) ,
\end{align} 
where $K_0$ is a $C^1(\bt^2)$ function. By Young's inequality,
\begin{align}
\| K_0 \ast (\rho - 1) \|_{L^\infty(\bt^2)}  \leq \| K_0 \|_{L^\infty(\bt^2)} \| \rho - 1 \|_{L^1(\bt^2)} \leq C .
\end{align}
We split the integral term into a part where $|x-y|$ is small and a part where $|x-y|$ is large:
\be
 \int_{\bt^d} \frac{x-y}{|x-y|^2} \left ( \rho(y) - 1 \right ) \di y =  \int_{|x-y| \leq l} \frac{x-y}{|x-y|^2} \left ( \rho(y) - 1 \right ) \di y + \int_{|x-y| \geq l} \frac{x-y}{|x-y|^2} \left ( \rho(y) - 1 \right ) \di y .
\ee

For the part where $|x-y|$ is large, we use Young's inequality with the $L^2$ control on $\rho$:
\begin{align}
\left \lvert \int_{|x-y| \geq l} \frac{x-y}{|x-y|^2} \left ( \rho(y) - 1 \right ) \di y \right \rvert & \leq  \int_{|x-y| \geq l} \frac{1}{|x-y|} \left | \rho(y) - 1 \right | \di y\\
& \leq \left \| \frac{1}{|x|} \mathbbm{1}_{|x| \geq l} \ast |\rho - 1| \right \|_{L^\infty(\bt^2)} \\
& \leq \left ( \int_{|x| \geq l} \frac{1}{|x|^2} \di x \right )^{1/2} \| \rho - 1\|_{L^2(\bt^2)} \\
& \leq C (C_0 + 1) \, | \log{l} | ^{1/2} .
\end{align}

Where $|x-y|$ is small, we use Young's inequality with the $L^\infty$ control on $\rho$:
\begin{align}
\left \lvert \int_{|x-y| \leq l} \frac{x-y}{|x-y|^2} \left ( \rho(y) - 1 \right ) \di y \right \rvert & \leq  \int_{|x-y| \leq l} \frac{1}{|x-y|} \left | \rho(y) - 1 \right | \di y\\
& \leq \left \| \frac{1}{|x|} \mathbbm{1}_{|x| \leq l} \ast |\rho - 1| \right \|_{L^\infty(\bt^2)} \\
& \leq \| \rho - 1\|_{L^\infty(\bt^2)} \int_{|x| \leq l} \frac{1}{|x|} \di x   \\
& \leq C M \, l .
\end{align}
Altogether we obtain
\be
\| \bar E _\e \|_{L^\infty(\bt^2)} \leq C \e^{-2} \left [ 1 + C_0 |\log l |^{1/2} + M l \right ] .
\ee
We choose $l = M^{-1}$ and conclude that
\be
\| \bar E _\e \|_{L^\infty(\bt^2)} \leq C \e^{-2} \left ( 1 +  |\log M |^{1/2}  \right ) ,
\ee
where $C$ depends on $C_0$ only.

\end{proof}

By using this estimate, we can deduce a differential inequality on $R_t$ that can be closed.

\begin{lem} \label{lem:supp-growth-2d}
Let $f_\e$ be a solution of \eqref{vpme-quasi} with bounded energy and compact support contained in $\bt^2 \times B_{\br^2}(0 ; R_t)$ at time $t$. Then $R$ satisfies the estimate
\be
R_t \leq e^{C \e^{-2}} t \left ( 1 + R_0 + (\log{t} \vee 0 )  \right ) .
\ee
\end{lem}

\begin{proof}

We consider the velocity coordinate $V_t(x,v)$ of an arbitrary characteristic trajectory starting from $(x,v)$ at time $t=0$. We have
\begin{align}
|V_t(x,v)| & \leq |v| + \int_0^t \| E_\e \|_{L^\infty(\bt^2)} \di s\\
& \leq |v| + \int_0^t \| \widehat E_\e \|_{L^\infty(\bt^2)} + \| \bar E_\e \|_{L^\infty(\bt^2)} \di s \\
& \leq |v| + \int_0^t \exp{(C \e^{-2})} + C \e^{-2} \left ( 1 +  |\log R_s |^{1/2}  \right ) \di s
\end{align}

The size of the support is controlled by the modulus of the furthest-reaching characteristic trajectory that starts within the support of $f_\e(0)$:
\begin{align}
R_t & \leq \sup_{(x,v) \in \bt^2 \times B_{\br^2}(0 ; R_0)} |V_t(x,v)| \\
& \leq \sup_{(x,v) \in \bt^2 \times B_{\br^2}(0 ; R_0)} \left \{ |v| + \int_0^t \exp{(C \e^{-2})} + C \e^{-2} \left ( 1 +  |\log R_s |^{1/2}  \right ) \di s \right \} \\
& \leq R_0 + \int_0^t \exp{(C \e^{-2})} + C \e^{-2} |\log R_s |^{1/2} \di s .
\end{align}

We compare this with the function
$$
z(t) = (1 + 2C  t) \left [ R_0 + \log{(1 + 2C  t)} \right ] .
$$
By \cite[Lemma A.1]{IGP-WP}, this satisfies
$$
\dot z \geq C  (1 + \log{(1+z)}) .
$$
We deduce that
\be
R_t \leq e^{C \e^{-2}} t \left ( \e^{-2} + R_0 + (\log{t} \vee 0 )  \right ) .
\ee

\end{proof}

\begin{proof}[Proof of Proposition~\ref{prop:growth}, case $d=2$]
We combine Lemma~\ref{lem:supp-growth-2d} with the elementary estimate \eqref{rho-supp}: for all $t \in [0,T]$,
\begin{align}
\| \rho_\e(t) \|_{L^\infty(\bt^d)} & \leq C R_t^2 \\
& \leq e^{C \e^{-2}} t \left ( \e^{-2} + R_0 + (\log{t} \vee 0 )  \right )^2 \\
& \leq C_T e^{C \e^{-2}} \left ( R_0 + \e^{-2}   \right )^2 .
\end{align}

\end{proof}

\subsection{Proof of Proposition \ref{prop:growth} in the three dimensional case}
\label{sec:growth3d}
In this section, we prove a mass bound in the case $d=3$. In this case, the conservation of energy gives us a uniform bound on $\| \rho_\e \|_{L^{5/3}(\bt^3)}$. This integrability is not enough to allow us to use the elementary interpolation approach that we used in the the two dimensional case. Instead, we will adapt estimates devised by Batt and Rein \cite{BR} for the classical Vlasov-Poisson equation on $\bt^3 \times \br^3$. We used this approach in \cite{IGP-WP} to prove the existence of solutions of the VPME system with bounded density. Here we focus on identifying how the bounds on $\| \rho_\e \|_{L^\infty(\bt^3)}$ depend on $\e$.

As in the two dimensional case, Batt and Rein's argument relies on controlling the mass density using the characteristic trajectory with velocity component of greatest Euclidean norm starting within the support of $f_\e(0)$ at time zero. They prove a bootstrap estimate on the convolution integral defining the singular part of the electric field. We recall this estimate in the following technical lemma.

\begin{lem}
\label{lem:BR}
Let $(X(t; s, x, v), V(t; s,x, v))$ denote the solution at time $t$ of an ODE
\be
\begin{pmatrix} \dot X(t) \\ \dot V(t) \end{pmatrix} = a(X(t), V(t)) , \qquad
\begin{pmatrix} X(s) \\ V(s) \end{pmatrix} = \begin{pmatrix} x \\ v \end{pmatrix},
\ee
where $a$ is of the form
$$
a(X, V) = \begin{pmatrix} V \\ a_2(X,V) \end{pmatrix} .
$$
Assume that $f = f(t,x,v)$ is the pushforward of $f_0$ along the associated characteristic flow; that is, for all $\phi \in C_b(\bt^3 \times \br^3)$,
$$
\int_{\bt^3 \times \br^3} f(t,x,v) \phi(x,v) \di x \di v = \int_{\bt^3 \times \br^3} f(s,x,v) \phi (X(t; s, x, v), V(t; s,x, v)) \di x \di v,
$$
and that $f$ is bounded with a uniformly bounded second moment in velocity:
\be
\lVert f \rVert_{L^{\infty}([0,T] \times \bt^3 \times \br^3 )}  \leq C, \qquad\sup_{t \in [0,T]} \lVert f_t |v|^2 \rVert_{L^{1}(\bt^3 \times \br^3)}  \leq C .
\ee
Define the quantities $h_\rho, h_\eta$ by
\begin{align} \label{def:h_rho-eta}
&h_{\rho}(t)  : = \sup \{ \lVert \rho_f(s) \rVert_{L^{\infty}(\bt^3)} ; 0 \leq s \leq t \} \\
& h_\eta (t, \Delta)  : = \sup \{ |V(t_1, \tau, x, v) - V(t_2, \tau, x, v)| ; \ 0 \leq t_1, t_2, \tau \leq t ,\  |t_1 - t_2| \leq \Delta, \ (x,v) \in \bt^3 \times \br^3 \} .
\end{align}
Assume that there exists $C_*>1$ such that
\be \label{h-beta}
h_\eta(t, \Delta) \leq C_* h_\rho(t)^\beta \Delta.
\ee
Then for all $t_1 < t_2 \leq t$, if
$$
h_\rho(t)^{-\beta/2} \leq \Delta \leq t 
$$
then
 \be
\int_{t_1}^{t_2} \int_{\bt^3} |X(s)-y|^{-2} \rho_f(s, y) \di y \di s \leq C\,  {C_*}^{4/3} \left (  h_\rho(t)^{2 \beta /3} + h_\rho(t)^{1/6} \right )\Delta .
\ee
 \end{lem}
 
 We complete the proof of Proposition~\ref{prop:growth} by combining Lemma~\ref{lem:BR} with the estimates on $\widehat E_\e$ from Proposition~\ref{prop:regU}.
 
 \begin{proof}[Proof of Proposition~\ref{prop:growth}, case $d=3$]
 
 For any characteristic trajectory $(X_t, V_t)$, we have for any $0 \leq t_1 < t_2 \leq T$,
 \be
\left \lvert V_{t_1} - V_{t_2} \right \rvert \leq \int_{t_2}^{t_1} |E_\e (X_s)| \di s .
 \ee
 We can write the total force $E_\e$ in the form
 \be \label{total-force}
 E_\e(x) = \e^{-2} [K_0 \ast (\rho_\e - 1)] (x) + C \e^{-2} \int_{\bt^3} \frac{x-y}{|x-y|^3} \rho_\e(y) \di y + \widehat E_\e .
 \ee
 Since $K_0$ is a $C^1(\bt^3)$ function and $\rho_\e$ has unit mass, the first term satisfies the bound
 \begin{align}
 | \e^{-2} [K_0 \ast (\rho_\e - 1)] (x) | &\leq \e^{-2} \| K_0 \|_{L^\infty(\bt^3)} \| \rho_\e - 1 \|_{L^1(\bt^3)} \\
 & \leq C \e^{-2} .
 \end{align}
 For the last term, we use Proposition~\ref{prop:regU}:
 \begin{align}
 |\widehat E_\e|  \leq \exp{\left [ C \left ( 1 + \| \rho_\e \|_{L^{5/3}(\bt^3)} \right )  \right ]} \leq \exp{(C \e^{-2})} .
 \end{align}
 Therefore,
  \be \label{h-eta-initial}
\left \lvert V_{t_1} - V_{t_2} \right \rvert \leq \int_{t_2}^{t_1} \left [ \exp{(C \e^{-2})} + C \e^{-2} \int_{\bt^3} |X(s) - y|^{-2} \rho_\e(y) \di y \right ]\di s .
 \ee
 From \cite{BR} we have the estimate  \be\label{eq:4/9}
\int_{\bt^3} |x-y|^{-2} \rho_\e(s,y) \di y \leq C\, \lVert \rho_\e(s,\cdot) \rVert_{L^{5/3}}^{5/9} \lVert \rho_\e(s,\cdot) \rVert_{L^\infty(\bt^3)}^{4/9} \leq C\, \lVert \rho_\e(s,\cdot) \rVert_{L^\infty(\bt^3)}^{4/9},
\ee
where $C$ depends only on $\| f_\e(0)\|_{L^\infty(\bt^3 \times \br^3)}$ and the initial energy. Alternatively, by an elementary argument we can prove the estimate
\be
\int_{\bt^3} |x-y|^{-2} \rho_\e(s,y) \di y \leq C\, \lVert \rho_\e(s,\cdot) \rVert_{L^\infty(\bt^3)}^{4/9 + \eta},
\ee
 for any $\eta > 0$; see \cite[Remark 5.7]{IGP-WP}. The choice of exponent does not affect the number of iterations used in the next part of the proof.
 
 By \eqref{h-eta-initial}, we have
\be \label{initialise}
h_\eta(t, \Delta) \leq  \left ( C \e^{-2} h_\rho(t)^{\frac49} + e^{C \e^{-2}}  \right ) \Delta .
\ee
Since $\rho_\e$ has total mass 1, $h_\rho \geq 1$. Thus
\be \label{initialise-2}
h_\eta(t, \Delta) \leq  e^{C \e^{-2}} \Delta h_\rho(t)^{\frac49} .
\ee
This means that condition \eqref{h-beta} is satisfied with $C_* = e^{C \e^{-2}}$. We apply Lemma~\ref{lem:BR} to improve our bound on 
\be
\int_{\bt^3} |X(s)-y|^{-2} \rho_\e(y) \di y \leq C \, C_*^{4/3}  \left ( h_\rho(t)^{\frac23 \cdot \frac49} + h_\rho(t)^{1/6} \right ) \Delta .
\ee
Feeding this new estimate into \eqref{h-eta-initial}, we obtain
\be \label{bootstrap}
h_\eta(t, \Delta) \leq C \e^{-2} \, C_*^{4/3}  \left ( h_\rho(t)^{\frac23 \cdot \frac49} + h_\rho(t)^{1/6} \right ) \Delta + ( C\e^{-2} + C e^{C \e^{-2}} ) \Delta\leq e^{C \e^{-2}} h_\rho(t)^{\frac{8}{27}}\Delta ,
\ee
as long as
$$
h_\rho(t)^{-2/9} \leq \Delta \leq t.
$$
 
 We will iterate this process until we achieve the lowest possible exponent for $h_\rho$, i.e. $\frac16$. Applying Lemma~\ref{lem:BR} a second time, we obtain
 \be
\int_{\bt^3} |X(s)-y|^{-2} \rho_\e(y) \di y \leq C \, C_*^{4/3}  \left ( h_\rho(t)^{\frac23 \cdot \frac{8}{27}} + h_\rho(t)^{1/6} \right ) \Delta ,
\ee
with $C_* = e^{C \e^{-2}}$, provided that
$$
h_\rho(t)^{-4/27} \leq \Delta \leq t ,
$$
and therefore
\be \label{bootstrap-2}
h_\eta(t, \Delta) \leq C \e^{-2} \, C_*^{4/3}  \left ( h_\rho(t)^{\frac23 \cdot \frac{8}{27}} + h_\rho(t)^{1/6} \right ) \Delta + ( C\e^{-2} + C e^{C \e^{-2}} ) \Delta\leq e^{C \e^{-2}} h_\rho(t)^{\frac{16}{81}}\Delta .
\ee

Applying Lemma~\ref{lem:BR} once more, we obtain 
 \be
\int_{\bt^3} |X(s)-y|^{-2} \rho_\e(y) \di y \leq C \, C_*^{4/3}  \left ( h_\rho(t)^{\frac23 \cdot \frac{16}{81}} + h_\rho(t)^{1/6} \right ) \Delta ,
\ee
with $C_* = e^{C \e^{-2}}$, provided that
$$
h_\rho(t)^{-8/81} \leq \Delta \leq t ,
$$
and therefore
\be \label{bootstrap-3}
h_\eta(t, \Delta) \leq C \e^{-2} \, C_*^{4/3}  \left ( h_\rho(t)^{\frac23 \cdot \frac{16}{81}} + h_\rho(t)^{1/6} \right ) \Delta + ( C\e^{-2} + C e^{C \e^{-2}} ) \Delta\leq e^{C \e^{-2}} h_\rho(t)^{\frac{1}{6}}\Delta ,
\ee
since $\frac{32}{243} < \frac{1}{6}$ and $h_\rho \geq 1$.

Finally, we use this new growth estimate on characteristic trajectories to control the mass density. Assuming that $f_\e(0)$ is supported in $\bt^3 \times B_{\br^3}(0 ; R_t)$, we have
\be
h_\rho \leq C \| f_\e \|_{L^\infty(\bt^d \times \br^d)} \left ( R_0 + h_\eta(t,t) \right )^3
\ee
Since we work with $L^\infty(\bt^d \times \br^d)$ solutions, we have a uniform bound on $\| f_\e \|_{L^\infty(\bt^d \times \br^d)}$ depending only on the initial data. Therefore, using \eqref{bootstrap-3}, we find that if $h_\rho(t)^{-8/81} \leq t$,
\begin{align}
h_\rho &\leq C \left ( R_0 + h_\eta(t,t) \right )^3 \leq C \left ( R_0 + e^{C \e^{-2}} h_\rho(t)^{\frac{1}{6}} t \right )^3 \\
& \leq CR_0^3+ e^{C \e^{-2}}  h_\rho(t)^{1/2} t^3 \leq CR_0^3+ \frac{\bigl(e^{C \e^{-2}}  t^3\bigr)^2+ h_\rho(t)}{2} .
\end{align}
Hence
\be
h_\rho \leq C \left ( R_0^3+ e^{C \e^{-2}}  t^6 \right ) .
\ee
If instead $h_\rho(t)^{-8/81} \geq t$, then
\be
h_\rho(t) \leq t^{- \frac{81}{8}} .
\ee
Therefore, we may conclude that
$$
h_\rho(t) \leq \max \{ t^{-81/8}, C(R_0^3 + e^{C \e^{-2}} t^6) \} .
$$

 \end{proof}

\section{Quasineutral limit: proof of Theorem \ref{thm:quasi}} \label{sec:QN}

In this section we will prove Theorem~\ref{thm:quasi}. The main idea is to consider the unique bounded density solution $g_\e$ of the VPME system \eqref{vpme-quasi} with initial datum $g_\e(0)$. Since $g_\e(0)$ is compactly supported, a bounded density solution exists by \cite[Theorem 2.4]{IGP-WP} and is unique by \cite[Theorem 2.3]{IGP-WP}, so $g_\e$ is well defined. We will use $g_\e$ as a stepping stone between $f_\e$ and a solution $g$ of the KIE system \eqref{KIE}.

\begin{proof}[Proof of Theorem~\ref{thm:quasi}]
Let $g_\e$ denote the solution of \eqref{vpme-quasi} with data $g_\e(0)$. We will use $g_\e$ to interpolate between the solution $f_\e$ of \eqref{vpme-quasi} starting from $f_\e(0)$ and the solution $g$ of \eqref{KIE} starting from $g(0)$. By the triangle inequality,
\be \label{triangle}
W_1 (f_\e(t), g(t)) \leq W_1 (f_\e(t), g_\e(t)) + W_1 (g_\e(t), g(t)).
\ee

The quasineutral limit for the VPME system \eqref{vpme-quasi} with uniformly analytic data can be proved using the methods of Grenier \cite{Grenier96}, with the modifications for the massless electrons case described in \cite[Proposition 4.1]{IHK1}. Grenier's result gives an $H^s$ convergence of a representation of the VP system as a multi-fluid pressureless Euler system. In \cite[Corollary 4.2]{IHK1}, it is shown how this implies convergence in $W_1$.
Since the initial data $g_\e(0)$ satisfy Grenier's assumptions,
there exists a solution $g$ of \eqref{KIE} on a time interval $[0, T_*]$ with initial data $g(0)$ such that
\be \label{eq:gre}
\lim_{\e \to 0} \sup_{t \in [0,T_*]} W_1 (g_\e(t), g(t)) = 0.
\ee

To deal with the first term of \eqref{triangle}, we use a stability estimate around $g_\e$ for the VPME system. By Proposition~\ref{prop:Wstab},
\begin{align} \label{str-str-stab}
W_1(f_\e(t), g_\e(t)) & \leq W_2(f_\e(t), g_\e(t))  \\
&  \leq \begin{cases}
C \exp{\left [ C\(1 + \log{\frac{W_2(f_\e(0), g_\e(0))}{4 \sqrt{d}}}\)\, e^{- C_\e t} \right ]} & \text{ if }\  W_2(f_\e(0), g_\e(0)) \leq d \\
W_2(f_\e(0), g_\e(0))\, e^{C_\e t} & \text{ if }\  W_2(f_\e(0), g_\e(0)) > d .
\end{cases}
\end{align}
where $C_\e$ may be chosen to satisfy
$$
C_\e \leq  \,\mathbf{\overline{\exp}_2}{(C \e^{-2})} (M + 1) .
$$
for any $M$ satisfying
\be \label{rho-Mbdd}
\sup_{[0,T_*]} \lVert \rho_{f_\e}(t) \rVert_{L^\infty_x}, \qquad \sup_{[0,T_*]} \lVert \rho_{g_\e}(t) \rVert_{L^\infty_x} \leq M .
\ee

By Proposition~\ref{prop:growth}, we may take $M$ such that
\be \label{M-epsctrl}
M \leq C\, e^{C \e^{-2}} .
\ee
The constant $C$ depends on $T_*$, $C_0$, the dimension $d$, and $C_1$, the rate of growth of the initial support. We emphasise again that the appearance of an exponential rate here is a consequence of the form of the equation rather than because condition \eqref{quasi:data-spt} allows fast growth of the initial support. It follows that we may estimate
\be \label{C-tripleexp}
C_\e\, t \leq C \,\mathbf{\overline{\exp}_2} (C \e^{-2}) 
\ee
for all $t \in [0,T_*]$. Hence we have convergence if 
\be
\frac{\left \lvert \log{\frac{W_2(f_\e(0), g_\e(0))}{4 \sqrt{d}}} \right \rvert}{\,\mathbf{\overline{\exp}_3}  (C \e^{-2})}  \to \infty
\ee
as $\e$ tends to zero. This holds if
\be
W_2(f_\e(0), g_\e(0)) \leq (\,\mathbf{\overline{\exp}_4} (C \e^{-2}))^{-1}
\ee
for sufficiently large $C$. In this case, it follows by \eqref{str-str-stab} that
$$
\lim_{\e \to 0} \sup_{t \in [0,T_*]}  W_1(f_\e(t), g_\e(t)) = 0 .
$$
Combined with \eqref{eq:gre} and \eqref{triangle}, this completes the proof.

\end{proof}

\section{Mean field limit: proof of Theorem \ref{thm:MFL}} \label{sec:MFL}

In this section we will prove Theorem~\ref{thm:MFL}. Recall that we want to show that the empirical measure $\mu^N_r$ corresponding to a solution of the regularised particle system \eqref{ODEreg} converges to a solution $f$ of \eqref{vpme}, in the Wasserstein sense:
\be
\sup_{t \in [0,T]} W_2(f(t), \mu^N_r(t)) \to 0 ,
\ee
provided that the initial data $\mu^N_r(0)$ converge sufficiently fast to $f(0)$. 

The key observation is that the particle system \eqref{ODEreg} is constructed such that for each $N$, $\mu^N_r$ is a weak solution of a regularised version of the VPME system:
\begin{equation}
\label{vpme-reg}
 \left\{ \begin{array}{ccc}\pt_t f_r+v\cdot \nabla_x f_r+ E_r \cdot \nabla_v f_r=0,  \\
E_r =- \chi_r \ast \nabla U_r, \\
\Delta U_r=e^{U_r} - \chi_r \ast \rho[f_r] ,\\
f_r\vert_{t=0}=f(0) \ge0,\ \  \int_{\bt^d \times \br^d} f(0) \,dx\,dv=1.
\end{array} \right.
\end{equation}
The convergence of $\mu^N_r$ to $f$ is therefore a kind of stability result, in two stages. 

First, as $N$ tends to infinity with $r$ fixed, $\mu^N_r$ converges to $f_r$, the solution of \eqref{vpme-reg} with initial datum $f(0)$. This holds because the force in \eqref{vpme-reg} is regular enough that the equation has a stability property even in the class of measures (as investigated for example by Dobrushin \cite{Dob}). Of course, the rate of this convergence will degenerate in the limit as $r$ tends to zero. Our goal is therefore to quantify this convergence in $W_2$, optimising the constants so as to minimise the rate of blow-up as $r$ tends to zero. This is the key difference between our approach, which is based on \cite{Laz}, and the approach of \cite{Dob}. We use a weak-strong stability estimate, which holds for initial data converging sufficiently quickly, to improve the dependence of the constants on $r$.

Secondly, in \cite{IGP-WP} we proved that, if $f(0)$ is bounded and compactly supported, then as $r$ tends to zero, $f_r$ converges to $f$, where $f$ is the unique bounded density solution of the VPME system \eqref{vpme} with initial datum $f(0)$. This observation was part of our construction of solutions to \eqref{vpme}. Our strategy will be to show that this convergence can be quantified in $W_2$, again aiming to optimise the rate.

By combining these two limits, we will identify a regime for the initial data in which $\mu^N_r$ converges to $f$.

\subsection{Behaviour of the $W_2$ distance under regularisation}

We recall some useful results on the behaviour of Wasserstein distances under regularisation by convolution. See \cite[Proposition 7.16]{Vil03} for proofs. Our first observation is that regularising two measures cannot increase the Wasserstein distance between them.

\begin{lem} \label{W_moll_two}
Let $\mu$, $\nu$ be probability measures, $r>0$ any positive constant and $\chi_r$ a mollifier as defined in \eqref{Def_chi}. Then
\be \label{Est_W_conv}
W_p(\chi_r * \mu, \chi_r * \nu) \leq W_p(\mu, \nu) .
\ee
\end{lem}

We also have an explicit control on the Wasserstein distance between a measure and its regularisation:
\begin{lem} \label{W_moll_one}
Let $\mu$ be a probability measure and $r> 0$. Let $\chi_r$ be a mollifier as defined in \eqref{Def_chi}. Then
\be \label{Est_W_self_conv}
W_p(\chi_r * \mu, \mu) \leq r .
\ee
\end{lem}

If a measure $\nu$ is close to an $L^p$-function in Wasserstein sense, then it is possible to estimate the $L^p$-norm of the regularised measure $\chi_r \ast \nu$ in a way that exploits this fact. We will use this in our estimates to control the regularised mass density $\chi_r \ast \rho_\mu$. The following estimate is shown for $p=\infty$ in \cite[Lemma 4.3]{Laz}, but it is straightforward to adapt it to the case of general $p$.

\begin{lem} 
\label{Lem_mu_moll} Let $\nu$ be a probability measure on $\bt^d$ and $h \in L^{\infty}(\bt^d)$ a probability density function. Then, for all $r > 0$, $p \in [1, \infty]$, $q \in [1, \infty)$,
\begin{equation} \label{Est_mu_moll}
\lVert \chi_r * \nu \rVert_{L^{p}(\bt^d)} \leq C_d \left ( \rVert h \lVert_{L^{p}(\bt^d)} + r^{-(q + d)} W_q^q(h, \nu) \right ).
\end{equation}
\end{lem}

\subsection{The regularised VPME system}

First, we must justify the existence of a unique solution for \eqref{vpme-reg}, under our assumptions on the initial datum. We recall a well-posedness result from \cite[Lemma 6.1]{IGP-WP}.

\begin{lem}[Existence of regularised solutions] \label{lem:exist-vpme-reg}
For every $f(0) \in \mc{P}(\bt^d \times \br^d)$, there exists a unique solution $f_r \in C([0,T] ; \mc{P}(\bt^d \times \br^d))$ of \eqref{vpme-reg}. If $f(0) \in L^p(\bt^d \times \br^d)$ for some $p \in [1, \infty]$, then for all $t \in [0,T]$
\be
\lVert f_r(t) \rVert_{L^p(\bt^d \times \br^d)} \leq \lVert f(0) \rVert_{L^p(\bt^d \times \br^d)} .
\ee
\end{lem}

An important point about the regularisation approach we have chosen is that \eqref{vpme-reg} has an associated conserved quantity. This is the regularised energy, defined by
\be \label{def:Ee-reg}
\mc{E}_r [f] : = \frac{1}{2} \int_{\bt^d \times \br^d} |v|^2 f \di x \di v + \frac{1}{2} \int_{\bt^d} \lvert \nabla U_r \rvert^2 \di x + \int_{\bt^d} U_r e^{U_r} \di x .
\ee
For solutions with bounded mass and finite initial energy, the energy is conserved: if $f_r$ is a solution of \eqref{vpme-reg} with finite initial energy and $\rho[f_r] \in L^\infty(0,T ; L^\infty(\bt^d))$, then for all $t \in [0,T]$,
\be
\mc{E}_r [f_r(t)] = \mc{E}_r [f(0)] . 
\ee

Note that $x e^x \geq - e^{-1}$ for all $x \in \br$. Hence if $f_r$ is a solution of \eqref{vpme-reg} with finite initial energy and $\rho[f_r] \in L^\infty([0,T] ; L^\infty(\bt^d))$, then for all $t \in [0,T]$,
\be
 \int_{\bt^d \times \br^d} |v|^2 f_r(t) \di x \di v \leq C \left ( \mc{E}_r [f(0)] + 1 \right ) .
\ee
Thus by Lemma~\ref{lem:rho-Lp},
\be \label{reg-rho-Lp}
\sup_{t \in [0,T]} \| \rho[f_r(t)] \|_{L^{\frac{d+2}{d}(\bt^d)}} \leq C \left ( \mc{E}_r [f(0)] + 1 \right ) .
\ee

Moreover, the growth estimates in Proposition~\ref{prop:growth} also apply to the regularised system. Hence, if the initial datum is compactly supported then the mass density is bounded on any compact time interval: $\rho \in L^\infty_{\text{loc}} ([0,\infty) ; L^\infty(\bt^d))$. In particular solutions of \eqref{vpme-reg} beginning from compactly supported data have conserved energy.

\subsection{$W_2$ stability for the regularised VPME system}

In this section we follow the methods of \cite{Laz} to prove a weak-strong stability estimate for the regularised equation \eqref{vpme-reg} in the Wasserstein distance $W_2$, optimised to degenerate slowly as $r$ tends to zero. This will allow us to use solutions of \eqref{vpme-reg} as a bridge between the particle system \eqref{ODEreg} and the VPME system \eqref{vpme}.

\begin{lem}[Weak-strong stability for the regularised equation] \label{lem:wkstr}
For each $r > 0$, let $f_r, \mu_r$ be solutions of \eqref{vpme-reg}, where the $f_r$ have uniformly bounded density and initial energy:
\begin{align} \label{hyp:rhobdd}
\sup_r \sup_{t \in [0,T]} \lVert \rho_{f_r} \rVert_{L^{\infty}(\bt^d)} \leq M, \\ \label{hyp:energy}
\sup_r \mc{E}_r[f_r(0)] \leq C_0,
\end{align}
for some $C_0,M>0$.
Assume that the initial data satisfy, for some sufficiently large constant $C>0$ depending on $T$, $C_0$, $M$ and $d$,
\be
\limsup_{r \to 0} \frac{W_2^2(f_r(0), \mu_r(0)) }{r^{(d + 2 + C |\log{r}|^{-1/2})}} < 1.
\ee
Then
\be
\lim_{r \to 0} \sup_{t \in [0,T]} W_2^2(f_r(t), \mu_r(t)) = 0.
\ee
\end{lem}

\begin{proof} To lighten the notation, we drop the subscript $r$ from $f_r$ and $\mu_r$. Fix an arbitrary coupling of the initial data $\pi_0 \in \Pi(\mu_0, f_0)$. As in the proof of Proposition~\ref{prop:Wstab}, we define a special coupling that follows the characteristic flow of \eqref{vpme-reg}.

Consider the following systems of ordinary differential equations:
\be  \label{MF-eqn:char-2}
\begin{array}{ll}
\begin{cases}\dot X^{(1)}_{x,v} = V^{(1)}_{x,v} \\
\dot V^{(1)}_{x,v}  = E^{(\mu)}_r  (X^{(1)}_{x,v})\\
(X^{(1)}_{x,v}(0), V^{(1)}_{x,v}(0)) = (x, v)\\
E^{(\mu)}_r = - \chi_r \ast \nabla U_r^{(\mu)}\\
\Delta U_r^{(\mu)} = e^{U_r^{(\mu)}} - \chi_r \ast \rho_\mu
\end{cases}
&\qquad\qquad \begin{cases} \dot X^{(2)}_{x,v} = V^{(2)}_{x,v} \\
 \dot V^{(2)}_{x,v}  = E^{(f)}_r (X^{(2)}_{x,v})\\
 (X^{(2)}_{x,v}(0), V^{(2)}_{x,v}(0)) = (x, v) \\
  E^{(f)}_r = - \chi_r \ast \nabla U_r^{(f)} \\
  \Delta U_r^{(f)}  = e^{U_r^{(f)}} - \chi_r \ast \rho_{f} .
 \end{cases} 
\end{array}
\ee

First, we check that unique global solutions exist for both systems. The same argument applies to both $\mu$ and $f$, so we will write it for $\mu$ only. Observe that since $\mu$ is a probability measure and $\chi_r$ is smooth, $\chi_r \ast \rho_\mu$ is a function with
\be
\| \chi_r \ast \rho_\mu \|_{L^\infty(\bt^d)} \leq \| \chi_r \|_{L^\infty(\bt^d)} .
\ee
Hence by the regularity estimates in Proposition~\ref{prop:regU}, $U_r^{(\mu)}$ is a $C^1$ function with
\be
\|U_r^{(\mu)}\|_{C^1(\bt^d)} \leq \exp{\left[ C\left (1 + \|\chi_r\|_{L^\infty(\bt^d)} \right ) \right]} .
\ee
Then $E_r^{(\mu)}= - \chi_r \ast \nabla U_r^{(\mu)} $ is a smooth function with bounded derivative
\be
\| E_r^{(\mu)} \|_{C^1(\bt^d)} \leq \|\chi_r\|_{C^1(\bt^d)} \exp{\left[ C\left (1 + \|\chi_r\|_{L^\infty(\bt^d)} \right ) \right]} .
\ee
Therefore there is a unique $C^1$ flow corresponding to \eqref{MF-eqn:char-2}.

As we argued in the proof of Proposition~\ref{prop:Wstab}, $\mu$ and $f$ can be represented as the pushforwards of $\mu(0)$ and $f(0)$ respectively along these characteristic flows. That is, for $\phi \in C_b(\bt^d \times \br^d)$ we have
\begin{align} \label{pushforward}
\int_{\bt^d \times \br^d} \phi \, \di \mu(t) ( x , v) &= \int_{\bt^d \times \br^d} \phi \left ( X^{(1)}_{x,v}(t), V^{(1)}_{x,v}(t) \right )  \di \mu(0) ( x , v), \\
\int_{\bt^d \times \br^d} \phi \, f(t) \di x \di v &= \int_{\bt^d \times \br^d} \phi \left ( X^{(2)}_{x,v}(t), V^{(2)}_{x,v}(t) \right )  f(0) \di x \di v .
\end{align}

We define $\pi_t$ to be the measure on $\left (\bt^d \times \br^d \right )^2$ such that for all $\phi \in C_b\left[\left (\bt^d \times \br^d \right )^2\right]$
\be \label{def:pit}
\int_{(\bt^d \times \br^d)^2} \phi(x_1,v_1,x_2,v_2) \di \pi_t = \int_{\bt^d \times \br^d} \phi \left (X^{(1)}_{x_1,v_1}, V^{(1)}_{x_1,v_1}, X^{(2)}_{x_2,v_2}, V^{(2)}_{x_2,v_2} \right ) \di \pi_0 .
\ee
Then $\pi_t$ is a coupling of $\mu(t)$ and $f(t)$. We can see this by using a calculation similar to \eqref{pit-marg} to check the marginals.

Using $\pi_t$, we define an anisotropic functional $D$ which controls the squared Wasserstein distance $W^2_2(f_t, \mu_t)$. For $\lambda > 0$, let
\be \label{def:Dl}
D(t) = \int_{(\bt^d \times \br^d)^2} \lambda^2 |x_1-x_2|^2 + |v_1-v_2|^2 \di \pi_t(x_1, v_1, x_2, v_2) .
\ee
We will choose $\lambda$ later in order to optimise our eventual rate. We note some relationships between $D$ and the Wasserstein distance. By Definition~\ref{def:Wass}, since $\pi_t$ is a particular coupling of $\mu_t$ and $f_t$, as long as we choose $\lambda^2 > 1$, we have
\be \label{Dl-ctrl-W}
W_2^2(\mu_t, f_t) \leq D (t) .
\ee
If we only look at the spatial variables, we can get a sharper estimate:
\be \label{Dl-ctrl-W-rho}
W_2^2(\rho_{\mu}(t), \rho_{f}(t)) \leq  \lambda^{-2} D (t).
\ee
Since $\pi_0 \in \Pi(\mu_0, f_0)$ was arbitrary we may take the infimum to obtain
\be \label{W-ctrl-Dl-init}
\inf_{\pi_0} D(0) \leq  \lambda^2 W_2^2(\mu_0, f_0) .
\ee

We now perform a Gr\"{o}nwall estimate on $D$. Taking a time derivative, we obtain
\be
\dot D = 2 \int_{(\bt^d \times \br^d)^2} \lambda^2 (x_1 - x_2) \cdot (v_1 - v_2 ) + (v_1 - v_2 ) \cdot \left ( E^{(\mu)}_r(x_1) - E^{(f)}_r (x_2) \right) \di \pi_t.
\ee
Using a weighted Cauchy inequality, we find that for any $\alpha > 0$,
\begin{multline}
\dot D = \lambda \int_{(\bt^d \times \br^d)^2} \lambda^2 |x_1 - x_2|^2 + |v_1 - v_2 |^2 \di \pi_t \\
+ \alpha \int_{(\bt^d \times \br^d)^2} |v_1 - v_2 |^2 \di \pi_t + \frac{1}{\alpha} \int_{(\bt^d \times \br^d)^2} \left | E^{(\mu)}_r(x_1) - E^{(f)}_r (x_2) \right |^2 \di \pi_t.
\end{multline}
Therefore
\be
\dot D \leq (\alpha + \lambda) D + \frac{C}{ \alpha} \sum_{i=1} I_i,
\ee
where
\be \label{def:Ii-MF}
\begin{split}
I_1 & := \int |\bar E_r^{(\mu)}(X^{(1)}_t) - \bar E^{(\mu)}_r(X^{(2)}_t)|^2 \di \pi_0 , \qquad I_2  := \int |\bar E_r^{(f)}(X^{(2)}_t) - \bar E^{(\mu)}_r(X^{(2)}_t)|^2 \di \pi_0, \\
I_3 & := \int |\widehat E_r^{(\mu)}(X^{(1)}_t) - \widehat E^{(\mu)}_r(X^{(2)}_t)|^2 \di \pi_0,
\qquad
I_4  := \int |\widehat E_r^{(f)}(X^{(2)}_t) - \widehat E^{(\mu)}_r(X^{(2)}_t)|^2 \di \pi_0 .
\end{split}
\ee
We have again used the decomposition $E_r^{(f)} = \bar E_r^{(f)} + \widehat E_r^{(f)}$, and the analogous form for $E^{(\mu)}_r$.

To estimate these quantities, we first note some basic regularity properties, which follow from Proposition~\ref{prop:regU}. First, we wish to control the regularised mass density $\chi_r \ast \rho_\mu$ with an estimate that behaves well as $r$ tends to zero. For this we use Lemma~\ref{Lem_mu_moll} with $q=2$ and \eqref{Dl-ctrl-W-rho}:
\begin{align} \label{Est-rho-mu}
\lVert \chi_r * \rho_\mu \rVert_{L^{p}(\bt^d)} & \leq C_d \left ( \rVert \rho_f \lVert_{L^{p}(\bt^d)} + r^{-(d+2)} W_2^2(\rho_\mu, \rho_f) \right )\\
& \leq C_d \left ( \rVert \rho_f \lVert_{L^{p}(\bt^d)} + r^{-(d+2)}  \lambda^{-2} D \right ) .
\end{align}
We will use this estimate in the cases $p = \frac{d+2}{d}$ and $p= \infty$. For $p=\infty$, by assumption \eqref{hyp:rhobdd} we obtain
\be \label{Est-rho-mu-infty}
\lVert \chi_r * \rho_\mu \rVert_{L^{\infty}(\bt^d)} \leq C_d \left ( M + r^{-(d+2)}  \lambda^{-2} D \right ) .
\ee
For $p= \frac{d+2}{d}$, by the initial energy assumption \eqref{hyp:energy} and \eqref{reg-rho-Lp} we have
\be
\| \rho_{f} \|_{L^\frac{d+2}{d}(\bt^d)} \leq C_d ,
\ee
where $C_d$ depends on $C_0$ and $d$. Thus
\be \label{Est-rho-mu-pd}
\lVert \chi_r * \rho_\mu \rVert_{L^{\frac{d+2}{d}}(\bt^d)} \leq C_d \left ( 1 + r^{-(d+2)}  \lambda^{-2} D \right ) 
\ee
for $C_d$ depending on $C_0$ and $d$.

We also wish to control the regularity of $\widehat U^{(\mu)}_r$. Using \eqref{Est-rho-mu-pd} and Proposition~\ref{prop:regU}, we obtain
\be \label{hatUmu-reg}
\| \widehat U^{(\mu)}_r \|_{C^2(\bt^d)} \leq C_d \,\mathbf{\overline{\exp}_2} {\left ( C_d \left ( 1 + r^{-(d+2)} \lambda^{-2} D \right )  \right )} .
\ee
We estimate $I_1$ and $I_2$ in the same way as in \cite{Laz}. For $I_1$ we use the regularity estimate
\begin{equation} \label{Lip_moll}
\lVert \chi_r * K * h \rVert_{\text{Lip} } \leq C | \log{r} | (1 + \lVert h \rVert_{L^{\infty}}) ;
\end{equation}
see \cite[Lemma 4.2(ii)]{Laz}. We combine this with the mass density estimate \eqref{Est-rho-mu-infty} to obtain
\be \label{MF-I1-est}
I_1 \leq C (\log{r})^2 \left ( M + r^{-(d+2)} \lambda^{-2} D \right )^2 \lambda^{-2} D .
\ee
For $I_2$, we use Proposition~\ref{prop:Ustab} and \eqref{Est-rho-mu-infty} to obtain 
\be \label{MF-I2-est}
I_2 \leq C M (M + r^{-(d+2)} \lambda^{-2} D) \lambda^{-2} D .
\ee
For $I_3$ we compute:
\begin{align} \label{MF-I3-est}
I_3 & =  \int |\widehat E_r^{(\mu)}(X^{(1)}_t) - \widehat E^{(\mu)}_r(X^{(2)}_t)|^2 \di \pi_0 = \int \lvert \chi_r \ast (\nabla \widehat U_r^{(\mu)}(x) - \nabla \widehat U_r^{(\mu)}(y)) \rvert^2 \di \pi_t \\
& \leq \int_{(\bt^d \times \br^d )^2} \lVert \chi_r \ast \nabla \widehat U_r^{(\mu)} \rVert^2_{\text{Lip}} |x-y|^2 \di \pi_t  \leq \int_{(\bt^d \times \br^d )^2} \lVert \widehat U_r^{(\mu)} \rVert^2_{C^{2}(\bt^d)} |x-y|^2 \di \pi_t .
\end{align}
We apply the regularity estimate \eqref{hatUmu-reg} to obtain
\begin{align}
I_3 & \leq C  \,\mathbf{\overline{\exp}_2} {\left [ C_{d}  (1 + r^{-(d + 2)} \lambda^{-2} D)  \right ]} \int_{(\bt^d \times \br^d )^2} |x-y|^2 \di \pi_t \\
& \leq C  \,\mathbf{\overline{\exp}_2} {\left [ C_{d}  (1 + r^{-(d + 2)} \lambda^{-2} D)  \right ]} \lambda^{-2} D .
\end{align}
For $I_4$ we compute
\begin{align}
I_4 & = \int |\widehat E_r^{(f)}(X^{(2)}_t) - \widehat E^{(\mu)}_r(X^{(2)}_t)|^2 \di \pi_0  = \int \lvert \chi_r \ast (\nabla \widehat U_r^{(f)}(X^{(2)}_t) - \nabla \widehat U_r^{(\mu)}(X^{(2)}_t)) \rvert^2 \di \pi_0 \\  \label{MF-I4-est}
& = \int_{\bt^d} \lvert \chi_r \ast (\nabla \widehat U_r^{(f)}(x) - \nabla \widehat U_r^{(\mu)}(x)) \rvert^2 \rho_f(x) \di x \leq \lVert \rho_f \rVert_{L^{\infty}(\bt^d)} \lVert \nabla \widehat U_r^{(f)} - \nabla \widehat U_r^{(\mu)} \rVert^2_{L^2(\bt^d)}.
\end{align}
By Proposition~\ref{prop:Ustab},
\begin{align}
\lVert \nabla \widehat U_r^{(f)} - \nabla \widehat U_r^{(\mu)} \rVert^2_{L^2(\bt^d)}
& \leq C  \,\mathbf{\overline{\exp}_2} {\left [ C_{d}(1 + r^{-(d + 2)} \lambda^{-2} D)  \right ]} \lVert \bar U_r^{(f)} - \bar U_r^{(\mu)} \rVert_{L^2(\bt^d)}^2 \\
& \leq C \,\mathbf{\overline{\exp}_2} {\left [ C_{d} (1 + r^{-(d + 2)} \lambda^{-2} D)  \right ]} (M + r^{-(d+2)} \lambda^{-2} D ) \lambda^{-2} D,
\end{align}
thus
\be
I_4 \leq C M \,\mathbf{\overline{\exp}_2} {\left [ C_{d} (1 + r^{-(d + 2)} \lambda^{-2} D)  \right ]} (M + r^{-(d+2)} \lambda^{-2} D ) \lambda^{-2} D .
\ee
We summarise all these bounds as 
\begin{align} \label{Gr-MF-nonlin}
\dot D & \leq (\lambda + \alpha)D + \frac{1}{4 \alpha} C (1 + | \log{r} |)^2  \left ( M + r^{-(d+2)} \lambda^{-2} D \right )^2 \lambda^{-2} D \\
& \qquad + C  \,\mathbf{\overline{\exp}_2} {\left [ C_{d} (1 + r^{-(d + 2)} \lambda^{-2} D)  \right ]} \left ( 1 + M ( M + r^{-(d+2)} \lambda^{-2} D ) \right ) \lambda^{-2} D .
\end{align}

Unfortunately, this is a nonlinear estimate and so cannot be closed in its current form. To deal with this, we introduce a truncated functional, rescaled to be of order 1:
\be \label{def:hatD1}
\hat D = 1 \wedge \left ( r^{-(d+2)} \lambda^{-2} D \right )
\ee
In terms of this new functional, \eqref{Gr-MF-nonlin} becomes
$$
\frac{\di }{ \di t} \hat D \leq (\lambda + \alpha) \hat D + \frac{1}{\alpha} C_d \left [ 1 + |\log{r}|^2 \right ] M^2 \lambda^{-2} \hat D .
$$
We then optimise the exponent by choosing
\be
\label{def:lstar}
\alpha_*  = C_d(1+ |\log{r}|) M \lambda^{-1}_* ,\qquad 
\lambda_*  =  C_d (1 + |\log{r}|)^{1/2} \sqrt{M} .
\ee
Then
$$
\frac{\di }{ \di t} \hat D \leq C_d \lambda_* \hat D .
$$
From this we deduce
\begin{align}
\sup_{[0,T]} \hat D(t) & \leq \exp(C_d \lambda_* T) \hat D(0)   \leq r^{-(d+2)} \lambda_*^{-2} \exp(C_d \lambda_* T) D(0) \\
& \leq C_d \exp{\left [ |\log{r}|\left ( (d+2)  + C_dT \sqrt{M} |\log{r}|^{-1/2} \right ) \right ]} \lambda_*^{-2}  D(0) .
\end{align}

In order to use this estimate to control the Wasserstein distance, we need to ensure that for $r$ sufficiently small,
\be \label{hatD-small}
 \inf_{\pi_0} \sup_{t \in [0,T]} \hat D(t) < 1 .
\ee
If \eqref{hatD-small} holds, then by \eqref{Dl-ctrl-W}, for all $t \in [0,T]$,
\be 
W_2^2(\mu_t, f_t)  \leq r^{d+2} \lambda_*^2  \inf_{\pi_0} \sup_{t \in [0,T]} \hat D(t)  \leq r^{d+2} \lambda_*^2 \to 0
\ee
as $r$ tends to zero, since $\lambda_*^2$ only grows like $|\log{r}|$ by definition \eqref{def:lstar}. 
Using \eqref{W-ctrl-Dl-init}, we obtain for any $\pi_0$,
\begin{align}
\sup_{[0,T]} \hat D(t) & \leq C_d \exp{\left [ |\log{r}|\left ( (d+2)  + C_d T \sqrt{M} |\log{r}|^{-1/2} \right )  \right ]} \lambda_*^{-2} \inf_{\pi_0} D(0) \\
& \leq C_d \exp{\left [ |\log{r}|\left ( (d+2)  + C_d T \sqrt{M} |\log{r}|^{-1/2} \right ) \right ]} W_2^2(\mu_0, f_0) .
\end{align}
Since we assumed that the initial data satisfy
\be
\limsup_{r \to 0} \frac{W_2^2(\mu_0, f_0) }{r^{d + 2 + C_{d,M,T}|\log{r}|^{-1/2}}} < 1,
\ee
for large $C_{d,M,T}$, \eqref{hatD-small} holds for sufficiently small $r$. This completes the proof.
\end{proof}

\subsection{Convergence to the original equation}

\begin{lem}[Approximation of (VPME)] \label{lem:RUR}
Fix $f(0) \in L^1 \cap L^\infty(\bt^d \times \br^d)$ satisfying a uniform bound on the energy as defined in \eqref{def:Ee-reg}: 
\be \label{energy-reg-bdd}
\sup_r \mc{E}_r[f(0)] \leq C_0,
\ee
for some $C_0>0$.

For each $r > 0$, let $f_r$ be a solution of \eqref{vpme-reg} with initial datum $f(0)$. Let $f$ be a solution of \eqref{vpme} with the same initial datum $f(0)$. Assume that $(f_r)_{r >0}, f$ have uniformly bounded density:
\be \label{dens-bound}
\sup_r \sup_{t \in [0,T]} \lVert \rho_{f_r} \rVert_{L^{\infty}(\bt^d)} \leq M , \qquad  \sup_{t \in [0,T]} \| \rho_f \|_{L^\infty(\bt^d)} \leq M 
\ee
for some $M>0$.
Then
\be
\lim_{r \to 0} \sup_{t \in [0,T]} W_2^2(f(t), f_r(t)) = 0.
\ee
\end{lem}

\begin{proof}

We fix an initial coupling $\pi_0 \in \Pi(f(0), f(0))$ and construct $\pi_t \in \Pi(f(t), f_r(t))$ as in \eqref{def:pit}, using the characteristic flows 
\be  \label{RUR:char-1}
\begin{array}{ll}
\begin{cases}\dot X^{(1)}_{x,v} = V^{(1)}_{x,v} \\
\dot V^{(1)}_{x,v}  = E  (X^{(1)}_{x,v}) \\
(X^{(1)}_{x,v}(0), V^{(1)}_{x,v}(0)) = (x, v)\\
E = - \nabla U\\
\Delta U = e^{U} -  \rho_f
\end{cases}
&\qquad\qquad 
\begin{cases}
\dot X^{(2)}_{x,v} = V^{(2)}_{x,v}\\
\dot V^{(2)}_{x,v}  = E_r (X^{(2)}_{x,v}) \\
(X^{(2)}_{x,v}(0), V^{(2)}_{x,v}(0)) = (x, v) \\
E_r = - \chi_r \ast \nabla U_r \\
\Delta U_r  = e^{U_r} - \chi_r \ast \rho_{f_r} .
\end{cases}
\end{array}
\ee

We define $D$ as in \eqref{def:Dl}. As in Lemma~\ref{lem:wkstr}, we obtain for any $\alpha > 0$,
\be
\dot D \leq (\alpha + \lambda) D + \frac{C}{ \alpha} \sum_{i=1}^4 I_i,
\ee
where
\be \label{def:Ii-MF2}
\begin{split}
I_1 & := \int |\bar E_r(X^{(1)}_t) - \bar E_r(X^{(2)}_t)|^2 \di \pi_0, \qquad I_2  := \int |\bar E_r(X^{(1)}_t) - \bar E(X^{(1)}_t)|^2 \di \pi_0, \\
I_3 & := \int |\widehat E_r(X^{(1)}_t) - \widehat E_r(X^{(2)}_t)|^2 \di \pi_0, \qquad I_4  := \int |\widehat E(X^{(1)}_t) - \widehat E_r(X^{(1)}_t)|^2 \di \pi_0 .
\end{split}
\ee
For $I_1$ we use the regularity estimate \eqref{Lip_moll} to deduce
\be  \label{RUR-I1-est}
I_1  \leq  |\log{r}|^2 M^2 \int |X^{(1)}_t - X^{(2)}_t|^2 \di \pi_0  \leq  |\log{r}|^2 M^2 \, \lambda^{-2} D .
\ee
For $I_2$ we use the stability estimate from Lemma~\ref{lem:Loep}:
\be 
\| \nabla \bar U - \nabla \bar U_r \|_{L^2(\bt^d)} \leq \sqrt{M} W_2(\chi_r \ast \rho_{f_r}, \rho_f) .
\ee
By Lemma~\ref{W_moll_one},
\be
W_2(\chi_r \ast \rho_{f_r}, \rho_f) \leq r + W_2(\rho_{f_r}, \rho_f).
\ee
Hence
\be  \label{RUR-I2-est}
\| \nabla \bar U - \nabla \bar U_r \|_{L^2(\bt^d)} \leq \sqrt{M} (r + \lambda^{-1} \sqrt{D}) .
\ee
We account for the extra regularisation by elementary methods: first, for $g \in C^1(\bt^d)$,
\begin{align}
\| g - \chi_r \ast g \|^2_{L^2(\bt^d)} & = \int_{\bt^d} \left | \int_{\bt^d} \chi_r(y) \left [ g(x) - g(x-y) \right ] \di y \right |^2 \di x \\
& \leq \int_{\bt^d} \left | \int_{\bt^d} \int_0^1 - y \chi_r(y) \nabla g(x - h y) \di h \di y \right |^2 \di x .
\end{align}
Hence, by Jensen's inequality applied to the probability measure $\chi_r(y) \di y \di h$ on $\bt^d \times [0,1]$,
\begin{align}
\label{eq:g ast}
\| g - \chi_r \ast g \|^2_{L^2(\bt^d)}  & \leq \int_0^1 \int_{\bt^d}  |y|^2 \chi_r(y)  \int_{\bt^d} \left | \nabla g(x - h y) \right |^2 \di x \di y \di h \nonumber\\
& \leq r^{2} \| \nabla g \|^2_{L^2(\bt^d)} \int_{\bt^d} |y|^2 \chi_1(y) \di y  \leq C r^{2} \| \nabla g \|^2_{L^2(\bt^d)} .
\end{align}
This estimate extends by density to $g \in W^{1,2}(\bt^d)$.

Next, standard estimates for the Poisson equation and $L^p$ interpolation inequalities imply that
\be  \label{est-rho-L2}
\| \nabla^2 \bar U_r \|_{L^2(\bt^d)}  \leq C \| \rho_{f_r} \|_{L^2(\bt^d)}  \leq  C \| \rho_{f_r} \|^{\frac{d-2}{2d}}_{L^\infty(\bt^d)} \| \rho_{f_r} \|^{\frac{d+2}{2d}}_{L^{\frac{d+2}{d}}(\bt^d)}  \leq C M^{\frac{d-2}{2d}} ,
\ee
since by \eqref{energy-reg-bdd} and \eqref{reg-rho-Lp} we have
\be \label{reg-rho-Lp-RUR}
 \| \rho_{f_r} \|_{L^{\frac{d+2}{d}}(\bt^d)} \leq C
\ee
for some $C$ depending on $d$ and $C_0$ only. 

Therefore, using \eqref{eq:g ast},
\be
\| \nabla \bar U_r - \chi_r \ast \nabla \bar U_r \|_{L^2(\bt^d)} \leq C\| \nabla^2 \bar U_r \|_{L^2(\bt^d)}r\leq  CM^{\frac{d-2}{2d}} r,
\ee
and we conclude that
\be
I_2 \leq C M^2 (r + \lambda^{-1} \sqrt{D}) ^2 .
\ee
The term $I_3$ is estimated like $I_1$, using the regularity estimate from Proposition~\ref{prop:regU} and \eqref{reg-rho-Lp-RUR}:
\be
\| \widehat U_r \|_{C^{2, \alpha}(\bt^d)} \leq C ,
\ee
where $C$ depends only on the constant $C_0$ controlling the initial energy in \eqref{energy-reg-bdd}. We obtain
\be 
I_3 \leq C\int |X^{(1)}_t - X^{(2)}_t|^2 \di \pi_0  \leq C \, \lambda^{-2} D .
\ee
Finally, $I_4$ is estimated in the same way as $I_2$, using the stability estimate Lemma~\ref{lem:hatU-stab}, Proposition~\ref{prop:regU} and \eqref{reg-rho-Lp-RUR}:
\be
\| \nabla \widehat U  - \nabla \widehat U_r \|_{L^2(\bt^d)} \leq C \| \nabla \bar U - \nabla \bar U_r \|_{L^2(\bt^d)} \leq C \sqrt{M} (r + \lambda^{-1} \sqrt{D})  .
\ee
By Proposition~\ref{prop:regU} and \eqref{reg-rho-Lp-RUR}, we have
\begin{align}
\| U_r \|_{L^\infty} \leq \| \bar U_r \|_{C^{0,\alpha}(\bt^d)} + \| \widehat U_r \|_{C^{1,\alpha}(\bt^d)}  \leq \exp{\Bigl(C \Bigl(1 +  \lVert \chi_r \ast \rho_{f_r} \rVert_{L^{\frac{d+2}{d}}(\bt^d)} \Bigr) \Bigr)}
 \leq C,
\end{align}
where $C$ depends on $C_0$ and $d$ only. Hence
\be
\| e^{U_r} \|_{L^2(\bt^d)} \leq C .
\ee
Since $\Delta \widehat U_r =e^{U_r}-1$,
by standard regularity results for the Poisson equation this implies that $$\| D^2 \widehat U_r \|_{L^2(\bt^d)}\leq C.$$ 
Therefore, using \eqref{eq:g ast} again,
\be
\| \nabla \widehat U_r - \chi_r \ast \nabla \widehat U_r \|_{L^2(\bt^d)} \leq C \| D^2 \widehat U_r  \|_{L^2(\bt^d)} \, r \leq Cr,
\ee
and we conclude that
\be
I_4 \leq C M^2 (r + \lambda^{-1} \sqrt{D}) ^2 .
\ee
Altogether we have
\be
\dot D \leq (\alpha + \lambda) D + \frac{C}{ \alpha \lambda^2}  \left ( |\log{r}|^2  + 1  \right )M^2 D + C M^2 r^2.
\ee
Optimising the exponent, we deduce that
\be
\dot D \leq  C \lambda_* D  + CM^2 r^2 ,
\ee
where
\be
\lambda_* = \sqrt{M} (1 + |\log{r}|)^{1/4} .
\ee
Hence
\begin{align}
D(t) & \leq \left ( D(0) +  C M^2 r^2 \right ) \exp{\[ C \sqrt{M} (1 + |\log{r}|)^{1/4} t \]} \\
& \leq ( D(0) + CM^2 r^2 ) e^{C\sqrt{M |\log{r}| \, } t}  .
\end{align}
Since $D$ controls the squared Wasserstein distance \eqref{Dl-ctrl-W},
\be
W_2^2(f_r(t), f(t)) \leq  ( D(0) + CM^2r^2 ) e^{C\sqrt{M |\log{r}| \, } t}  .
\ee
Then, by \eqref{W-ctrl-Dl-init}, and since $f$ and $f_r$ share the same initial datum,
\begin{align}
W_2^2(f_r(t), f(t)) & \leq  ( \inf_{\pi_0} D(0) + CM^2r^2 ) e^{C\sqrt{M |\log{r}| \, } t} \\
& \leq \left ( \lambda_*^{2} W_2^2(f(0), f(0)) + CM^2r^2 \right ) e^{C \sqrt{M |\log{r}| \, } t} \\
& \leq CM^2r^2 e^{C\sqrt{M |\log{r}| \, } t}  .
\end{align}
Since
\be
 r^2 e^{C\sqrt{M |\log{r}| \, } t} = r^{2 - C \sqrt{\frac{ M t^2}{|\log{r}|}}},
\ee
we conclude that for any compact time interval $[0,T]$,
\be 
\lim_{r \to 0} \sup_{t \in [0,T]} W_2^2(f_r(t), f(t)) = 0 .
\ee
\end{proof}

\begin{proof}[Proof of Theorem~\ref{thm:MFL}]

Let $f_r$ be the unique solution of \eqref{vpme-reg} with initial datum $f_0$. By the triangle inequality for $W_2$, we have
\be
\sup_{t \in [0,T]} W_2^2(\mu_r^N(t), f(t)) \leq \sup_{t \in [0,T]} W_2^2(\mu_r^N(t), f_r(t)) + \sup_{t \in [0,T]} W_2^2(f_r(t), f(t)) .
\ee
We apply Lemma~\ref{lem:wkstr} to the first term using the assumption on the initial configurations and deduce that it converges to zero as $r$ tends to zero. For the second term we apply Lemma~\ref{lem:RUR}, since $f$ and $f_r$ have the same initial datum. This completes the proof.

\end{proof}

\section{Combined quasineutral and mean field limit: proof of \newline Theorem \ref{thm:MFQN}}
\label{sec:proof-MFLQN}

In this section, we prove Theorem~\ref{thm:MFQN}. The idea is to use the scaled Vlasov-Poisson equation \eqref{vpme-quasi} as a bridge between the particle system \eqref{ODE:rQN} and the kinetic isothermal Euler system \eqref{KIE}. This approach was previously used for the classical Vlasov-Poisson system in \cite{IGP}. In order to pass from the particle system to the Vlasov-Poisson system, we need to revisit the estimates from Section~\ref{sec:MFL} in order to quantify the dependence of the constants on $\e$.

We will be using a scaled version of the regularised VPME system \eqref{vpme-reg}:
\begin{equation}
\label{vpme-reg-QN}
 \left\{ \begin{array}{ccc}\pt_t f^{(r)}_\e+v\cdot \nabla_x f^{(r)}_\e+ E_{\e,r} \cdot \nabla_v f^{(r)}_\e=0,  \\
E_{\e,r} =- \chi_r \ast \nabla U_{\e,r}, \\
\e^2 \Delta U_{\e,r}=e^{U_{\e,r}} - \chi_r \ast \rho[f^{(r)}_\e] ,\\
f^{(r)}_\e\vert_{t=0}=f_{\e}(0)\ge0,\ \  \int_{\bt^d \times \br^d} f_{\e}(0)\,dx\,dv=1.
\end{array} \right.
\end{equation}
We begin with a few remarks about this system. The well-posedness theory is clearly the same as for the case $\e =1$. The associated energy needs to be scaled with $\e$ in the following way:
\be \label{def:Ee-reg2}
\mc{E}_{\e,r} [f] : = \frac{1}{2} \int_{\bt^d \times \br^d} |v|^2 f \di x \di v + \frac{\e^2}{2} \int_{\bt^d} \lvert \nabla U_r \rvert^2 \di x + \int_{\bt^d} U_r e^{U_r} \di x .
\ee
Using the conservation of this energy, we can show that any solution $f_\e^{(r)}$ of \eqref{vpme-reg-QN} on the time interval $[0,T]$ with bounded mass density and $f_\e^{(r)}(0) \in L^\infty(\bt^d \times \br^d)$ satisfies
\be \label{reg-rho-Lp-eps}
\sup_{t \in [0,T]} \| \rho[ f_\e^{(r)}(t) ] \|_{L^{\frac{d+2}{d}(\bt^d)}} \leq C \left ( \mc{E}_{\e,r} [f_\e^{(r)}(0)] + 1 \right ) ,
\ee
where $C$ is independent of $\e$. This follows from the same argument as we gave before \eqref{reg-rho-Lp}.

The following lemma is a quantified version of Lemma~\ref{lem:wkstr}.

\begin{lem}[Weak-strong stability for the regularised equation, with quasineutral scaling] \label{lem:wkstr-QN}
Let $f_{\e}(0)$ satisfy the assumptions of Theorem~\ref{thm:quasi}. Let $f_{\e}^{(r)}$ be the solution of \eqref{vpme-reg-QN} with data $f_{\e}(0)$.
Fix $T>0$. Then there exists a constant $K$ depending on $T$ and $\{f_{\e}(0)\}$ such that the following holds. For each $(\e,r)$, let $\mu_{\e}^{(r)}$ be a measure solution of \eqref{vpme-reg-QN}. Assume that for some $\eta > 0$, $r = r(\e)$ satisfies 
\be
r \leq \[ \mathbf{\overline{\exp}_3}{(K \e^{-2})} \]^{-1}\qquad
\text{and}\qquad 
\lim_{r \to 0}   \frac{W_2\left (\mu_{\e}^{(r)}(0), f_{\e}(0) \right)}{r^{(d+ 2+\eta)/2} } = 0.
 \\ \label{MFQN-initconv}
\ee
Then
\be
\lim_{r \to 0} \sup_{t \in [0,T]} W_2\left (\mu_{\e}^{(r)}(t), f_{\e}^{(r)}(t) \right) = 0.
\ee

\end{lem}

\begin{proof} For ease of notation we drop the sub and superscripts on $f_\e^{(r)}$ and $\mu_\e^{(r)}$. We use the same method as for Lemma~\ref{lem:wkstr}, tracking the dependence on $\e$. Let $\pi_0 \in \Pi(f_0, \mu_0)$. We define a time dependent coupling $\pi_t$ as in \eqref{def:pit} with $(X^{(i)}_{x,v}, V^{(i)}_{x,v})$ being the characteristic flows corresponding to $f$ and $\mu$:
\be  \label{MF-eqn:char-2-QN}
\begin{array}{ll}
\begin{cases}\dot X^{(1)}_{x,v} = V^{(1)}_{x,v}  \\
\dot V^{(1)}_{x,v}  = E^{(\mu)}_{\e,r}  (X^{(1)}_{x,v})\\
(X^{(1)}_{x,v}(0), V^{(1)}_{x,v}(0)) = (x, v)\\
E^{(\mu)}_{\e,r} = - \chi_r \ast \nabla U_{\e,r}^{(\mu)}\\
\e^2 \Delta U_{\e,r}^{(\mu)} =e^{U_{\e,r}^{(\mu)}} - \chi_r \ast \rho_\mu
\end{cases} &\quad \quad
\begin{cases}
\dot X^{(2)}_{x,v} = V^{(2)}_{x,v} \\
\dot V^{(2)}_{x,v}  = E^{(f)}_{\e,r} (X^{(2)}_{x,v}) \\
(X^{(2)}_{x,v}(0), V^{(2)}_{x,v}(0)) = (x, v)\\
E^{(f)}_{\e,r} = - \chi_r \ast \nabla U_{\e,r}^{(f)}\\
\e^2 \Delta U_{\e,r}^{(f)}  = e^{U_{\e,r}^{(f)}} - \chi_r \ast \rho_{f} 
\end{cases} 
\end{array}
\ee
Then, since $f$ and $\mu$ are the pushforwards of $f_0$ and $\mu_0$ along their respective characteristic flows, $\pi_t$ is a coupling of $f_t$ and $\mu_t$. 

Again we define an anisotropic functional $D$:
\be \label{def:Dl-eps}
D(t) = \frac{1}{2} \int \lambda^2 |x-y|^2 + |v-w|^2 \di \pi_t(x,v,y,w) .
\ee

As before,
\be
\dot D \leq (\alpha + \lambda) D + \frac{C}{ \alpha} \sum_{i=1} I_i,
\ee
where
\be  \label{def:Ii-MF3}
\begin{split}
I_1  := \int |\bar E_{\e,r}^{(\mu)}(X^{(1)}_t) - \bar E^{(\mu)}_{\e,r}(X^{(2)}_t)|^2 \di \pi_0 , \qquad &
I_2  := \int |\bar E_{\e,r}^{(f)}(X^{(2)}_t) - \bar E^{(\mu)}_{\e,r}(X^{(2)}_t)|^2 \di \pi_0, \\
I_3  := \int |\widehat E_{\e,r}^{(\mu)}(X^{(1)}_t) - \widehat E^{(\mu)}_{\e,r}(X^{(2)}_t)|^2 \di \pi_0, \qquad &
I_4  := \int |\widehat E_{\e,r}^{(f)}(X^{(2)}_t) - \widehat E^{(\mu)}_{\e,r}(X^{(2)}_t)|^2 \di \pi_0 .
\end{split}
\ee

To estimate these quantities, we first note some basic $L^p(\bt^d)$ estimates on the regularised mass density $\chi_r \ast \rho_\mu$, using \eqref{Est-rho-mu}. For $p= \frac{d+2}{d}$, since $f_{\e}(0)$ satisfies \eqref{unif-energy}, by \eqref{reg-rho-Lp-eps} we obtain
\be
\| \rho_{f} \|_{L^\frac{d+2}{d}(\bt^d)} \leq C ,
\ee
where $C$ depends on $C_0$ and $d$ only. Thus
\be  \label{Est-rho-mu-pd-eps}
\lVert \chi_r * \rho_\mu \rVert_{L^{\frac{d+2}{d}}(\bt^d)} \leq C_d \left ( 1 + r^{-(d+2)}  \lambda^{-2} D \right ) ,
\ee
for $C_d$ depending on $C_0$ and $d$.

For $p=\infty$ we obtain
\be \label{Est-rho-mu-infty-eps}
\lVert \chi_r * \rho_\mu \rVert_{L^{\infty}(\bt^d)} \leq C_d \left ( M_\e + r^{-(d+2)}  \lambda^{-2} D \right ) .
\ee
where $M_\e$ is a constant such that
\be
\sup_{t \in [0,T]} \|\rho_{f_t}\|_{L^\infty(\bt^d)} \leq M_\e .
\ee
By Proposition~\ref{prop:growth} (which also applies to the regularised system), there exists a constant $C$ depending on $T$ such that
\be \label{mass-scaled}
\sup_{t \in [0,T]} \|\rho_{f_t}\|_{L^\infty(\bt^d)} \leq \exp{(C \e^{-2})}.
\ee
Therefore $M_\e$ may be chosen to satisfy
\be \label{M-eps}
M_\e \leq \exp{(C \e^{-2})}.
\ee

We estimate $I_1$ as in \eqref{MF-I1-est}. An extra factor of $\e^{-4}$ appears due to the quasineutral scaling on the force:
\be
I_1 \leq C \e^{-4} (\log{r})^2 \left ( M_\e + r^{-(d+2)} \lambda^{-2} D \right )^2 \lambda^{-2} D .
\ee
Similarly, $I_2$ is estimated as in \eqref{MF-I2-est} using Lemma~\ref{lem:Loep} and \eqref{Est-rho-mu-infty-eps} to obtain 
\be
I_2 \leq C \e^{-4} M_\e (M_\e + r^{-(d+2)} \lambda^{-2} D) \lambda^{-2} D .
\ee

For $I_3$ the same computation as in \eqref{MF-I3-est} implies that
\be 
I_3 \leq \int_{(\bt^d \times \br^d )^2} \lVert \widehat U_{\e,r}^{(\mu)} \rVert^2_{C^{2}(\bt^d)} |x-y|^2 \di \pi_t .
\ee
Using \eqref{Est-rho-mu-pd-eps} we apply Proposition~\ref{prop:regU} to obtain
\be \label{hatUmu-reg-eps}
\| \widehat U^{(\mu)}_{\e,r} \|_{C^2(\bt^d)} \leq C_d \, \mathbf{\overline{\exp}_2}{\left ( C_d \e^{-2} \left ( 1 + r^{-(d+2)} \lambda^{-2} D \right )  \right )} .
\ee
Thus
\begin{align}
I_3 & \leq C \, \mathbf{\overline{\exp}_2}{\left [ C_{d} \e^{-2} (1 + r^{-(d + 2)} \lambda^{-2} D)  \right ]} \int_{(\bt^d \times \br^d )^2} |x-y|^2 \di \pi_t \\
& \leq C \, \mathbf{\overline{\exp}_2}{\left [ C_{d} \e^{-2} (1 + r^{-(d + 2)} \lambda^{-2} D)  \right ]} \lambda^{-2} D .
\end{align}

For $I_4$, as in \eqref{MF-I4-est} we obtain
\be 
I_4 \leq \lVert \rho_f \rVert_{L^{\infty}(\bt^d)} \lVert \nabla \widehat U_{\e,r}^{(f)} - \nabla \widehat U_{\e,r}^{(\mu)} \rVert^2_{L^2(\bt^d)} .
\ee
By Lemma~\ref{lem:hatU-stab} and \eqref{Est-rho-mu-infty-eps},
\begin{align}
\lVert \nabla \widehat U_{\e,r}^{(f)} - \nabla \widehat U_{\e,r}^{(\mu)} \rVert^2_{L^2(\bt^d)}
& \leq C \e^{-2} \,  \mathbf{\overline{\exp}_2}{\left [ C_{d} \, \e^{-2} (1 + r^{-(d + 2)} \lambda^{-2} D)  \right ]} \lVert \bar U_{\e,r}^{(f)} - \bar U_{\e,r}^{(\mu)} \rVert_{L^2(\bt^d)}^2 \\
& \leq C \e^{-6} \, \mathbf{\overline{\exp}_2}{\left [ C_{d} \, \e^{-2} (1 + r^{-(d + 2)} \lambda^{-2} D)  \right ]} (M_\e + r^{-(d+2)} \lambda^{-2} D ) \lambda^{-2} D .
\end{align}
Thus
\be 
I_4 \leq C \e^{-6} M_\e \, \mathbf{\overline{\exp}_2}{\left [ C_{d} \e^{-2} (1 + r^{-(d + 2)} \lambda^{-2} D)  \right ]} (M_\e + r^{-(d+2)} \lambda^{-2} D ) \lambda^{-2} D .
\ee

We summarise this as 
\begin{align} \label{Gr-MF-nonlin-eps}
\dot D & \leq (\lambda + \alpha)D + \frac{1}{\alpha} C \e^{-4 }(1 + | \log{r} |)^2  \left ( M_\e + r^{-(d+2)} \lambda^{-2} D \right )^2 \lambda^{-2} D \\
& \qquad + C  \mathbf{\overline{\exp}_2}{\left [ C_{d} \e^{-2} (1 + r^{-(d + 2)} \lambda^{-2} D)  \right ]} \left ( 1 + M_\e ( M_\e + r^{-(d+2)} \lambda^{-2} D ) \right ) \lambda^{-2} D .
\end{align}

Again we introduce the truncated functional
\be \label{def:hatD2}
\hat D = 1 \wedge \left ( r^{-(d+2)} \lambda^{-2} D \right ) .
\ee
In terms of this new functional, \eqref{Gr-MF-nonlin-eps} becomes
$$
\frac{\di }{ \di t} \hat D \leq (\lambda + \alpha) \hat D + \frac{1}{4 \alpha} C_d \left [ \e^{-4} \left ( 1 + |\log r | \right )^2 + \mathbf{\overline{\exp}_2}{(C \e^{-2})} \right ] M_\e^2 \lambda^{-2} \hat D .
$$

We optimise the exponent by choosing
\begin{align}
\alpha_* & = C_d  \left [ \e^{-4} \left ( 1 + |\log r | \right )^2 + \mathbf{\overline{\exp}_2}{(C \e^{-2})} \right ]^{1/2} M_\e \lambda^{-1}_* \\ \label{def:lstar-eps}
\lambda_* & =  C_d  \left [ \e^{-4} \left ( 1 + |\log r | \right )^2 + \mathbf{\overline{\exp}_2}{(C \e^{-2})} \right ]^{1/4} \sqrt{M_\e} .
\end{align}
Then
$$
\frac{\di }{ \di t} \hat D \leq C_d \lambda_* \hat D .
$$
From this we deduce
\begin{align}
\sup_{[0,T]} \hat D(t) & \leq \exp(C_{d} \lambda_* T) \hat D(0)  \\
& \leq r^{-(d+2)} \lambda_*^{-2} \exp(C_{d} \lambda_* T) D(0) \\
& \leq r^{-(d+2)} \exp(C_{d} \lambda_* T) D(0) \\
& \leq C_{d} \exp{\left [ |\log{r}|\left ( (d+2)  + C_{d}T \sqrt{M_\e} \e^{-1} |\log{r}|^{-1/2} \right ) + \sqrt{M_\e} \,\mathbf{\overline{\exp}_2}{(C \e^{-2})} \right ]} D(0) .
\end{align}
By the estimate \eqref{M-eps} on $M_\e$, we obtain
\be 
\sup_{[0,T]} \hat D(t) \leq C_{d} \exp{\left [ |\log{r}|\left ( (d+2)  + C_{d}T \exp{(C \e^{-2})} |\log{r}|^{-1/2} \right ) \right ]} \mathbf{\overline{\exp}_3}{(C \e^{-2})} D(0) .
\ee
Since by assumption \eqref{MFQN-initconv}
\be
|\log{r}|^{-1/2} \leq \exp{\[ - \frac{1}{2} \exp{(K \e^{-2})} \]},
\ee
we have 
\be 
\sup_{[0,T]} \hat D(t) \leq C_{d} \exp{\left [ |\log{r}|\left ( (d+2)  + \exp{(C \e^{-2} - \frac{1}{2} \exp{(K \e^{-2})} )} \right ) \right ]} \,\mathbf{\overline{\exp}_3} {(C \e^{-2})} D(0) .
\ee
By assumption on $W_2(\mu_{\e,r}^N(0), f_\e^{(r)}(0))$, for all sufficiently small $r$ there exists a choice of initial coupling $\pi_0^{(r)}$ such that
\be
D(0) < r^{d+2+\eta}.
\ee
Then
\be 
\sup_{[0,T]} \hat D(t) \leq C_{d} \exp{\left [ |\log{r}|\left ( \exp{(C \e^{-2} - \frac{1}{2} \exp{(K \e^{-2})} )} - \eta \right ) + \mathbf{\overline{\exp}_2}{(C \e^{-2})} \right ]}  .
\ee
For sufficiently small $\e$,
\begin{align}
\sup_{[0,T]} \hat D(t) & \leq C_{d} \exp{\left [ - \frac{1}{2} \eta |\log{r}| + \mathbf{\overline{\exp}_2}{(C \e^{-2})} \right ]}  \\
& \leq C_{d} \exp{\left [ - \frac{1}{2} \eta \,\mathbf{\overline{\exp}_2} {(K \e^{-2})} + \mathbf{\overline{\exp}_2}{(C \e^{-2})} \right ]}  .
\end{align}
Thus if $K > C$, then
\be 
\sup_{[0,T]} \hat D(t) \to 0
\ee
as $\e$ tends to zero. In particular, for $\e$ sufficiently small,
\be  \label{hatD-small2}
 \inf_{\pi_0} \sup_{t \in [0,T]} \hat D(t) < 1 .
\ee
Hence, for $\e$ sufficiently small,
\be
 \inf_{\pi_0} \sup_{t \in [0,T]} \hat D(t) =  \inf_{\pi_0} \sup_{t \in [0,T]} r^{-(d+2)} \lambda_*^{-2} D(t).
\ee
Thus
\be 
\sup_{t \in [0,T_*]} W_2^2(\mu^N_{\e,r}(t), f_\e^{(r)}(t))  \leq \inf_{\pi_0} \sup_{t \in [0,T_*]}D(t)= r^{d+2} \lambda_*^2  \inf_{\pi_0} \sup_{t \in [0,T]} \hat D(t) \leq r^{d+2} \lambda_*^2 .
\ee
By \eqref{def:lstar-eps},
\be
\lambda^2_*  =  C_d  \left [ \e^{-4} \left ( 1 + |\log r | \right )^2 + \,\mathbf{\overline{\exp}_2} {(C \e^{-2})} \right ]^{1/2} \exp{(C \e^{-2})} .
\ee
Hence, for any $\alpha > 0$,
\be 
r^{d+2} \lambda_*^2 \leq C r^{d+2 - \alpha} \,\mathbf{\overline{\exp}_2} {(C \e^{-2})} \leq C \exp{\left \{ \exp{(C \e^{-2})} - (d+2 - \alpha)\,\mathbf{\overline{\exp}_2} {(K \e^{-2})} \right \}}.
\ee
The right hand side converges to zero as $\e$ tends to zero. Therefore
\be
\lim_{\e \to 0} \sup_{t \in [0,T]} W_2^2(f_t, \mu_t) = 0 .
\ee

\end{proof}

The next lemma is a quantified version of Lemma~\ref{lem:RUR}.

\begin{lem}[Approximation of (VPME) in quasineutral scaling] \label{lem:RUR-QN}
Let $f_{\e}(0)$ satisfy the assumptions of Theorem~\ref{thm:quasi}. Let $f_\e^{(r)}$ be the solution of the scaled and regularised Vlasov equation \eqref{vpme-reg-QN} with initial datum $f_{\e}(0)$. Let $f_\e$ be the unique bounded density solution of \eqref{vpme-quasi}.

Fix $T > 0$. Then there exists a constant $C$ depending on $T$ and on $\{f_{\e}(0)\}_\e$ such that the following holds. If $r$ and $\e$ satisfy
\be \label{reps-RUR-QN}
r \leq \[ \mathbf{\overline{\exp}_3}{(C \e^{-2})} \]^{-1},
\ee
then
\be
\lim_{r \to 0} \sup_{[0,T]} W_2 \left ( f_\e^{(r)}(t), f_\e(t) \right ) = 0.
\ee

\end{lem}

\begin{proof}
By \eqref{unif-energy}, Lemma~\ref{lem:rho-Lp} and \eqref{reg-rho-Lp-eps} there exists a constant $C$ depending on $C_0$ and $d$ only such that
\be \label{mass-Lp-scaled}
\| \rho_{f^{(r)}_\e} \|_{L^{\infty}([0,T]; L^{\frac{d+2}{d}}(\bt^d))}, \| \rho_{f_\e} \|_{L^\infty([0,T]; L^{\frac{d+2}{d}}(\bt^d))} \leq C .
\ee
By Proposition~\ref{prop:growth}, there exists a constant $C$ depending on the initial data and $T$ such that
\be \label{mass-scaled-2}
\| \rho_{f^{(r)}_\e} \|_{L^\infty([0,T]; L^\infty(\bt^d))}, \| \rho_{f_\e} \|_{L^\infty([0,T]; L^\infty(\bt^d))} \leq \exp{(C \e^{-2})}.
\ee

We will control the Wasserstein distance between $f_\e^{(r)}$ and $f_\e$ using a particular coupling $\pi_t$. Since both solutions share the same initial datum, we take $\pi_0$ to be the trivial coupling $f_{\e}(0)(x,v) \delta((x,v) - (y,w)) \di x \di v \di y \di w$. We construct $\pi_t \in \Pi(f_\e(t), f_\e^{(r)}(t))$ as in \eqref{def:pit}, using the scaled version of the characteristic systems in \eqref{RUR:char-1}:
\be  \label{RUR:char-1-scaled}
\begin{array}{ll}
\begin{cases}
\dot X^{(1)}_{x,v} = V^{(1)}_{x,v} \\
\dot V^{(1)}_{x,v}  = E  (X^{(1)}_{x,v})\\
(X^{(1)}_{x,v}(0), V^{(1)}_{x,v}(0)) = (x, v)\\
E = - \nabla U\\
\e^2 \Delta U = e^{U} -  \rho_{f_\e}
\end{cases}  &\qquad\qquad 
\begin{cases}
\dot X^{(2)}_{x,v} = V^{(2)}_{x,v} \\
\dot V^{(2)}_{x,v}  = E_r (X^{(2)}_{x,v}) \\
(X^{(2)}_{x,v}(0), V^{(2)}_{x,v}(0)) = (x, v)\\
E_r = - \chi_r \ast \nabla U_{r}\\
\e^2 \Delta U_{r}  = e^{U_{r}} - \chi_r \ast \rho_{f_\e^{(r)}}.
\end{cases} 
\end{array}
\ee

We define $D$ as in \eqref{def:Dl}. As in Lemma~\ref{lem:wkstr}, we obtain for any $\alpha > 0$,
\be
\dot D \leq (\alpha + \lambda) D + \frac{C}{ \alpha} \sum_{i=1}^5 I_i,
\ee
where
\be \label{def:Ii-MF4}
\begin{split}
&I_1  := \int |\nabla \bar U_{r}(X^{(1)}_t) - \nabla \bar U_{r}(X^{(2)}_t)|^2 \di \pi_0, \qquad
I_2  := \int |\nabla \bar U_{r}(X^{(1)}_t) - \nabla \bar U(X^{(1)}_t)|^2 \di \pi_0, \\
&I_3  := \int |\nabla \widehat U_{r}(X^{(1)}_t) - \nabla \widehat U_{r}(X^{(2)}_t)|^2 \di \pi_0, \qquad
I_4  := \int |\nabla \widehat U(X^{(1)}_t) - \nabla \widehat U_{r}(X^{(1)}_t)|^2 \di \pi_0 \\
&I_5 : =  \int | \chi_r \ast \nabla U_{r}(X^{(2)}_t) - \nabla U_{r}(X^{(2)}_t)|^2 \di \pi_0 .
\end{split}
\ee

$I_1$ is estimated as in \eqref{RUR-I1-est}. There is an extra factor of $\e^{-4}$ due to the quasineutral scaling and the mass bound $M=M_\e$ depends on $\e$:
\be 
I_1 \leq C \e^{-4} |\log{r}|^2 M_\e^2 \, \lambda^{-2} D .
\ee

We estimate $I_2$ as in \eqref{RUR-I2-est}, keeping track of the dependence on $\e$ in Lemma~\ref{lem:Loep}:
\be
\| \nabla \bar U - \nabla \bar U_{r} \|_{L^2(\bt^d)} \leq C \e^{-2} \sqrt{M_\e} (r + \lambda^{-1} \sqrt{D}) .
\ee
We conclude that
\be
I_2 \leq C \e^{-4} M_\e^2 (r + \lambda^{-1} \sqrt{D}) ^2 .
\ee

For $I_3$, by the regularity estimate from Proposition~\ref{prop:regU} and the uniform $L^{\frac{d+2}{d}}(\bt^d)$ estimate on the density \eqref{mass-Lp-scaled} we have
\be
\| \widehat U_{r} \|_{C^{2, \alpha}(\bt^d)} \leq \mathbf{\overline{\exp}_2}{(C \e^{-2})} ,
\ee
where $C$ depends only on $C_0$. We obtain
\begin{align}
I_3 & \leq \mathbf{\overline{\exp}_2}{(C \e^{-2})} \int_{\bt^d \times \br^d} |X^{(1)}_t - X^{(2)}_t|^2 \di \pi_0 \\
& \leq \mathbf{\overline{\exp}_2}{(C \e^{-2})} \, \lambda^{-2} D .
\end{align}

For $I_4$ we use the stability estimate from Lemma~\ref{lem:hatU-stab} and the $L^{\frac{d+2}{d}}(\bt^d)$ estimate \eqref{mass-Lp-scaled}:
\be
\| \nabla \widehat U - \nabla \widehat U_{r} \|_{L^2(\bt^d)} \leq \mathbf{\overline{\exp}_2}{(C \e^{-2})} \| \nabla \bar U - \nabla \bar U_{r} \|_{L^2(\bt^d)} \leq \mathbf{\overline{\exp}_2}{(C \e^{-2})} \sqrt{M_\e} (r + \lambda^{-1} \sqrt{D})  .
\ee
Hence
\be
I_4 \leq \mathbf{\overline{\exp}_2}{(C \e^{-2})} M_\e^2 (r + \lambda^{-1} \sqrt{D}) ^2 .
\ee

For $I_5$, by \eqref{eq:g ast} we have
\be
\| \chi_r \ast \nabla U_r - \nabla U_r \|_{L^2(\bt^d)}^2 \leq C r^2 \| U_r \|_{W^{2,2}(\bt^d)}^2 .
\ee
Thus
\begin{align}
I_5 & : = \int_{\bt^d} \left \lvert \chi_r \ast \nabla U_{r}(x) - \nabla U_{r}(x) \right \rvert^2 \rho_{f^{(r)}_\e}(\di x) \\
& \leq \| \rho_{f^{(r)}_\e} \|_{L^\infty(\bt^d)} \|\chi_r \ast \nabla U_{r} - \nabla U_{r}\|_{L^2(\bt^d)}^2 \\
& \leq C r^2 M_\e \|U_{r}\|_{W^{2,2}(\bt^d)}^2 \\
& \leq C r^2 \e^{-4} M_\e \| e^{U_{r}} - \rho_{f^{(r)}_\e}\|_{L^2(\bt^d)}^2 .
\end{align}
As in \eqref{est-rho-L2}, we have
\be 
\| \rho_{f_\e^{(r)}} \|_{L^2(\bt^d)}
\leq   \| \rho_{f_\e^{(r)}} \|^{\frac{d-2}{2d}}_{L^\infty(\bt^d)} \| \rho_{f_\e^{(r)}} \|^{\frac{d+2}{2d}}_{L^{\frac{d+2}{d}}(\bt^d)} \leq C M_\e^{\frac{d-2}{2d}} .
\ee
To estimate $e^{U_{r}} $, first note that by Proposition~\ref{prop:regU} and \eqref{mass-Lp-scaled},
\be \label{Ur-Linfty}
\| \bar U_{r}\|_{L^\infty(\bt^d)} \leq C \e^{-2} .
\ee
By \eqref{est:eU-Lp},
\be
\| e^{\widehat{U}_r} \|_{L^2(\bt^d)} \leq \exp(C \e^{-2}) .
\ee
Thus
\be
\| e^{U_r} \|_{L^2(\bt^d)} \leq C \exp{\left (\| \bar U_r\|_{L^\infty(\bt^d)} \right )} \, \| e^{\widehat{U}_r} \|_{L^2(\bt^d)} \leq \exp{(C \e^{-2})} .
\ee
Therefore
\be
I_5 \leq C r^2 \e^{-4} M_\e \left ( M_\e^{\frac{d-2}{d}} + e^{C \e^{-2}}  \right ) .
\ee

Putting these five estimates together, we obtain
\begin{multline}
\dot D \leq (\alpha + \lambda) D  +  \frac{C}{\alpha \lambda^2} r^2 \mathbf{\overline{\exp}_2}{(C \e^{-2})} M_\e^2 \\
+ \frac{C}{ \alpha \lambda^2}  \left \{ \mathbf{\overline{\exp}_2}{(C \e^{-2})} M_\e^2 + \e^{-4}|\log r|^2 M_\e^2 +  \mathbf{\overline{\exp}_2}{(C \e^{-2})} \right \} D.
\end{multline}

By \eqref{mass-scaled-2}, we may estimate that
\be
M_\e \leq \exp{(C \e^{-2})} .
\ee
From this we deduce
\be
\dot D \leq (\alpha + \lambda) D + \frac{C}{ \alpha \lambda^2}  \left \{ \exp{(C \e^{-2})}|\log r|^2 + \mathbf{\overline{\exp}_2}{(C \e^{-2})} \right \} D + \frac{C}{\alpha \lambda^2} \mathbf{\overline{\exp}_2}{(C \e^{-2})} r^2 .
\ee

After choosing $\alpha$ and $\lambda$ so as to minimise the constant in front of $D$ we obtain
\be
\dot D \leq  C \lambda_* D  + \mathbf{\overline{\exp}_2}{(C \e^{-2})} r^2 ,
\ee
where
\be
\lambda_* =  \left [ \exp{(C \e^{-2})}|\log r|^2 + \mathbf{\overline{\exp}_2}{(C \e^{-2})} \right ]^{1/4} \geq 1 ,
\ee
for $r, \e$ sufficiently small.

Therefore, by a Gr\"onwall estimate
\be 
D(t) \leq \left ( D(0) +  \mathbf{\overline{\exp}_2}{(C \e^{-2})}  r^2 \right ) \exp{\[ C t \lambda_* \]} .
\ee
Since $\pi_0$ was trivial, $D(0) = 0$. Since $D$ controls the squared Wasserstein distance,
\be
W_2^2(f_\e^{(r)}(t), f_\e(t)) \leq  \exp{(C \e^{-2})}  r^2 \exp{\[ C t \lambda_* \]} .
\ee
Then, by definition of $\lambda_*$,
\be
\sup_{t \in [0,T]} W_2^2(f_\e^{(r)}(t), f_\e(t)) \leq r^2 \, \exp{\[ \exp{(C_T \e^{-2})} |\log r|^{1/2} \]} \cdot \mathbf{\overline{\exp}_3} {(C_T \e^{-2})} .
\ee
If $r \leq \[ \mathbf{\overline{\exp}_3}{(K \e^{-2})} \]^{-1}$, then
\be
e^{C \e^{-2}} |\log r|^{-1/2} \leq \exp{\[ C_T \e^{-2} - \frac{1}{2} \exp{(K \e^{-2})} \]} \to 0
\ee
as $\e$ tends to zero, for any $K \geq 0$. Hence, for any $\eta > 0$, for $\e$ sufficiently small,
\begin{align}
r^2 \, \exp{\[ \exp{(C_T \e^{-2})} |\log r|^{1/2} \]} & \leq r^{2 - \exp{(C_T \e^{-2})} |\log r|^{-1/2}} \\
& \leq r^{2 - \eta} .
\end{align}
Moreover,
\be 
r^{2 - \eta} \, \mathbf{\overline{\exp}_3}{(C_T \e^{-2})} \leq \exp{\[ \mathbf{\overline{\exp}_2}{(C_T \e^{-2})} - (2 - \eta) \mathbf{\overline{\exp}_2}{(K \e^{-2})}  \]} ,
\ee
which converges to zero for any $\eta < 2$ as $\e$ tends to zero, as long as $K \geq C_T$.

Therefore, if $r \leq \[ \mathbf{\overline{\exp}_3}{(K \e^{-2})} \]^{-1}$ for $K \geq C_T$, then as $\e$ tends to zero (and so $r$ also tends to zero),
\be
\sup_{t \in [0,T]} W_2(f_\e^{(r)}(t), f_\e(t)) \to 0.
\ee

\end{proof}

\begin{proof}[Proof of Theorem~\ref{thm:MFQN}]
In the following, $\mu^N_{\e,r}$ denotes the empirical measure associated to the solution of the particle system \eqref{ODE:rQN}, while $f_{\e}^{(r)}$ denotes the solution of the scaled version of \eqref{vpme-reg-QN} with initial datum $f_{\e}(0)$. Let $g$ be the solution of the KIE system \eqref{KIE}, on some time interval $[0,T_*]$, obtained in the quasineutral limit from $f_\e$ using Theorem~\ref{thm:quasi}. By the triangle inequality for $W_1$,
\begin{multline} \label{MFQN-triangle}
\sup_{t \in [0,T_*]} W_1(\mu^N_{\e,r}(t), g(t)) \leq \sup_{t \in [0,T_*]} W_1(\mu^N_{\e,r}(t), f_\e^{(r)}(t))\\ + \sup_{t \in [0,T_*]} W_1(f_\e^{(r)}(t), f_\e(t)) + \sup_{t \in [0,T_*]} W_1(f_\e(t), g(t)) .
\end{multline}
The last term converges to zero as $\e$ tends to zero, by Theorem~\ref{thm:quasi}.

For the other two terms, we first observe that
\be
W_1(\mu^N_{\e,r}(t), f_\e^{(r)}(t)) \leq W_2(\mu^N_{\e,r}(t), f_\e^{(r)}(t)) ,\qquad W_1(f_\e^{(r)}(t), f_\e(t))  \leq W_2(f_\e^{(r)}(t), f_\e(t)).
\ee
Then the second term of \eqref{MFQN-triangle} converges to zero by Lemma~\ref{lem:wkstr-QN} and the third term of \eqref{MFQN-triangle} converges to zero by Lemma~\ref{lem:RUR-QN}, provided that \eqref{MFQN-initconv} is satisfied for $C$ depending on $T_*$.

\end{proof}

\section{Typicality: proof of Theorems~\ref{thm:typ-MF} and \ref{thm:typ-MFQN}}
\label{sec:proof-typicality}

In this last section, we prove Theorems~\ref{thm:typ-MF} and \ref{thm:typ-MFQN} concerning the relation between the choice of parameters in the mean field (or combined mean field-quasineutral) limit and the initial configurations. 
The underlying observation is that if one constructs a collection of empirical measures $(\nu^N)_N$ by drawing $N$ independent samples from a reference measure $\nu$, then, by the Glivenko-Cantelli theorem, almost surely $\nu^N$ will converge to $\nu$ as $N$ tends to infinity in the sense of weak convergence of measures. We want to use a quantitative version of this result to find configurations for which the associated empirical measures converge to our reference data $f_0$ sufficiently quickly.
A result of this type was proved by Fournier and Guillin in \cite{FG}, with the distance between $\nu^N$ and $\nu$ measured in Wasserstein sense. We use two slightly different versions of their concentration estimates. 

In the unscaled case $\e=1$ (Theorem~\ref{thm:typ-MF}), we give a result for more general data $f_0$ satisfying a moment condition. For this we need the following result from \cite[Theorem 2]{FG}.

\begin{thm} \label{thm:conc-mom}
Let $\nu$ be a probability measure on $\br^m$ and let $\nu^N$ denote the empirical measure of $N$ independent samples from $\nu$. Assume that $\nu$ has a finite $k$th moment for some $k > 2p$:
\be
M_k(\nu) : = \int_{\br^m} |x|^k \di \nu( x) < + \infty .
\ee
Then there exist constants $c, C$ depending on $p,m$ and $M_k(\nu)$ such that for any $x > 0$,
\be
\bb{P} \left ( W^p_p(\nu^N, \nu) \geq x \right) \leq a(N,x) \mathbbm{1}_{\{x \leq 1\}} + b(N,x) ,
\ee
where
\be
a(N,x) = C \begin{cases} \exp{(-cN x^2)} &p > \frac{m}{2} \\ \label{def:aN}
\exp{\(-cN \[\frac{x}{\log{(2 + \frac{1}{x})}}\]^2\)} &p = \frac{m}{2} \\
\exp{\(-cN x^{m/p}\)} &p < \frac{m}{2} 
\end{cases}
\ee
and
\be
b(N,x) = C N(Nx)^{-(k - \alpha)/p}
\ee
for any $\alpha \in (0,k)$.

\end{thm}

For the combined mean field and quasineutral limit, we work with a different initial datum $f_{\e}(0)$ for each $\e$. To use the Wasserstein concentration estimates, we need to take care of the dependence of the constants $c, C$ on (the moments of) $f_{\e}(0)$. In fact, since in Theorem~\ref{thm:quasi} we work with compactly supported data, we will find it more convenient to use a slightly different version of the estimates, designed for compactly supported measures. The following result is from \cite[Proposition 10]{FG}.

\begin{thm} \label{thm:conc-cpct}
Let $\nu$ be a probability measure supported on $(-1,1]^m$. Let $\nu^N$ denote the empirical measure of $N$ independent samples from $\nu$. Then there exist constants $c, C$ depending on $p$ and $m$ only such that for any $x > 0$,
\be
\bb{P} \left ( W^p_p(\nu^N, \nu) \geq x \right) \leq a(N,x) \mathbbm{1}_{\{x \leq 1\}} ,
\ee
where $a(N,x)$ is defined by \eqref{def:aN}.
\end{thm}

We can now prove Theorems \ref{thm:typ-MF} and \ref{thm:typ-MFQN}. Since the proofs are rather similar to the ones used in \cite{IGP} for the classical VP system, we invite the reader to consult that paper for more details.

\begin{proof}[Proof of Theorem~\ref{thm:typ-MF}]
The idea of the proof is to show that, for the choice $r = c N^{- \gamma}$,
\be \label{full-prob}
\bb{P} \left ( \limsup_{N \to \infty} \frac{W^2_2(\mu^N_0, f_0)}{r^{d + 2 + C_{T_*, M}|\log r|^{-1/2}}} < 1 \right ) = 1.
\ee
Then we may apply Theorem~\ref{thm:MFL} to conclude that the mean field limit holds on this full probability event. To prove \eqref{full-prob}, observe that
\be
\bigcup_{n} \bigcap_{N \geq n} A^c_N \subset \left \{ \limsup_{N \to \infty} \frac{W^2_2(\mu^N_0, f_0)}{r^{d + 2 + C_{T_*, M}|\log r|^{-1/2}}} < 1 \right \},
\ee
where $A_N$ is the event
\be
A_N := \left \{ W^2_2(\mu^N_0, f_0) > \frac{1}{2} r^{d + 2 + C_{T_*, M}|\log r|^{-1/2} } \right \}.
\ee
Since $\left ( \bigcup_{n} \bigcap_{N \geq n} A^c_N \right )^c = \bigcap_{n} \bigcup_{N \geq n} A_N$, by the Borel-Cantelli lemma it suffices to show that
\be
\label{eq:BC}
\sum_N \bb{P}(A_N) < \infty .
\ee
We estimate $\bb{P}(A_N)$ using Theorem~\ref{thm:conc-mom}, with 
\be 
x_N = \frac{1}{2} r^{d + 2 + C_{T_*, M}|\log r|^{-1/2}}= c N^{- \gamma(d+2) - C_{T_*, M,\gamma}|\log N|^{-1/2}}.
\ee
Note that $p=2$ and $m = 2d$. The assumptions on $\gamma$ in \eqref{def:gamma} are chosen such that
\be
\sum_N a(N, x_N) + b(N, x_N) < \infty.
\ee
In this way \eqref{eq:BC} holds and the result follows.

\end{proof}

\begin{proof}[Proof of Theorem~\ref{thm:typ-MFQN}]
We use the same strategy as for Theorem~\ref{thm:typ-MF}. However, since the constants in Theorem~\ref{thm:conc-mom} depend on the $k$th moment of $f_{\e}(0)$ which may change with $\e$, we will instead use the compact version of the estimate from Theorem~\ref{thm:conc-cpct}. Recall that we already assumed that $f_{\e}(0)$ were compactly supported with the support in velocity growing no faster than $e^{C \e^{-2}}$ for some $C$. We perform a scaling argument in order to work with measures that are supported in $(-1, 1]^{2d}$.

Given a probability measure $\nu$, let $\mc{S}_R [\nu]$ be the measure such that for any $A_X \in \mc{B}(\bt^d)$ and $B_V \in \mc{B}(\br^d)$,
\be
\mc{S}_R[\nu] (A_X \times B_V) = \nu(A_X \times R B_V) .
\ee
This maps measures supported in $[-1, 1]^d \times [-R, R]^d$ to measures supported in $[-1,1]^{2d}$. By \cite[Lemma 7.6]{IGP}, if $\nu_1$, $\nu_2$ are measures on $\bt^d \times \br^d$, then for any $p \in [1, \infty)$,
\be \label{Est_W_scale}
W_p(\nu_1, \nu_2) \leq R\, W_p(\mc{S}_R[\nu_1], \mc{S}_R[\nu_2]) . 
\ee
Note also that if $(Z^{(R)}_i)_{i=1}^N$ are $N$ independent samples from $\mc{S}_R[\nu]$, then $(Z^{(R)}_i)_{i=1}^N$ has the same law as $(X_i, \frac{1}{R} V_i)_{i=1}^N$, where $(Z_i)_{i=1}^N = (X_i, V_i)_{i=1}^N$ are $N$ independent samples from $\nu$.
It is therefore enough to show that
\be
\label{eq:BC2}
\sum_N \bb{P}(A_N) < \infty,
\ee
where $A_N$ denotes the event
\be
A_N : = \left \{ W^2_2(\mc{S}_{e^{-C \e^{-2}}}[\mu^N_\e(0)],\  \mc{S}_{e^{-C \e^{-2}}}[f_{0,\e}]) > \frac{1}{2} r^{d + 2 + \eta } \exp(-2C \e^{-2}) \right \} .
\ee
We observe that the assumption
\be
r < \left [ \mathbf{\overline{\exp}_3}{(K \e^{-2})} \right ]^{-1},
\ee
implies that
\be
\exp{(-C \e^{-2})} > c \left ( \log \log N \right )^{-\zeta} > c N^{-\alpha} 
\ee
for $\zeta$ depending on $C$ and $K$, any $\alpha > 0$ and $c$ depending on $\alpha$, $C$ and $K$. We then apply Theorem~\ref{thm:conc-cpct} with the choice
\be
x_N = c N^{-\gamma(d+2+\eta) - \alpha}.
\ee
The assumption \eqref{def:gamma-QN} on $\gamma$ implies that it is possible to find $\eta > 0$ such that
\be
\sum_N a(N, x_N) < \infty.
\ee
This yields \eqref{eq:BC2},
which completes the proof.

\end{proof}

 {\it Acknowledgements:} This work was supported by the UK Engineering and Physical Sciences Research Council (EPSRC) grant EP/L016516/1 for the University of Cambridge Centre for Doctoral Training, the Cambridge Centre for Analysis; and the European Research Council (ERC) under the European Union's Horizon 2020 research and innovation programme (grant agreement No 726386).

\bibliography{VPME-bib-11}

\begin{thebibliography}{10}

\bibitem{Amb}
L.~Ambrosio.
\newblock {Transport equation and Cauchy problem for Non-smooth Vector Fields}.
\newblock In B.~Dacorogna and P.~Marcellini, editors, {\em Calculus of
  Variations and Nonlinear Partial Differential Equations: Lectures given at
  the C.I.M.E. Summer School held in Cetraro, Italy, June 27 - July 2, 2005}.
  Springer-Verlag Berlin Heidelberg, 2008.

\bibitem{Bardos-Besse}
C.~Bardos and N.~Besse.
\newblock {The Cauchy problem for the Vlasov-Dirac-Benney equation and related
  issues in fluid mechanics and semi-classical limits}.
\newblock {\em Kinet. Relat. Models}, 6(4):893--917, 2013.

\bibitem{BGNS18}
C.~Bardos, F.~Golse, T.~T. Nguyen, and R.~Sentis.
\newblock The {M}axwell-{B}oltzmann approximation for ion kinetic modeling.
\newblock {\em Phys. D}, 376/377:94--107, 2018.

\bibitem{Bardos-Nouri}
C.~Bardos and A.~Nouri.
\newblock {A Vlasov equation with Dirac potential used in fusion plasmas}.
\newblock {\em J. Math. Phys.}, 53(11):115621, 2012.

\bibitem{BR}
J.~Batt and G.~Rein.
\newblock {Global classical solutions of the periodic Vlasov-Poisson system in
  three dimensions}.
\newblock {\em C. R. Acad. Sci. Paris S{\'{e}}r. I Math.}, 313(6):411--416,
  1991.

\bibitem{BFJJ}
M.~Bossy, J.~Fontbona, P.-E. Jabin, and J.-F. Jabir.
\newblock {Local existence of analytical solutions to an incompressible
  Lagrangian stochastic model in a periodic domain}.
\newblock {\em Comm. Partial Differential Equations}, 38(7):1141--1182, 2013.

\bibitem{Bouchut}
F.~Bouchut.
\newblock {Global weak solution of the Vlasov-Poisson system for small
  electrons mass}.
\newblock {\em Comm. Partial Differential Equations}, 16(8-9):1337--1365, 1991.

\bibitem{Braun-Hepp}
W.~Braun and K.~Hepp.
\newblock {The Vlasov dynamics and its fluctuations in the {$1/N$} limit of
  interacting classical particles}.
\newblock {\em Comm. Math. Phys.}, 56(2):101--113, 1977.

\bibitem{BG}
Y.~Brenier and E.~Grenier.
\newblock {Limite singuli{\`{e}}re du syst{\`{e}}me de Vlasov-Poisson dans le
  r{\'{e}}gime de quasi neutralit{\'{e}}: le cas ind{\'{e}}pendant du temps}.
\newblock {\em C. R. Acad. Sci. Paris S{\'{e}}r. I Math.}, 318(2):121--124,
  1994.

\bibitem{Dob}
R.~L. Dobrushin.
\newblock {Vlasov Equations}.
\newblock {\em Funktsional. Anal. i Prilozhen.}, 13(2):48--58, 1979.

\bibitem{Evans}
L.~C. Evans.
\newblock {\em {Partial differential equations}}, volume~19 of {\em Graduate
  Studies in Mathematics}.
\newblock American Mathematical Society, 2010.

\bibitem{FG}
N.~Fournier and A.~Guillin.
\newblock {On the rate of convergence in Wasserstein distance of the empirical
  measure}.
\newblock {\em Probab. Theory Related Fields}, 162(3-4):707--738, 2015.

\bibitem{GT}
D.~Gilbarg and N.~S. Trudinger.
\newblock {\em {Elliptic partial differential equations of second order}},
  volume 224 of {\em Grundlehren der mathematischen Wissenschaften}.
\newblock Springer-Verlag Berlin Heidelberg, 1977.

\bibitem{Golse}
F.~Golse.
\newblock {On the dynamics of large particle systems in the mean field limit}.
\newblock In A.~Muntean, J.~Rademacher, and A.~Zagaris, editors, {\em
  Macroscopic and Large Scale Phenomena: Coarse Graining, Mean Field Limits and
  Ergodicity}, volume~3 of {\em Lect. Notes Appl. Math. Mech.}, pages 1--144.
  Springer, 2016.

\bibitem{Grenier95}
E.~Grenier.
\newblock {Defect measures of the Vlasov-Poisson system in the quasineutral
  regime}.
\newblock {\em Comm. Partial Differential Equations}, 20(7-8):1189--1215, 1995.

\bibitem{Grenier96}
E.~Grenier.
\newblock {Oscillations in quasineutral plasmas}.
\newblock {\em Comm. Partial Differential Equations}, 21(3-4):363--394, 1996.

\bibitem{IGP-WP}
M.~Griffin-Pickering and M.~Iacobelli.
\newblock Global well-posedness in 3-dimensions for the {Vlasov--Poisson}
  system with massless electrons.
\newblock Preprint, 2018, 1810.06928.

\bibitem{IGP}
M.~Griffin-Pickering and M.~Iacobelli.
\newblock A mean field approach to the quasi-neutral limit for the
  {Vlasov--Poisson} equation.
\newblock {\em SIAM J. Math. Anal.}, 50(5):5502--5536, 2018.

\bibitem{HKH}
D.~Han-Kwan and M.~Hauray.
\newblock Stability issues in the quasineutral limit of the one-dimensional
  {Vlasov-Poisson} equation.
\newblock {\em Comm. Math. Phys.}, 334(2):1101--1152, 2015.

\bibitem{IHK2}
D.~Han-Kwan and M.~Iacobelli.
\newblock Quasineutral limit for {Vlasov-Poisson} via {Wasserstein} stability
  estimates in higher dimension.
\newblock {\em J. Differential Equations}, 263(1):1--25, 2017.

\bibitem{IHK1}
D.~Han-Kwan and M.~Iacobelli.
\newblock {The quasineutral limit of the Vlasov-Poisson equation in Wasserstein
  metric}.
\newblock {\em Commun. Math. Sci.}, 15(2):481--509, 2017.

\bibitem{HKR}
D.~Han-Kwan and F.~Rousset.
\newblock {Quasineutral limit for Vlasov-Poisson with Penrose stable data}.
\newblock {\em Ann. Sci. {\'{E}}c. Norm. Sup{\'{e}}r. (4)}, 49(6):1445--1495,
  2016.

\bibitem{Hauray-Jabin}
M.~Hauray and P.-E. Jabin.
\newblock Particle approximation of {Vlasov} equations with singular forces:
  propagation of chaos.
\newblock {\em Ann. Sci. {\'{E}}c. Norm. Sup{\'{e}}r. (4)}, 48(4):891--940,
  2015.

\bibitem{Horst}
E.~Horst.
\newblock {Global solutions of the relativistic Vlasov-Maxwell system of plasma
  physics}.
\newblock {\em Dissertationes Math. (Rozprawy Mat.)}, 292, 1990.

\bibitem{Laz}
D.~Lazarovici.
\newblock The {Vlasov-Poisson} dynamics as the mean field limit of extended
  charges.
\newblock {\em Comm. Math. Phys.}, 347(1):271--289, 2016.

\bibitem{Laz-Pickl}
D.~Lazarovici and P.~Pickl.
\newblock {A mean field limit for the Vlasov-Poisson system}.
\newblock {\em Arch. Ration. Mech. Anal.}, 225(3):1201--1231, 2017.

\bibitem{Lions-Perthame}
P.~L. Lions and B.~Perthame.
\newblock {Propagation of moments and regularity for the 3-dimensional
  Vlasov-Poisson system}.
\newblock {\em Invent. Math.}, 105(2):415--430, 1991.

\bibitem{Loep}
G.~Loeper.
\newblock {Uniqueness of the solution to the Vlasov-Poisson system with bounded
  density}.
\newblock {\em J. Math. Pures Appl. (9)}, 86(1):68--79, 2006.

\bibitem{BM}
A.~Majda and A.~Bertozzi.
\newblock {\em {Vorticity and Incompressible Flow}}, volume~27 of {\em
  Cambridge Texts in Applied Mathematics}.
\newblock Cambridge University Press, 2002.

\bibitem{Neunzert}
H.~Neunzert.
\newblock {An introduction to the nonlinear Boltzmann-Vlasov equation}.
\newblock In {\em Kinetic Theories and the Boltzmann Equation}, volume 1048 of
  {\em Lecture Notes in Math.}, pages 60--110. Springer, Berlin, Heidelberg,
  1984.

\bibitem{Neunzert-Wick}
H.~Neunzert and J.~Wick.
\newblock {Die Approximation der L{\"{o}}sung von
  Integro-Differentialgleichungen durch endliche Punktmengen}.
\newblock In {\em Numerische Behandlung nichtlinearer Integrodifferential-und
  Differentialgleichungen.}, volume 395 of {\em Lecture Notes in Math.}, pages
  275--290, Berlin, Heidelberg, 1974. Springer.

\bibitem{Pfaffelmoser}
K.~Pfaffelmoser.
\newblock {Global classical solutions of the Vlasov-Poisson system in three
  dimensions for general initial data}.
\newblock {\em J. Differential Equations}, 95(2):281--303, 1992.

\bibitem{Rein-book}
G.~Rein.
\newblock {Collisionless kinetic equations from astrophysics - the
  Vlasov-Poisson system}.
\newblock In {\em Handbook of Differential Equations: Evolutionary Equations},
  volume III, chapter~5, pages 383--476. Elsevier/North Holland, Amsterdam,
  2007.

\bibitem{Schaeffer}
J.~Schaeffer.
\newblock {Global existence of smooth solutions to the Vlasov-Poisson system in
  three dimensions}.
\newblock {\em Comm. Partial Differential Equations}, 16(8-9):1313--1335, 1991.

\bibitem{Titch}
E.~Titchmarsh.
\newblock {\em Eigenfunction Expansions Associated with Second-Order
  Differential Equations, Part II}.
\newblock Oxford University Press, 1958.

\bibitem{Ukai-Okabe}
S.~Ukai and T.~Okabe.
\newblock On classical solutions in the large in time of two-dimensional
  {Vlasov's} equation.
\newblock {\em Osaka J. Math.}, 15(2):245--261, 1978.

\bibitem{Vil03}
C.~Villani.
\newblock {\em {Topics in optimal transportation}}, volume~58 of {\em Graduate
  Studies in Mathematics}.
\newblock American Mathematical Society, Providence, RI, 2003.

\end{thebibliography}
\bibliographystyle{habbrv}

\end{document}